\documentclass{elsarticle}

\usepackage[utf8]{inputenc}
\usepackage{amsmath, amssymb, amsthm, amsfonts}
\usepackage{bm}
\usepackage{titlesec}
\usepackage{mathtools}
\usepackage{stmaryrd}
\usepackage{tikz, graphicx}
\usepackage{caption, subcaption}
\usepackage{hyperref, cleveref}
\usepackage[titletoc]{appendix}
\usepackage[normalem]{ulem}
\usepackage{geometry}
\usepackage{algorithm}
\usepackage{algpseudocode}

\usepackage{multirow}

\usetikzlibrary{positioning}

\theoremstyle{definition}

\theoremstyle{lemma}
\newtheorem{lemma}{Lemma}
\theoremstyle{theorem}
\newtheorem{theorem}{Theorem}
\theoremstyle{assumption}

\renewcommand{\hat}[1]{\widehat{#1}} 
\newcommand{\diff}[1]{\mathop{}\!{\mathrm{d}#1}} 
 
\newcommand{\pd}[2]{\frac{\partial#1}{\partial#2}}

\newcommand{\nor}[1]{\left\| #1 \right\|} 
\newcommand{\LRp}[1]{\left( #1 \right)} 
\newcommand{\LRs}[1]{\left[ #1 \right]} 
\newcommand{\LRa}[1]{\left\langle #1 \right\rangle} 
\newcommand{\LRb}[1]{\left| #1 \right|} 
\newcommand{\LRc}[1]{\left\{ #1 \right\}} 
\newcommand{\LRl}[1]{\left. #1 \right|} 
\newcommand{\jump}[1] {\ensuremath{\left\llbracket#1\right\rrbracket}} 
\newcommand{\avg}[1] {\ensuremath{\LRc{\!\!\LRc{#1}\!\!}}}

\newcommand{\note}[1]{#1}
\newcommand{\bnote}[1]{#1}
\newcommand{\rnote}[1]{#1}
\newcommand{\vnote}[1]{#1}
\newcommand{\gnote}[1]{#1}

\newtheorem{remark}{Remark}

\date{}

\begin{document}


\begin{frontmatter}
\title{An artificial viscosity approach to high order \\ entropy stable discontinuous Galerkin methods}
\author[rice]{Jesse Chan}
\ead{jesse.chan@rice.edu}
\address[rice]{Department of Computational Applied Mathematics and Operations Research\\Rice University, 6100 Main St, Houston, TX, 77005}
\begin{abstract}
Entropy stable discontinuous Galerkin (DG) methods improve the robustness of high order DG simulations of nonlinear conservation laws. These methods yield a semi-discrete entropy inequality, and rely on an algebraic flux differencing formulation which involves both summation-by-parts (SBP) discretization matrices and entropy conservative two-point finite volume fluxes. However, explicit expressions for such two-point finite volume fluxes may not be available for all systems, or may be computationally expensive to compute. 

This paper proposes an alternative approach to constructing entropy stable DG methods using an \gnote{entropy correction artificial viscosity, where the artificial viscosity coefficient is determined} based on the local violation of a cell entropy inequality and the local entropy dissipation. The resulting method \bnote{is a modification of the entropy correction introduced by Abgrall, \"{O}ffner, and Ranocha in \cite{abgrall2022reinterpretation}, and} recovers the same global semi-discrete entropy inequality that is satisfied by entropy stable flux differencing DG methods. The \gnote{entropy correction artificial viscosity} coefficients are parameter-free and locally computable over each cell, and the resulting artificial viscosity preserves both high order accuracy and a hyperbolic maximum stable time-step size under explicit time-stepping. 
\end{abstract}
\end{frontmatter}


\section{Introduction}

Entropy stable discontinuous Galerkin (DG) methods significantly improve the robustness of high order DG methods for the time-dependent nonlinear conservation laws which govern high speed and compressible fluid flows \cite{chen2020review, gassner2021novel}. These methods yield a semi-discrete entropy inequality, and rely on an algebraic ``flux differencing'' formulation which involves summation-by-parts (SBP) operators and entropy conservative two-point finite volume fluxes \cite{tadmor1987numerical}. However, flux differencing entropy stable DG methods face challenges related to their reliance on entropy conservative flux formulas. For example, explicit expressions for such two-point finite volume fluxes may not be available for all systems, or may be computationally expensive to compute \cite{ranocha2023efficient, montoya2024efficient}. Moreover, the use of such fluxes can lead to local linear stability issues \cite{gassner2022stability, ranocha2021preventing}. 

In this paper, we propose a construction of entropy stable methods using \bnote{an ``entropy correction''} artificial viscosity. This approach is motivated by \gnote{local correction terms introduced in \cite{abgrall2018general, abgrall2022reinterpretation},} as well as by observations in \cite{lin2023high} where the authors enforced a cell entropy inequality on nodal DG methods using subcell algebraic flux correction techniques. This cell entropy inequality is \gnote{related} to the cell entropy inequality satisfied by flux differencing entropy stable DG methods, and was observed to preserve high order accuracy and deliver results which were qualitatively similar to flux differencing entropy stable schemes. 

The idea of enforcing entropy stability using a vanishing viscosity is among the oldest and most popular ideas in computational fluid dynamics (CFD) \cite{vonneumann1950method, majda1979numerical, tadmor1987numerical, guermond2011entropy, lv2016entropy, glaubitz2018artificial, berthon2023artificial}, making a comprehensive literature review difficult. However, what distinguishes the artificial viscosity in this paper from most traditional artificial viscosity approaches in that it is intended only to restore an entropy inequality, and not to perform shock capturing (e.g., suppressing spurious oscillations in the solution due to sharp gradients or under-resolved solution features). For example, for a linear constant coefficient PDE where a standard DG formulation is already entropy/energy stable, the artificial viscosity in this paper vanishes. 

The proposed artificial is \bnote{most} closely related to the ``entropy correction'' terms proposed in \cite{abgrall2018general, abgrall2022reinterpretation} and extended in \cite{edoh2024conservative, gaburro2023high, mantri2024fully} and other papers. While these local correction terms are simpler to implement and also restore a cell entropy inequality, the \bnote{entropy correction} artificial viscosity proposed in this paper using consistently discretized diffusion operators appear to be more robust, especially for higher orders of approximation. Moreover, we prove that the entropy violation estimate used in this work is smaller than the estimates of entropy violation used in \cite{abgrall2022reinterpretation, gaburro2023high, mantri2024fully}, resulting in less artificial viscosity being applied. 


The paper proceeds as follows: Section~\ref{sec:nonlinear_cons_laws} reviews nonlinear conservation laws and different forms of an entropy inequality, and Section~\ref{sec:dg} reviews a standard DG weak formulation for a nonlinear conservation law. Section~\ref{sec:av} describes the viscous discretization and proposes an \gnote{entropy correction artificial viscosity} which enforces a cell entropy inequality. \bnote{We compare the proposed method against local entropy correction terms in the literature \cite{abgrall2018general, abgrall2022reinterpretation} (as well as against some existing artificial viscosity methods) in Section~\ref{sec:entropy_correction}.} We conclude with numerical experiments analyzing the robustness, high order accuracy, and \gnote{local} linear stability of \gnote{entropy correction artificial viscosity} in Section~\ref{sec:numerical}. Finally, additional numerical experiments and the derivation of a subcell version of the \gnote{entropy correction artificial viscosity} are included in \ref{sec:additional_1d}, \ref{sec:subcell}.

\section{Nonlinear conservation laws and entropy inequalities}
\label{sec:nonlinear_cons_laws}
The focus of this paper is on the numerical approximation of solutions to systems of nonlinear conservation laws in $d$ dimensions
\begin{equation}
\pd{\bm{u}}{t} + \sum_{m=1}^d \pd{\bm{f}_m(\bm{u})}{x_m} = 0.
\label{eq:ncl}
\end{equation}
where $\bm{u}(\bm{x}, t) \in \mathbb{R}^n$ and $\bm{f}_m : \mathbb{R}^n \rightarrow \mathbb{R}^n$. We assume that \eqref{eq:ncl} admits one or more entropy inequalities of the form 
\begin{equation}
\pd{S(\bm{u})}{t} + \sum_{m=1}^d \pd{F_m(\bm{u})}{x_m} \leq 0, \qquad F_m(\bm{u}) = \bm{v}(\bm{u})^T\bm{f}_m(\bm{u}) - \psi_m(\bm{u}),
\label{eq:entropy_ineq}
\end{equation}
where $S(\bm{u})$ is a scalar convex entropy, $F_m(\bm{u})$ are the associated entropy fluxes, $\psi_m$ are the entropy potentials, and $\bm{v}(\bm{u})$ are the entropy variables 
\[
\bm{v}(\bm{u}) = \pd{S}{\bm{u}}.
\]
Note that the convexity of $S(\bm{u})$ guarantees that the mapping between conservative and entropy variables is invertible.  

For sufficiently regular solutions, one can derive the equality version of \eqref{eq:entropy_ineq} by multiplying \eqref{eq:ncl} by the entropy variables $\bm{v}(\bm{u})$. Applying the chain rule and properties of the entropy fluxes yields the corresponding entropy conservation condition. A vanishing viscosity argument yields the inequality version of \eqref{eq:entropy_ineq} for more general classes of solutions \cite{dafermos2005hyperbolic, godlewski2013numerical, chen2017entropy}. This work utilizes an integrated cell version of \eqref{eq:entropy_ineq} \cite{jiang1994cell}. Consider a closed domain $D \subset R^d$ with boundary $\partial D$. Then, integrating \eqref{eq:entropy_ineq} over $D$ and applying the divergence theorem yields
\begin{equation}
\int_{D}\pd{S(\bm{u})}{t} + \int_{\partial D} \sum_{m=1}^d\LRp{\bm{v}^T\bm{f}_m(\bm{u}) - \psi_m(\bm{u})} n_m\leq 0.
\label{eq:cell_entropy_ineq}
\end{equation}

\subsection{An intermediate entropy identity}

We will enforce the cell entropy inequality \eqref{eq:cell_entropy_ineq} by enforcing an intermediate identity. The derivation of the cell entropy inequality \ref{eq:cell_entropy_ineq} starts by testing with the entropy variables $\bm{v}(\bm{u})$ and integrating over some domain $D$. Doing so and integrating the spatial derivative term by parts yields
\begin{equation}
\int_{D} \pd{S(\bm{u})}{t} + \sum_{m=1}^d \LRs{\int_{D} -\pd{\bm{v}(\bm{u})}{x_m} \bm{f}_m(\bm{u}) + \int_{\partial D} \bm{v}(\bm{u})^T\bm{f}_m(\bm{u}) n_m} = 0.
\label{eq:entropy_identity_step}
\end{equation}
Subtracting \eqref{eq:cell_entropy_ineq} from \eqref{eq:entropy_identity_step} then yields
\begin{equation}
\rnote{\sum_{m=1}^d \LRp{\int_{D}-\pd{\bm{v}(\bm{u})}{x_m}^T\bm{f}_m(\bm{u}) + \int_{\partial D} \psi_m(\bm{u}) n_m \geq 0}}.
\label{eq:cell_entropy_identity}
\end{equation}
We will add artificial viscosity such that we satisfy a discrete version of \vnote{identity \eqref{eq:cell_entropy_identity}}.

\vnote{
We note that if $\bm{u}$ and $\bm{f}_m(\bm{u})$ are differentiable, \eqref{eq:cell_entropy_identity} should be an \textit{equality} 
\begin{equation}
\rnote{\sum_{m=1}^d\LRp{\int_{D}-\pd{\bm{v}(\bm{u})}{x_m}^T\bm{f}_m(\bm{u}) + \int_{\partial D} \psi_m(\bm{u}) n_m} = 0.}
\label{eq:cell_entropy_equality}
\end{equation}
This can be derived from properties of the entropy variables, flux, and entropy flux. For $\bm{u}$ and $\bm{f}_m(\bm{u})$ differentiable, we have the following:
\begin{equation}
\sum_{m=1}^d \bm{v}(\bm{u})^T \pd{\bm{f}_m(\bm{u})}{x_m} = \sum_{m=1}^d\pd{S(\bm{u})}{\bm{u}}^T \pd{\bm{f}_m}{\bm{u}} \pd{\bm{u}}{x_m} = \pd{F_m}{\bm{u}} \pd{\bm{u}}{x_m} = \sum_{m=1}^d \pd{F_m(\bm{u})}{x_m},
\label{eq:entropy_chain_rule}
\end{equation}
where we have used the chain rule and that $ \pd{S(\bm{u})}{\bm{u}}^T \pd{\bm{f}_m}{\bm{u}} = \pd{F_m}{\bm{u}}$ \cite{tadmor2003entropy, godlewski2013numerical}. Integrating the left hand side of \eqref{eq:entropy_chain_rule} by parts over some domain $D$ yields
\[
\int_{D} \sum_{m=1}^d\bm{v}(\bm{u})^T \pd{\bm{f}_m(\bm{u})}{x_m} = \sum_{m=1}^d \LRs{\int_{\partial D}\bm{v}(\bm{u})^T\bm{f}_m(\bm{u})n_m - \int_D \pd{\bm{v}(\bm{u})}{x_m}^T\bm{f}_m(\bm{u})}. 
\]
Integrating the right hand side of \eqref{eq:entropy_chain_rule} and using the definition of $F_m(\bm{u})$ in \eqref{eq:entropy_ineq} yields that
\[
\int_{D} \sum_{m=1}^d \pd{F_m(\bm{u})}{x_m} = \sum_{m=1}^d \int_{\partial D} \LRp{\bm{v}(\bm{u})^T\bm{f}_m(\bm{u}) - \psi_m(\bm{u})}n_m.
\]
Equating these two integrals yields \eqref{eq:cell_entropy_equality}.
}

\section{High order DG formulation}
\label{sec:dg}

We assume that the domain $\Omega \subset \mathbb{R}^d$ is triangulated by non-overlapping simplicial elements $D^k$, where each element $D^k$ is the image of a reference simplex $\hat{D}$ under some affine mapping $\bm{\phi}^k: \widehat{D} \rightarrow D^k$. Let $\bm{n}$ denote the outward normal vector $\bm{n} = [n_1, \ldots, n_d]$ on each face of $D^k$. Finally, let $(u,v)_{D^k}, \LRa{u,v}_{\partial D^k}$ denote the $L^2$ inner products on $D^k$ and the surface $\partial D^k$
\[
(u,v)_{D^k} \approx \int_{D^k} u(\bm{x})v(\bm{x}) \diff{x}, \qquad \LRa{u,v}_{\partial D^k} \approx \int_{\partial D^k} u(\bm{x})v(\bm{x}) \diff{x},
\]
where the approximate equality is due to the assumption that both volume and surface integrals in inner products are computed inexactly (e.g., approximated using quadrature). 

We consider a standard DG formulation for an approximate solution $\bm{u}_h$ on a single element $D^k$ for \eqref{eq:ncl} \cite{karniadakis2005spectral, hesthaven2007nodal}:
\begin{equation}
\LRp{\pd{\bm{u}_h}{t}, \bm{w}}_{D^k} + \sum_{m=1}^d \LRp{-\bm{f}_m(\bm{u}_h), \pd{\bm{w}}{x_m}}_{D^k} + \LRa{\bm{f}_n^*, \bm{w}}_{\partial D^k} = \bm{0}, \qquad \bm{w} \in \LRs{P^N(D^k)}^n,
\label{eq:dg_form}
\end{equation}
where $P^N(D^k)$ denotes the space of total degree $N$ polynomials on $D^k$. 
This formulation is derived by multiplying \eqref{eq:ncl} by a test function $\bm{w} \in \LRs{P^N(D^k)}^n$, integrating by parts, and introducing a numerical flux $\bm{f}^*_n$ across each inter-element interface. All volume and surface integrals are approximated using some quadrature rule, and summing up over all elements $D^k$ yields a global DG formulation. 

\subsection{Semi-discrete entropy estimate}

First, observe that the entropy variables $\bm{v}(\bm{u}_h)$ can be non-polynomial and do not (in general) lie in the test space $\LRs{P^N(D^k)}^n$. Thus, to derive a semi-discrete entropy estimate for \eqref{eq:dg_form}, we must instead test with the $L^2$ projection of the entropy variables $\Pi_N \bm{v}(\bm{u}_h)$ \cite{chan2018discretely}. Here, $\Pi_N$ denotes the $L^2$ projection operator onto $P^N(D^k)$ such that for $f\in L^2(D^k)$
\[
\LRp{\Pi_N f, v}_{D^k} = \LRp{f, v}_{D^k}, \qquad \forall v \in P^N(D^k).
\]
We note that similar observations on the non-polynomial nature of the entropy variables have been made by a variety of different groups in the literature \cite{williams2019analysis, gkanis2021new, colombo2022entropy, chan2022entropyprojection, andrews2024high}.

Since $\pd{\bm{u}_h}{t} \in \LRs{P^N\LRp{D^k}}^n$ for method of lines discretizations, testing the time derivative term in \eqref{eq:dg_form} with the projected entropy variables yields
\begin{equation}
\LRp{\pd{\bm{u}_h}{t}, \Pi_N\bm{v}(\bm{u}_h)}_{D^k} = \LRp{\pd{\bm{u}_h}{t}, \bm{v}(\bm{u}_h)}_{D^k} = \LRp{\pd{S(\bm{u}_h)}{t}, 1}_{D^k},
\label{eq:dSdt}
\end{equation}
where we have used the chain rule in time for the final step. Note that this equality still holds under inexact quadrature, so long as the $L^2$ projection operator $\Pi_N$ is defined using the same inexact quadrature rule. We also note that these arguments are equivalent to the matrix-based arguments made in \cite{chan2018discretely}. 

Using \eqref{eq:dSdt}, we recover that the semi-discrete local rate of change of entropy over $D^k$ is given by
\begin{equation}
\LRp{\pd{S(\bm{u}_h)}{t}, 1}_{D^k} + \sum_{m=1}^d \LRp{-\bm{f}_m(\bm{u}_h), \pd{\Pi_N\bm{v}(\bm{u}_h)}{x_m}}_{D^k} + \LRa{\bm{f}_n^*, \Pi_N\bm{v}(\bm{u}_h)}_{\partial D^k} = 0.
\label{eq:local_dSdt}
\end{equation}
Unfortunately, \eqref{eq:local_dSdt} does not in general yield a semi-discrete entropy equality or inequality. At this point in the continuous derivation of either an entropy equality or inequality, the chain rule is typically used to manipulate the volume integral. However, due to the presence of the $L^2$ projection of the entropy variables $\Pi_N\bm{v}(\bm{u})$, we are unable to mimic this step.  Moreover, because the integrals and inner products in this statement are typically discretized using inexact quadrature, the chain rule does not necessarily hold in spatially discrete settings. 

High order entropy stable DG methods based on ``flux differencing'' \cite{carpenter2014entropy, gassner2016split, chen2017entropy, crean2018entropy, chan2018discretely} circumvent the loss of the chain rule by modifying the volume term in the DG formulation to incorporate non-dissipative numerical fluxes which satisfy an entropy conservation condition  \cite{tadmor1987numerical}. When combined with summation-by-parts (SBP) discretizations, this yields DG formulations which are semi-discretely entropy conservative or entropy stable in the sense that they satisfy a discrete version of the cell entropy identity \eqref{eq:cell_entropy_identity}, and thus a quadrature version of the cell entropy inequality \eqref{eq:cell_entropy_ineq}
\[
\LRp{\pd{S(\bm{u}_h)}{t}, 1}_{D^k} + \LRa{\LRp{ \Pi_N\bm{v}(\bm{u}_h)}^T\bm{f}_n^* - \sum_{m=1}^d\psi_m(\tilde{\bm{u}}) n_m, 1}_{\partial D^k} = 0.
\]
where we have introduced the entropy projection $\tilde{\bm{u}}$:
\begin{equation}
\tilde{\bm{u}} = \bm{u}\LRp{\Pi_N\bm{v}(\bm{u}_h)}, 
\label{eq:entropy_proj}
\end{equation}
e.g., the conservative variables evaluated in terms of the $L^2$ projection of the entropy variables.

Recent work has demonstrated that it is possible to enforce a cell entropy inequality outside of the ``flux differencing'' framework \cite{lin2023high, vilar2024local}. In \cite{lin2023high}, a cell entropy inequality is enforced by solving a local knapsack problem to compute the optimal subcell blending of a standard high order nodal DG method with a compatible entropy stable low order scheme. In this paper, we add a minimal \gnote{entropy correction artificial viscosity} to a standard high order DG method which is sufficient to recover an entropy inequality. The artificial viscosity coefficient is locally computable and relies on standard stability estimates for DG discretizations of viscous terms \cite{gassner2018br1, chan2022entropy}.

\section{Recovering an entropy inequality via \bnote{an entropy correction} artificial viscosity}
\label{sec:av}
We consider the following ``monolithic'' viscous regularization of \eqref{eq:ncl}
\[
\pd{\bm{u}}{t} + \sum_{m=1}^d \pd{\bm{f}_m(\bm{u})}{x_m} = \sum_{i,j=1}^d \pd{}{x_i}\LRp{\epsilon_k(\bm{u})\pd{\bm{u}}{x_j}}.
\]
where $\epsilon_k(\bm{u}) \geq 0$. \vnote{Here, the term ``monolithic'' is taken from \cite{guermond2014viscous}.}
It is well known that it is possible to symmetrize many different viscous term by transforming to entropy variables \cite{harten1983symmetric, hughes1986new}. For the monolithic regularization, this can simply be done by applying the chain rule:
\begin{equation}
\pd{\bm{u}}{t} + \sum_{m=1}^d \pd{\bm{f}_m(\bm{u})}{x_m} = \sum_{i=1}^d \pd{}{x_i}\LRp{\epsilon_k(\bm{u})\pd{\bm{u}}{\bm{v}}\pd{\bm{v}}{x_i}}.
\label{eq:simple_av}
\end{equation}
By the convexity of $S(\bm{u})$, the Jacobian matrix $\pd{\bm{u}}{\bm{v}}$ is symmetric and positive definite. We will analyze the slightly more general viscous regularization
\begin{equation}
\pd{\bm{u}}{t} + \sum_{m=1}^d \pd{\bm{f}_m(\bm{u})}{x_m} = \sum_{i,j=1}^d \pd{}{x_i}\LRp{\epsilon_k(\bm{u})\bm{K}_{ij}\pd{\bm{v}}{x_j}}.
\label{eq:av}
\end{equation}
where $\bm{K}_{ij}$ denote blocks of a symmetric and positive semi-definite matrix $\bm{K}$
\begin{equation}
\bm{K} = \begin{bmatrix}
\bm{K}_{11} & \ldots & \bm{K}_{1d}\\
\vdots & \ddots & \vdots\\
\bm{K}_{d1} & \ldots & \bm{K}_{dd}
\end{bmatrix} = \bm{K}^T, \qquad \bm{K} \succeq 0.
\label{eq:K}
\end{equation}
For example, taking $\bm{K}_{ij} = \delta_{ij} \pd{\bm{u}}{\bm{v}}$ recovers \eqref{eq:simple_av}, \vnote{and we will utilize $\bm{K}_{ij} = \delta_{ij} \pd{\bm{u}}{\bm{v}}$ for the remainder of this paper.} Other viscous regularizations of \eqref{eq:ncl}  \cite{hughes1986new, guermond2014viscous, svard2024refining} can also be accommodated in this framework, but will not be the focus of this work. 

\subsection{DG formulation of artificial viscosity}

We discretize \eqref{eq:av} by adding a viscous contribution $\bm{g}_{\rm visc}$ to the DG formulation \eqref{eq:dg_form}:
\begin{equation}
\LRp{\pd{\bm{u}_h}{t}, \bm{w}}_{D^k} + \sum_{m=1}^d \LRp{-\bm{f}_m(\bm{u}_h), \pd{\bm{w}}{x_m}}_{D^k} + \LRa{\bm{f}_n^*, \bm{w}}_{\partial D^k} = \LRp{\bm{g}_{\rm visc}, \bm{w}}_{D^k}, \quad \bm{w} \in \LRs{P^N(D^k)}^n,
\label{eq:visc_dg}
\end{equation}
where the viscous terms $\bm{g}_{\rm visc}$ are discretizations of \eqref{eq:av} using a BR-1 type discretization \cite{chan2022entropy}. To simplify the notation, denote the $L^2$ projection of the entropy variables $\bm{v}_h = \Pi_N\bm{v}(\bm{u}_h)$. Then, $\bm{g}_{\rm visc}$ is given by the following formulation (where $i = 1,\ldots,d$): 
\begin{gather}
\LRp{\bm{\Theta}_i, \bm{w}_{1,i}}_{D^k} = \LRp{ \pd{\bm{v}_h}{x_i},\bm{w}_{1,i}}_{D^k} + \frac{1}{2}\LRa{\jump{\bm{v}_h}n_{i},\bm{w}_{1,i}}_{\partial D^k}, \quad \forall \bm{w}_{1,i}\in \LRs{P^N\LRp{D^k}}^n \label{eq:ldg1},\\
\LRp{\bm{\sigma}_i,\bm{w}_{2,i}}_{D^k} = \LRp{\sum_{j=1}^d\epsilon_k(\bm{u}_h)\bm{K}_{ij}\bm{\Theta}_j, \bm{w}_{2,i}}_{D^k}, \quad \forall \bm{w}_{2,i}\in \LRs{P^N\LRp{D^k}}^n  \label{eq:ldg2},\\
\LRp{\bm{g}_{\rm visc}, \bm{w}_{3}}_{D^k} 
= \sum_{i=1}^d \LRs{\LRp{ -\bm{\sigma}_i, \pd{\bm{w}_{3}}{x_i}}_{D^k} + \LRa{\avg{\bm{\sigma}_i}n_{i},\bm{w}_{3}}_{\partial D^k}},\quad \forall \bm{w}_{3}\in \LRs{P^N\LRp{D^k}}^n.   \label{eq:ldg3}
\end{gather}
Here, $\jump{\cdot}$ and $\avg{\cdot}$ denote the jump and average operations:
\[
\jump{u} = u^+ - u^-, \qquad  \avg{u} = \frac{1}{2}\LRp{u^+ + u^-},
\]
where $u^-$ denotes the interior value of $u$ on a face of $D^k$, and $u^+$ denotes the exterior (neighboring) value of $u$ across the same face. Here, $\bm{\Theta}_i$ are DG approximations of derivatives of the entropy variables $\bm{v}(\bm{u}_h)$ with respect to the $i$th coordinate. In \eqref{eq:ldg2}, we compute the viscous fluxes $\bm{\sigma}_i$ as the $L^2$ projection of $\sum_{j=1}^d \bm{K}_{ij}\bm{\Theta}_j$ for $i = 1,\ldots, d$ onto the approximation space of each element. The viscous terms  $\bm{g}_{\rm visc}$ are the result of computing the divergence of the viscous fluxes in \eqref{eq:ldg3}. 

It can be shown using a straightforward modification of the proofs in \cite{chan2022entropy} that the following stability estimate holds:
\begin{lemma}
\label{lemma:gvisc} (Adapted from Lemma 3.1 in \cite{chan2022entropy}).
Let $\bm{g}_{\rm visc}$ be given by \eqref{eq:ldg1}, \eqref{eq:ldg2}, \eqref{eq:ldg3}. Then, for a periodic domain, 
\[
\sum_k -\LRp{\bm{g}_{\rm visc}, \Pi_N\bm{v}(\bm{u}_h)}_{D^k} = \sum_k \sum_{i,j=1}^d \LRp{\epsilon_k(\bm{u}_h)\bm{K}_{ij}\bm{\Theta}_j, \bm{\Theta}_i}_{D^k} \geq 0.
\]
where $\Pi_N\bm{v}(\bm{u}_h)$ denotes the $L^2$ projection of the entropy variables.
\end{lemma}
Lemma~\ref{lemma:gvisc} also holds under other appropriate impositions of boundary conditions, \rnote{for example, the entropy stable impositions of no-slip wall boundary conditions in \cite{sayyari2021development} or far-field boundary conditions from \cite{svard2024entropy}. These will be treated in future work}. We also note that Lemma~\ref{lemma:gvisc} is agnostic to the regularity of $\epsilon_k(\bm{u}_h)$ and $\bm{K}_{ij}$; in particular, they can be discontinuous from element to element. Because the estimate in \note{Lemma}~\ref{lemma:gvisc} is localizable to each element, it is possible to determine exactly the amount of artificial viscosity necessary to satisfy a discrete version of the cell entropy identity \eqref{eq:cell_entropy_identity}
\begin{equation}
\sum_{m=1}^d \LRs{\LRp{-\bm{f}_m(\bm{u}_h), \pd{\Pi_N\bm{v}(\bm{u}_h)}{x_m}}_{D^k} + \LRa{\psi_m(\tilde{\bm{u}}) n_m, 1}_{\partial D^k}} \geq 0.
\label{eq:discrete_cell_entropy_identity}
\end{equation}
where $\tilde{\bm{u}}$ is the entropy projection \eqref{eq:entropy_proj}.

We first introduce the volume entropy residual $\delta_k(\bm{u}_h)$:
\begin{equation}
\delta_k(\bm{u}_h) = \sum_{m=1}^d \LRs{ \LRp{-\bm{f}_m(\bm{u}_h), \pd{\Pi_N\bm{v}(\bm{u}_h)}{x_m}}_{D^k} + \LRa{\psi_m(\tilde{\bm{u}})n_m, 1}_{\partial D^k}}.
\label{eq:entropy_ineq_error}
\end{equation}
Our goal will be to determine an artificial viscosity proportional to the violation of \eqref{eq:discrete_cell_entropy_identity}:
\begin{lemma}
\label{thm:cell_entropy_ineq}
Let $\epsilon_k(\bm{u}_h)$ on $D^k$ satisfy
\begin{equation}
\sum_{i,j=1}^d \LRp{\epsilon_k(\bm{u}_h)\bm{K}_{ij}\bm{\Theta}_j, \bm{\Theta}_i}_{D^k} \geq -\min(0, \delta_k(\bm{u}_h)),
\label{eq:cell_entropy_ineq_condition}
\end{equation}
and let $\tilde{\bm{u}}$ denote the entropy projection \eqref{eq:entropy_proj}. Then, \eqref{eq:visc_dg} satisfies the following global entropy inequality:
\begin{equation}
\sum_k \LRs{\LRp{\pd{S(\bm{u}_h)}{t}, 1}_{D^k} + \LRa{\LRp{ \Pi_N\bm{v}(\bm{u}_h)}^T\bm{f}_n^* - \sum_{m=1}^d\psi_m(\tilde{\bm{u}}) n_m, 1}_{\partial D^k}} \leq 0.
\label{eq:global_entropy_ineq}
\end{equation}
\end{lemma}
\begin{proof}
The inclusion of the viscous terms in \eqref{eq:visc_dg} yields a slightly modified version of the semi-discrete entropy balance \eqref{eq:local_dSdt}
\[
\LRp{\pd{S(\bm{u}_h)}{t}, 1}_{D^k} + \sum_{m=1}^d \LRp{-\bm{f}_m(\bm{u}_h), \pd{\Pi_N\bm{v}(\bm{u}_h)}{x_m}}_{D^k} + \LRa{\bm{f}_n^*, \Pi_N\bm{v}(\bm{u}_h)}_{\partial D^k}  = \LRp{\bm{g}_{\rm visc}, \Pi_N\bm{v}(\bm{u}_h)}_{D^k}.
\]
\bnote{To derive \eqref{eq:global_entropy_ineq}, we add and subtract $\LRa{\sum_{m=1}^d\psi_m(\tilde{\bm{u}}) n_m, 1}_{\partial D^k}$ on each element. Then, summing over all elements $D^k$ and using \eqref{lemma:gvisc} yields 
\begin{align*}
\sum_k\LRp{\pd{S(\bm{u}_h)}{t}, 1}_{D^k} &+ \delta_k(\bm{u}_h) +  \sum_{i,j=1}^d \LRp{\epsilon_k(\bm{u}_h)\bm{K}_{ij}\bm{\Theta}_j, \bm{\Theta}_i}_{D^k} \\
&+ \LRa{\LRp{ \Pi_N\bm{v}(\bm{u}_h)}^T\bm{f}_n^* - \sum_{m=1}^d\psi_m(\tilde{\bm{u}}) n_m, 1}_{\partial D^k} = 0.
\end{align*}
The proof is completed by noting that, by \eqref{eq:cell_entropy_ineq_condition}, 
\[
\delta_k(\bm{u}_h) +  \sum_{i,j=1}^d \LRp{\epsilon_k(\bm{u}_h)\bm{K}_{ij}\bm{\Theta}_j, \bm{\Theta}_i}_{D^k} \geq \min(0, \delta_k(\bm{u}_h)) +  \sum_{i,j=1}^d \LRp{\epsilon_k(\bm{u}_h)\bm{K}_{ij}\bm{\Theta}_j, \bm{\Theta}_i}_{D^k} \geq 0.
\]
}
\end{proof}
\vnote{
We note that, if instead of \eqref{eq:cell_entropy_ineq_condition}, we enforced
\begin{equation}
\sum_{i,j=1}^d \LRp{\epsilon_k(\bm{u}_h)\bm{K}_{ij}\bm{\Theta}_j, \bm{\Theta}_i}_{D^k} = -\delta_k(\bm{u}_h), 
\end{equation}
then the resulting formulation is entropy conservative in the sense that \eqref{eq:global_entropy_ineq} would become an equality instead. However, because this results in negative viscosity coefficients $\epsilon_k(\bm{u}_h)$ and anti-diffusive contributions, we do not consider it in this paper. Moreover, numerical results using this ``entropy conservative'' approach result in more spurious oscillations than if \eqref{eq:cell_entropy_ineq_condition} is used.
}
\begin{remark}
We note that it is possible to include additional semi-definite penalty terms involving the jumps of entropy variables in \eqref{eq:ldg3} \cite{chan2022entropy}. We do not include them here for simplicity of presentation; however, because they add additional entropy dissipation to the sufficient condition \eqref{eq:cell_entropy_ineq_condition}, they can be incorporated without changing the main results of this paper. Moreover, the viscous discretization is not restricted to a BR-1 discretization; any viscous discretization which yields a localizable energy estimate results in the same semi-discrete entropy stability estimate in Lemma~\ref{thm:cell_entropy_ineq}.
\end{remark}

One can also derive a local entropy balance as is done in \cite{shu2009general}, though the cell entropy inequality will contain additional consistent terms resulting from the viscous discretization. Furthermore, if the numerical flux $\bm{f}^*_n$ is entropy stable and evaluated using the entropy projection $\tilde{\bm{u}}$, then we have a global entropy inequality:
\begin{lemma}
\label{lemma:global_entropy_inequality}
Let the interface flux $\bm{f}^*_n = \bm{f}^*_n(\bm{u}_L, \bm{u}_R)$ be skew symmetric and entropy stable such that 
\begin{gather*}
\LRp{\bm{v}_L - \bm{v}_R}^T\bm{f}^*_n(\bm{u}_L, \bm{u}_R) \leq \sum_{m=1}^d \LRp{\psi_m(\bm{u}_L) - \psi_m(\bm{u}_R)}n_m\\
\bm{f}^*_n(\bm{u}_L, \bm{u}_R) = - \bm{f}^*_{-n}(\bm{u}_R, \bm{u}_L).
\end{gather*}
where $\bm{v}_L = \bm{v}(\bm{u}_L), \bm{v}_R = \bm{v}(\bm{u}_R)$ are the entropy variables corresponding to the left and right solution states $\bm{u}_L, \bm{u}_R$. Then, if the interface flux is evaluated in terms of the entropy projection $\bm{f}^*_n = \bm{f}^*_n(\tilde{\bm{u}}^+, \tilde{\bm{u}})$, \eqref{eq:global_entropy_ineq} implies the global entropy inequality
\[
{\LRp{\pd{S(\bm{u}_h)}{t}, 1}_{\Omega} + \LRa{\LRp{ \Pi_N\bm{v}(\bm{u}_h)}^T\bm{f}_n^* - \sum_{m=1}^d\psi_m(\tilde{\bm{u}}) n_m, 1}_{\partial \Omega}} \leq 0.
\]
\end{lemma}
\begin{proof}
The proof is identical to the proof of the global entropy inequality in \cite{chen2017entropy, chan2018discretely}. We repeat it for completeness here. Since the sum over the interface terms in \eqref{eq:global_entropy_ineq} includes all elements, each face $f$ between element $D^k$ and its neighbor $D^{k,+}$ yields two contributions:
\begin{align*}
&\int_{f \cap D^k} \LRp{\Pi_N\bm{v}(\bm{u}_h)}\cdot\bm{f}_n^*(\tilde{\bm{u}}^+, \tilde{\bm{u}}) - \sum_{m=1}^d\psi_m(\tilde{\bm{u}}) n_m\\
&+ \int_{f \cap D^{k,+}} \LRp{\Pi_N\bm{v}(\bm{u}_h)}^+\cdot\bm{f}_{n^+}^*(\tilde{\bm{u}}, \tilde{\bm{u}}^+) - \sum_{m=1}^d\psi_m(\tilde{\bm{u}}^+) n^+_m
\end{align*}
The global entropy inequality follows if the sums of these inter-element contributions are non-negative.
Noting that $\bm{n}^+ = -\bm{n}$ and using the skew-symmetry property, we can combine like terms. Then, using the entropy stability of $\bm{f}_n^*$ and that that $\tilde{\bm{u}} = \bm{u}\LRp{\Pi_N\bm{v}(\bm{u}_h)}$, we conclude that
\[
\int_{f}-\LRp{\LRp{\Pi_N\bm{v}(\bm{u}_h)}^+-\Pi_N\bm{v}(\bm{u}_h) }\cdot\bm{f}_n^*(\tilde{\bm{u}}^+, \tilde{\bm{u}})  + \sum_{m=1}^d \LRp{\psi_m(\tilde{\bm{u}}^+) - \psi_m(\tilde{\bm{u}}) } n_m \geq 0.
\]
\end{proof}
Note that, for periodic domains, the boundary contribution in Lemma~\ref{lemma:global_entropy_inequality} will also vanish. 

\begin{remark}
\label{remark:dissipation}
While the global entropy inequality in Lemma~\ref{lemma:global_entropy_inequality} is exactly the same as the one satisfied by flux differencing entropy stable high order DG methods (see Theorem 3.4 of \cite{chen2017entropy} or equation (87) in \cite{chan2018discretely}), we note that flux differencing entropy stable high order DG methods are less entropy dissipative, since they satisfy an equality version of \eqref{eq:global_entropy_ineq}. 
\end{remark}


\subsection{A piecewise constant viscosity coefficient}
\label{sec:piecewise_const}

We now seek to derive an expression for $\epsilon_k(\bm{u}_h)$. The simplest approach for determining $\epsilon_k(\bm{u}_h)$ is to assume it is constant over each element; then, we can determine the smallest value of $\epsilon_k(\bm{u}_h)$ necessary to \bnote{enforce condition \eqref{eq:cell_entropy_ineq_condition} in Lemma~\ref{thm:cell_entropy_ineq}} as follows:
\begin{equation}
\epsilon_k(\bm{u}_h) \geq \frac{-\min(0, \delta_k(\bm{u}_h))}{\sum_{i,j=1}^d \LRp{\bm{K}_{ij}\bm{\Theta}_j, \bm{\Theta}_i}_{D^k}},
\label{eq:eps}
\end{equation}
where $\delta_k(\bm{u}_h)$ is the volume entropy residual given by \eqref{eq:entropy_ineq_error}. 

%

\begin{remark}
\label{remark:ratio}
The denominator in \eqref{eq:eps} can be arbitrarily close to zero if $\bm{u}_h$ is close to a constant. To avoid dividing by very small numbers, we approximate ratios of the form $\frac{a}{b}$ by
\[
\frac{a}{b} \approx \frac{ab}{\delta + b^2}
\]
where \rnote{$\delta$ is some small tolerance close to machine precision. We take $\delta = 10^{-14}$ in all numerical experiments; however, we tested $\delta = 10^{-13}$ and $\delta = 10^{-15}$ as well, and did not notice any significant differences in the results}. The use of this regularized ratio enforces that, if the solution approaches a constant and the denominator vanishes, the artificial viscosity parameter \eqref{eq:eps} approaches zero. This is consistent since a constant solution over an element automatically satisfies the discrete entropy identity \eqref{eq:discrete_cell_entropy_identity}. 
\end{remark}

\begin{remark}
\label{remark:K_avg}
The viscous matrices $\bm{K}_{ij}$ are functions of the conservative or entropy variables; however, these matrices can be evaluated at any admissible solution state while still retaining entropy stability. A computationally attractive option is to evaluate the viscous matrices using the local solution average $\bar{\bm{u}}_h$ as done in \cite{gaburro2023high}, which results in an element-wise constant $\bm{K}_{ij}$. This allows for the use of more efficient quadrature-free formulations \cite{hesthaven2007nodal, chan2017gpu} when evaluating the viscous terms. 
\end{remark}

\begin{remark}
It was observed in \cite{barter2010shock, guermond2011entropy, klockner2011viscous} that artificial viscosities with continuous coefficients resulted in higher quality solutions compared with artificial viscosities with discontinuous coefficients. While we note that the artificial viscosity formulation used in this paper is stable and accurate even for discontinuous coefficients, users may smooth the locally computed $\epsilon_k(\bm{u}_h)$ without losing entropy stability. 
For example, one could construct $C^0$ viscosity coefficient whose value at vertices is the maximum of $\epsilon_k(\bm{u}_h)$ over all adjacent elements \cite{klockner2011viscous}. Interpolating these vertex values using $P^1$ or $Q^1$ elements would produce a new piecewise continuous viscosity coefficient which is pointwise greater than $\epsilon_k(\bm{u}_h)$ and thus would retain entropy stability. However, this interpolated $C^0$ viscosity yielded more dissipative solutions without improving accuracy in numerical experiments, so we do not include it in this work.
\end{remark}

We note that it also possible to derive explicit expressions for a minimum norm artificial viscosity coefficient $\epsilon_k(\bm{u}_h)$ which varies at a subcell level within an element $D^k$. This is described in more detail along with numerical examples in \ref{sec:subcell}.

\subsection{Simplification of the entropy projection under nodal collocation}

Proving entropy stability using \gnote{entropy correction artificial viscosity} requires one non-standard ingredient in the DG discretization: the entropy projection \eqref{eq:entropy_proj}. Recall that Lemma~\ref{lemma:global_entropy_inequality} requires that surface numerical fluxes be evaluated using the entropy projection in order to guarantee that the interface contributions in \eqref{eq:global_entropy_ineq} are entropy dissipative. However, for nodal collocation DG methods, this additional entropy projection step is not necessary.

For collocation methods, quadrature points are collocated with interpolation points. It can be shown that the $L^2$ projection operator $\Pi_N = I_N$, where $I_N$ denotes the degree $N$ nodal interpolation operator \cite{chan2018discretely}. Because the mapping between conservative and entropy variables is invertible, this implies that the values of $\tilde{\bm{u}}$ at nodal points is simply
\[
\tilde{\bm{u}}(\bm{x}_i) = \bm{u}\LRp{\LRl{I_N\bm{v}(\bm{u}_h)}_{\bm{x}=\bm{x}_i}} = \bm{u}(\bm{v}(\bm{u}_h(\bm{x}_i))) = \bm{u}_h(\bm{x}_i).
\]
such that the entropy projection at quadrature/interpolation points reduces to the evaluation of the local polynomial solution $\bm{u}_h$. If the surface quadrature points are a subset of the collocated volume quadrature points (as is the case for DGSEM or DG spectral element methods \cite{kopriva2010quadrature} \rnote{which utilize collocated Gauss-Lobatto-Legendre quadrature for both volume and surface integrals}), then the evaluation of the entropy projection at surface quadrature points reduces to the evaluation of the solution $\bm{u}_h$ at those points as well. A similar structure holds under diagonal-norm nodal SBP discretizations on simplicial elements \cite{chen2017entropy, crean2018entropy}.

\subsection{Estimates for the volume entropy residual}

We wish to estimate the magnitude of the volume entropy residual $\delta_k(\bm{u}_h)$, which turns out to enjoy a super-convergence property. We first assume exact integration for the proofs in this section, and discuss the effects of inexact quadrature in Section~\ref{sec:quadrature}. \vnote{Before proving this estimate, we recall the following theorem from \cite{godlewski2013numerical}:
\begin{theorem}{(Theorem 3.1 in \cite{godlewski2013numerical})}
\label{thm:roe_linearization}
Assume that the hyperbolic system \eqref{eq:ncl} has a strictly convex entropy. Then there exists at least one Roe-type linearization matrix $\bar{\bm{A}}_m = \bar{\bm{A}}_m(\bm{u}_1, \bm{u}_2)$ such that 
\[
\bm{f}_m(\bm{u}_1) - \bm{f}_m(\bm{u}_2) = \bar{\bm{A}}_m(\bm{u}_1, \bm{u}_2) (\bm{u}_1 - \bm{u}_2).
\]
\end{theorem}
The proof of Theorem~\ref{thm:roe_linearization} in \cite{godlewski2013numerical} constructs $\bar{\bm{A}}_m(\bm{u}_1, \bm{u}_2)$ as the product of matrices defined in terms of the flux and entropy variable Jacobian matrices integrated over phase space; thus, $\bar{\bm{A}}_m(\bm{u}_1, \bm{u}_2)$ is a continuous function of both its arguments.} We can now prove the following estimate:
\begin{lemma}
\label{lemma:entropy_error}
Let $\bm{u}_h \in \LRs{P^N{D^k}}^d$ and define \vnote{$\tilde{\bm{u}} = \bm{u}\LRp{\Pi_N\bm{v}(\bm{u}_h)}$. Assuming that $\tilde{\bm{u}}$, $\bm{f}_m(\bm{u}_h)$, and $\bm{f}_m(\tilde{\bm{u}})$ are differentiable,} the volume entropy residual $\delta_k(\bm{u}_h)$ \eqref{eq:entropy_ineq_error} satisfies
\begin{align}
&\delta_k(\bm{u}_h) = \rnote{\sum_{m=1}^d} {\int_{D^k}{-\pd{\Pi_N \bm{v}(\bm{u}_h)}{x_m}}^T\bm{f}_m(\bm{u}_h) + \int_{\partial {D^k}}{\psi_m(\tilde{\bm{u}}) n_m}} = \nonumber\\ &\rnote{\sum_{m=1}^d}\int_{D^k}\LRp{\bar{\bm{A}}_m\pd{\Pi_N\bm{v}(\bm{u}_h)}{x_m} - \Pi_N\LRp{\bar{\bm{A}}_m\pd{\Pi_N\bm{v}(\bm{u}_h)}{x_m}}}^T\LRp{\Pi_N \bm{v}(\bm{u}_h) - \bm{v}(\bm{u}_h)}
\label{eq:entropy_error}
\end{align}
for $m = 1,\ldots, d$, where \vnote{$\bar{\bm{A}}_m = \bar{\bm{A}}_m(\tilde{\bm{u}}, \bm{u}_h)$ is a Roe-type linearization matrix.}.
\end{lemma}
\begin{proof}
\vnote{Under the assumptions of differentiability}, the equality version of the entropy identity \eqref{eq:cell_entropy_equality} holds under exact integration since $\tilde{\bm{u}} = \bm{u}\LRp{\Pi_N\bm{v}(\bm{u}_h)}$ and
\begin{equation}
\int_{\partial {D^k}}{\psi_m(\tilde{\bm{u}}) n_m} = \int_{D^k} {\pd{\Pi_N\bm{v}(\bm{u}_h)}{x_m}}^T\bm{f}_m(\tilde{\bm{u}}).
\label{eq:equality_entropy_identity}
\end{equation}
Applying this to the left-hand side of \eqref{eq:entropy_error} yields
\[
\int_{D^k}{-\pd{\Pi_N \bm{v}(\bm{u}_h)}{x_m}}^T\bm{f}_m(\bm{u}_h) + \int_{\partial {D^k}}{\psi_m(\tilde{\bm{u}}) n_m}  = \int_{D^k}{\pd{\Pi_N \bm{v}(\bm{u}_h)}{x_m}}^T\LRp{\bm{f}_m(\tilde{\bm{u}})-\bm{f}_m(\bm{u}_h)}.
\]
Since $\tilde{\bm{u}} = \bm{u}\LRp{\Pi_N\bm{v}(\bm{u}_h)}$, \vnote{Theorem~\ref{thm:roe_linearization} implies there exists a matrix $\bar{\bm{A}}_m = \bar{\bm{A}}_m(\tilde{\bm{u}}, \bm{u}_h)$ such that}
\[
\bm{f}_m(\tilde{\bm{u}})-\bm{f}_m(\bm{u}_h) = \bar{\bm{A}}_m\LRp{\Pi_N \bm{v}(\bm{u}_h) - \bm{v}(\bm{u}_h)}.
\]
This yields that
\[
\int_{D^k}{\pd{\Pi_N \bm{v}(\bm{u}_h)}{x_m}}^T\LRp{\bm{f}_m(\tilde{\bm{u}})-\bm{f}_m(\bm{u}_h)} = \int_{D^k}\LRp{\bar{\bm{A}}_m\pd{\Pi_N\bm{v}(\bm{u}_h)}{x_m}}^T\LRp{\Pi_N \bm{v}(\bm{u}_h) - \bm{v}(\bm{u}_h)}.
\]
Finally, since ${\Pi_N \bm{v}(\bm{u}_h) - \bm{v}(\bm{u}_h)}$ is the $L^2$ projection error, it is orthogonal to any degree $N$ polynomial. This implies that it is orthogonal to $\Pi_N \LRp{\bar{\bm{A}}_m\pd{\Pi_N\bm{v}(\bm{u}_h)}{x_m}}$; the proof is completed by subtracting 
\[
\Pi_N \LRp{\bar{\bm{A}}_m\pd{\Pi_N\bm{v}(\bm{u}_h)}{x_m}}^T \LRp{\Pi_N \bm{v}(\bm{u}_h) - \bm{v}(\bm{u}_h)} = 0.
\]
\end{proof}
Applying the Cauchy-Schwarz inequality to \eqref{eq:entropy_error} yields that 
\begin{align*}
&\LRb{\int_{D^k}{-\pd{\Pi_N \bm{v}(\bm{u}_h)}{x_m}}^T\bm{f}_m(\bm{u}_h) + \int_{\partial {D^k}}{\psi_m(\tilde{\bm{u}}) n_m}} \leq\\ 
&\nor{\bar{\bm{A}}_m\pd{\Pi_N\bm{v}(\bm{u}_h)}{x_m} - \Pi_N\LRp{\bar{\bm{A}}_m\pd{\Pi_N\bm{v}(\bm{u}_h)}{x_m}}}_{L^2(D^k)}\nor{\Pi_N \bm{v}(\bm{u}_h) - \bm{v}(\bm{u}_h)}_{L^2(D^k)}.
\end{align*}
If $\bm{u}_h$ is a high order approximation of a sufficiently regular solution $\bm{u}$ and the entropy $S(\bm{u}_h)$ is convex and sufficiently regular (e.g., when density and internal energy are uniformly bounded away from zero for the compressible Euler equations) such that $\bm{v}(\bm{u})$ is also a sufficiently regular mapping, \vnote{then both $\bar{\bm{A}}_m\pd{\Pi_N\bm{v}(\bm{u}_h)}{x_m}$ and $\bm{v}(\bm{u}_h)$ are sufficiently regular functions. Then, standard $L^2$ best approximation estimates \cite{brenner2008mathematical} imply that $\nor{\bar{\bm{A}}_m\pd{\Pi_N\bm{v}(\bm{u}_h)}{x_m} - \Pi_N\LRp{\bar{\bm{A}}_m\pd{\Pi_N\bm{v}(\bm{u}_h)}{x_m}}}_{L^2(D^k)}$ and $\nor{\Pi_N \bm{v}(\bm{u}_h) - \bm{v}(\bm{u}_h)}_{L^2(D^k)}$ are both $O(h^{N+1+\sqrt{d}})$, where the factor of $\sqrt{d}$ comes from the scaling $\LRb{D^k} = O(h^d)$ in $d$ dimensions. 
As a result, the magnitude of the volume entropy residual is }
\begin{equation}
\delta_k(\bm{u}_h) = {\int_{D^k}{-\pd{\Pi_N \bm{v}(\bm{u}_h)}{x_m}}^T\bm{f}_m(\bm{u}_h) + \int_{\partial {D^k}}{\psi_m(\tilde{\bm{u}}) n_m}} = O(h^{2N+2 + d})
\label{eq:volume_entropy_residual_estimate}
\end{equation}
This result is a sharpening of Theorem 5.5 in \cite{chen2020review} and consistency estimates in \cite{gaburro2023high, mantri2024fully}.


\begin{remark}
The observation that \eqref{eq:equality_entropy_identity} is satisfied exactly under exact integration implies that standard weak form DG with the entropy projection and an entropy stable interface flux is entropy stable up to the accuracy of the volume quadrature. This was first pointed out in \cite{colombo2022entropy}. 
\end{remark}

\subsubsection{The effect of inexact quadrature}
\label{sec:quadrature}
While we have assumed exact integration in the proof of Lemma~\ref{lemma:entropy_error}, a similar estimate holds for a sufficiently accurate quadrature. Suppose that the volume quadrature is exact for polynomial integrands $f \in P^{M_{\rm vol}}(D^k)$ and the surface quadrature is exact for polynomial integrands $f \in P^{M_{\rm surf}}(\partial D^k)$. For example, approximating 1D volume integrals using an $(N+1)$ point Gauss quadrature rule would be exact for $f \in P^{2N+1}$, such that $M_{\rm vol} = 2N+1$. Then, since $\pd{\Pi_N \bm{v}(\bm{u}_h)}{x_m}\in P^{N-1}\LRp{D^k}$, under sufficient regularity of $\bm{f}_m$ and $\bm{u}_h$,
\begin{gather*}
\LRb{\LRp{\bm{f}_m(\tilde{\bm{u}}), \pd{\Pi_N \bm{v}(\bm{u}_h)}{x_m}} - \int_{D^k}\pd{\Pi_N \bm{v}(\bm{u}_h)}{x_m}^T\bm{f}_m(\tilde{\bm{u}}) } = O(h^{M_{\rm vol} + 1 + d}) \\
\LRb{\LRa{\psi_m(\tilde{\bm{u}}) n_m, 1} - \int_{\partial D^k}\psi_m(\tilde{\bm{u}})n_m} = O(h^{M_{\rm surf} + d})
\end{gather*}
Thus, if $M_{\rm vol} \geq 2N+1$ and $M_{\rm surf} \geq 2N+2$, then the quadrature error is of the same order as the volume entropy residual estimate in Lemma~\ref{lemma:entropy_error}. We note that these quadrature exactness conditions are one order higher than necessary to prove optimal rates of convergence in \cite{huang2017error}.

\subsection{Spurious \vnote{gradient} null space modes}

Recall that the \gnote{entropy correction artificial viscosity coefficient} \eqref{eq:eps} involves the ratio between the volume entropy residual and the viscous entropy dissipation on an element $D^k$:
\[
\epsilon_k(\bm{u}_h) = \frac{\min(0, \delta_k(\bm{u}_h))}{\sum_{i,j=1}^d  \LRp{\bm{\Theta}_i, \bm{K}_{ij}\bm{\Theta}_j}_{D^k}}.
\]
We wish to control the magnitude of the artificial viscosity coefficient; if $\epsilon_k(\bm{u}_h)$ is too large, it can negatively impact the maximum stable time-step size. Thus, if the denominator approaches zero, the numerator should approach zero at the same rate or faster. 

Since $\bm{\Theta}_i$ is a consistent approximation of $\pd{\Pi_N \bm{v}(\bm{u}_h)}{x_i}$, the denominator vanishes if the solution is constant. Luckily, the numerator (the volume entropy residual) also vanishes if the solution is constant, and numerical experiments suggest that the numerator converges to zero at the same rate or faster as $\bm{u}_h$ approaches a constant. Unfortunately, for the BR-1 viscous discretization used in this work, spurious modes make it possible for the viscous entropy dissipation (the denominator) to be arbitrarily small compared to the volume entropy residual (the numerator). 

The gradient $\bm{\Theta}_i$ is approximated in \eqref{eq:ldg1} using a standard DG discretization of the gradient with a central flux. It is known that BR-1 discretizations of the Laplacian result in spurious non-constant null space eigenmodes \cite{sherwin20062d, hesthaven2007nodal}. Discretizations of the gradient with a central flux \eqref{eq:ldg1} also admit similar spurious modes. such that applying \eqref{eq:ldg1} to such modes yields $\bm{\Theta}_1, \ldots, \bm{\Theta}_d = \bm{0}$ \cite{john2016stable}. Figure~\ref{fig:spurious} illustrates examples of such modes for degree $N=2$ approximations on uniform periodic meshes.
\begin{figure}
\centering
\includegraphics[width=.49\textwidth]{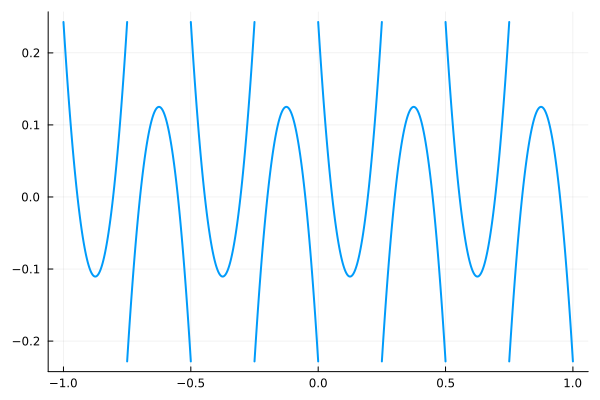}
\includegraphics[width=.49\textwidth]{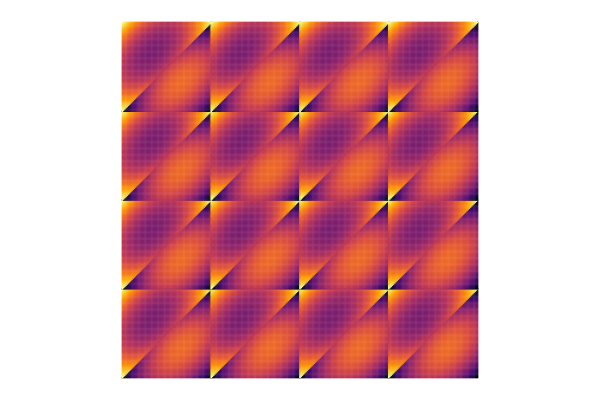}
\caption{Examples of spurious null space modes of the BR-1 gradient in 1D and 2D.}
\label{fig:spurious}
\end{figure}
While adding penalty terms to the divergence equation \eqref{eq:ldg3} suppresses such spurious eigenmodes for Laplacian and viscous flux derivatives \cite{carpenter2014entropy, chan2022entropy}, they do not suppress spurious \textit{gradient} modes. 

These spurious modes do not appear to be a major issue in practice when an upwind-like numerical flux is used for the convective discretization \eqref{eq:dg_form}, as the jump penalization tends to rapidly dissipate away spurious modes \cite{chan2017penalty}. We also note that non-central discretizations of the gradient, such as in the local DG method \cite{cockburn1998local, cockburn2007analysis}, do not produce spurious null space modes, and when combined with the approach in this work, tend to produce smaller \gnote{entropy correction artificial viscosity} coefficients. This will be analyzed in more detail in future work.

\section{Comparison with local entropy correction terms}
\label{sec:entropy_correction}

The \gnote{entropy correction artificial viscosity} constructed in this work \bnote{is inspired by and} most closely resembles the \bnote{local} entropy correction term of Abgrall \cite{abgrall2018general} and Abgrall, \"{O}ffner, and Ranocha \cite{abgrall2022reinterpretation}. \bnote{These local correction techniques add a dissipative term based on the zero-mean variation of the entropy variables over an element, and use tune the magnitude of this dissipative term to ensure a local entropy inequality}. More recent versions of this entropy correction term have both introduced the change of variables matrix $\bm{A}_0 = \pd{\bm{u}}{\bm{v}}$ as an inner product scaling and replaced the zero-mean variation term with the local derivative \cite{gaburro2023high, mantri2024fully}. \bnote{Since there are several versions of the aforementioned local entropy correction terms in the literature, we will construct two local entropy corrections which correspond most directly to the entropy correction artificial viscosity introduced in Section~\ref{sec:av}. We will then compare these two versions of the local entropy correction to some existing local entropy corrections in the literature \cite{abgrall2022reinterpretation, gaburro2023high, mantri2024fully}. 

The first local entropy correction we will construct is based on the deviation of the entropy variables from the mean value (corresponding to \cite{abgrall2018general, abgrall2022reinterpretation}). Using the notation introduced in this work, the mean-value local entropy correction can be expressed as follows: 
\begin{gather}
\LRp{\bm{g}_{\rm visc}, \bm{w}}_{D^k} = \LRp{\epsilon_k(\bm{u}_h) \bm{A}_0 (\Pi_N \bm{v}(\bm{u}_h) - \bar{\bm{v}}), \bm{w}}_{D^k}, \quad \forall \bm{w}\in \LRs{P^N\LRp{D^k}}^n \label{eq:mv_local_entropy_correction}\\
\epsilon_k(\bm{u}_h) \geq \frac{-\min(0, \delta_k(\bm{u}_h))}{\LRp{\bm{A}_0 (\Pi_N \bm{v}(\bm{u}_h) - \bar{\bm{v}}), \Pi_N \bm{v}(\bm{u}_h) - \bar{\bm{v}}}_{D^k}}, \label{eq:mv_local_entropy_correction_coeff}
\end{gather}
where $\bar{\bm{v}}$ denotes the cell average of $\Pi_N \bm{v}(\bm{u}_h)$. We additionally assume that $\bm{A}_0 =  \pd{\bm{u}}{\bm{v}}$ is evaluated using the cell average of the solution. Note that the entropy dissipation in the denominator of \eqref{eq:mv_local_entropy_correction_coeff} can then be recovered by taking $\bm{w} = \Pi_N\bm{v}(\bm{u}_h)$ in \eqref{eq:mv_local_entropy_correction} and using the fact that $\Pi_N \bm{v}(\bm{u}_h) - \bar{\bm{v}}$ is $L^2$ orthogonal to constants.

The second local entropy correction we will construct is a derivative-based version (corresponding to, for example, \cite{gaburro2023high, mantri2024fully}). 
The derivative-based local entropy correction can be expressed as follows:
\begin{gather}
\LRp{\bm{g}_{\rm visc}, \bm{w}}_{D^k} = \sum_{m=1}^d \LRp{\epsilon_k(\bm{u}_h) \bm{A}_0 \pd{\Pi_N\bm{v}(\bm{u}_h)}{x_m}, \pd{\bm{w}}{x_m}}_{D^k}, \quad \forall \bm{w}\in \LRs{P^N\LRp{D^k}}^n \label{eq:deriv_local_entropy_correction}
\\
\epsilon_k(\bm{u}_h) \geq \frac{-\min(0, \delta_k(\bm{u}_h))}{\sum_{m=1}^d \LRp{\bm{A}_0 \pd{\Pi_N\bm{v}(\bm{u}_h)}{x_m}, \pd{\Pi_N\bm{v}(\bm{u}_h)}{x_m}}_{D^k} }_{D^k}. \nonumber
\end{gather}
Note that, for degree $N=1$, $\pd{\Pi_N\bm{v}(\bm{u}_h)}{x_i}$ and $\Pi_N \bm{v}(\bm{u}_h) - \bar{\bm{v}}$ are identical up to a constant scaling. Since the local correction terms are independent of element-wise constant scalings of the entropy dissipation term (any scaling will simply be absorbed into the coefficient $\epsilon_k(\bm{u}_h)$), the contributions from both local entropy correction terms \eqref{eq:mv_local_entropy_correction} and \eqref{eq:deriv_local_entropy_correction} will be identical for the case of $N=1$.

The first difference between the local corrections defined in \eqref{eq:mv_local_entropy_correction}, \eqref{eq:deriv_local_entropy_correction} and the terms introduced in \cite{abgrall2022reinterpretation} is the presence of an inner product scaling $\bm{A}_0$. As mentioned previously, the inner product is now scaled by the change of variables matrix $\bm{A}_0 = \pd{\bm{u}}{\bm{v}}$, whereas in \cite{abgrall2022reinterpretation}, no matrix scaling was used. We have observed this to make a significant difference in the robustness of the local entropy correction: if the $L^2$ inner product is not scaled by $\bm{A}_0$, DG solutions of the compressible Euler equations which contain shocks tend to blow up rapidly. 

The second difference is the definition of the entropy residual. In \cite{abgrall2022reinterpretation}, the entropy residual is computed in terms of numerical entropy fluxes involving exterior values of the solution at element interfaces. Here, we use the volume entropy residual \eqref{eq:entropy_ineq_error} instead, which is a purely local estimate. Moreover, the estimates for the entropy residual derived in \cite{abgrall2018general} are $O(h^{N+1+d})$, while the estimates from Lemma~\ref{lemma:entropy_error} and \eqref{eq:volume_entropy_residual_estimate} for the magnitude of $\delta_k(\bm{u}_h)$ are $O(h^{2N+2+d})$.}

\begin{figure}
\centering
\subfloat[$N=1$, $128\times 128 \times 2$ elements]{\includegraphics[width=.4\textwidth, trim={31.5em 7em 31.5em 7em}, clip]{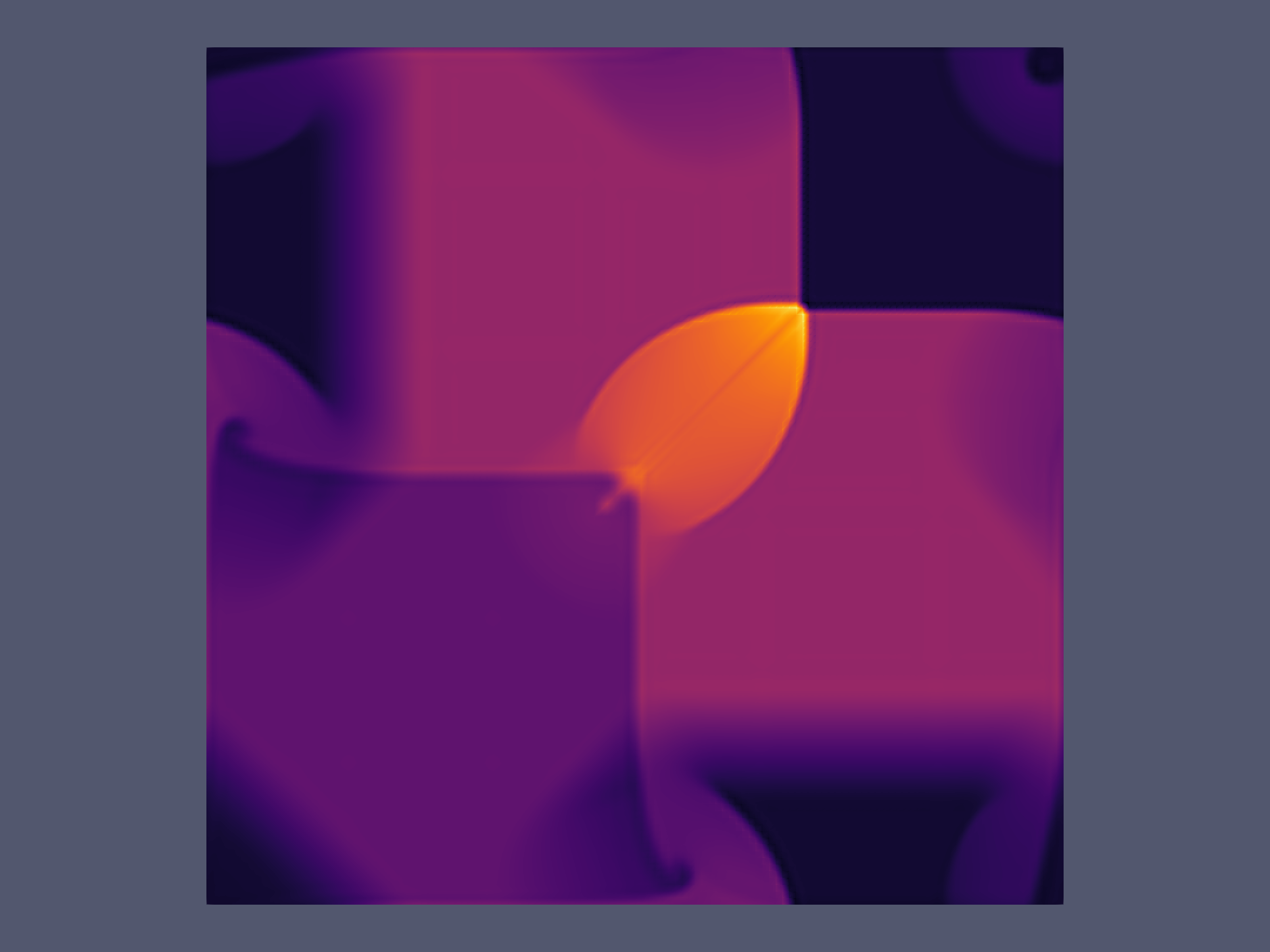}}
\hspace{.2em}
\subfloat[$N=3$, $64\times 64 \times 2$ elements]{\includegraphics[width=.4\textwidth, trim={31.5em 7em 31.5em 7em}, clip]{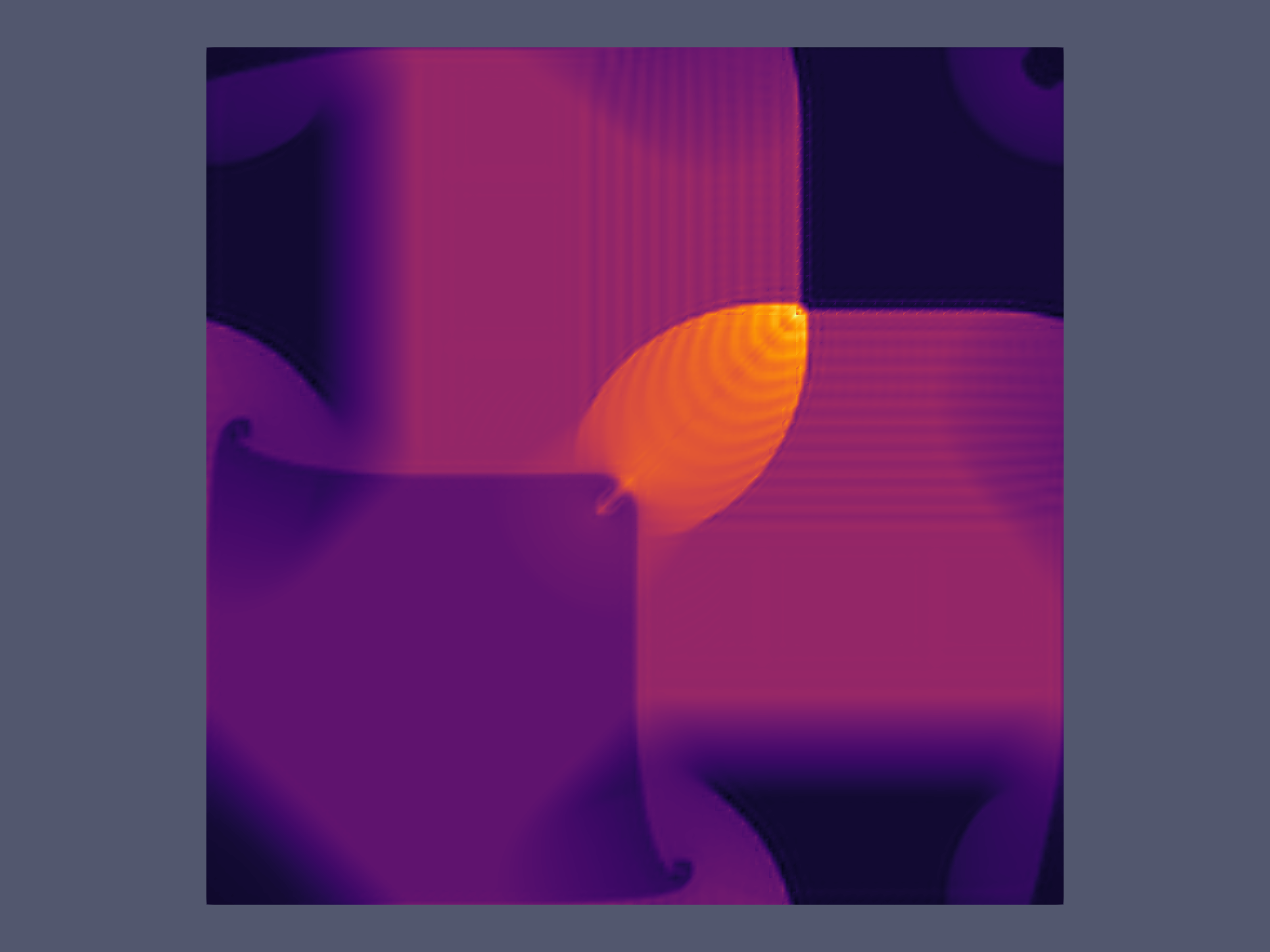}}
\caption{Solutions \rnote{(density)} to the 2D Riemann problem using \bnote{a matrix-weighted version of the local entropy correction term \eqref{eq:mv_local_entropy_correction}}. Note that the degree $N=3$ solution contains numerical artifacts along the shock front.}
\label{fig:correction}
\end{figure}

\bnote{Finally, we note that the two local entropy corrections \eqref{eq:mv_local_entropy_correction} and \eqref{eq:deriv_local_entropy_correction} do not behave the same in practice. While both approaches seemed to stabilize high order DG solutions with shocks in 1D, we observed that adding only the derivative-based local entropy correction term of \cite{gaburro2023high, mantri2024fully} was not sufficient to stabilize DG simulations of the 2D Riemann problem from \cite{chan2018discretely} (details of the problem setup are provided in Section~\ref{sec:riemann}). However, while \eqref{eq:mv_local_entropy_correction} successfully stabilizes the 2D Riemann problem, it also produces non-physical artifacts at higher orders of approximation. Figure~\ref{fig:correction} shows DG solutions using \eqref{eq:mv_local_entropy_correction}. The degree $N=3$ solution contains numerical artifacts along the shock front moving towards the upper-right hand corner, though these artifacts are not observed for the degree $N=1$ solution. }

\subsection{Comparison with other artificial viscosity methods}

The literature on artificial viscosity is vast; even when restricting the literature review specifically to high order DG methods, there are hundreds of papers on artificial viscosity methods for nonlinear conservation laws. These artificial viscosity methods incorporate a variety of techniques, from modal and residual-based indicator methods to data-driven or machine-learning approaches \cite{persson2006sub, barter2010shock, klockner2011viscous, lv2014taming, michoski2016comparison, glaubitz2019smooth, discacciati2020controlling, yu2020study, zeifang2021data}. 

The artificial viscosity in this paper differs from many artificial viscosity methods in that the ``indicator'' $\epsilon_k(\bm{u}_h)$ is chosen to be the smallest cell-local value to guarantee an entropy inequality. As a result of this choice, the artificial viscosity in this paper \textit{is not intended to damp spurious oscillations}. For example, when applied to the constant coefficient linear scalar advection equation, the artificial viscosity coefficient $\epsilon_k(\bm{u}_h)$ vanishes, since the standard weak formulation of DG already satisfies a cell entropy inequality and the volume entropy residual $\delta_k(\bm{u}_h) = 0$. This behavior is contrast to artificial viscosity methods which are designed based on solution regularity or smoothness indicators. We note that it is straightforward to increase the magnitude of the artificial viscosity indicator $\epsilon_k(\bm{u}_h)$ if one wishes to damp spurious oscillations; however, the intent of the artificial viscosity in this work is to determine a \textit{minimal} artificial viscosity which yields an entropy inequality. 

The \gnote{entropy correction artificial viscosity} and choice of $\epsilon_k(\bm{u}_h)$ \eqref{eq:eps} are also related to the ``entropy residual" or ``entropy commutator'' indicators used in \cite{guermond2018second, guermond2019invariant}, which are themselves a generalization of entropy viscosity and entropy residual methods \cite{guermond2011entropy, zingan2013implementation, lv2016entropy}. The entropy commutator approach to entropy viscosity adds dissipation proportionally to the violation of the chain rule
\begin{equation}
\bm{v}(\bm{u}_h)^T\pd{\bm{f}(\bm{u}_h)}{x} - \pd{F(\bm{u})}{x} \approx 0, \qquad F(\bm{u}) = \bm{v}(\bm{u}_h)^T \bm{f}(\bm{u}_h) - \psi(\bm{u}_h).
\label{eq:chainrule}
\end{equation}
The approach taken in this work also adds dissipation proportional to the violation of \eqref{eq:chainrule}, but with some distinctions. 
The magnitude of artificial viscosities (including entropy viscosity) typically depends on a heuristic normalization and scaling \cite{kornelus2017scaling}. In contrast, once the viscosity model (e.g., the viscous matrices $\bm{K}_{ij}$ and viscous DG formulation) are determined, there are no parameters to tune for the artificial viscosity used in this work. Moreover, in addition to using a $P^1$ continuous finite element approximation instead of a high order DG discretization, the entropy inequality in \cite{guermond2018second, guermond2019invariant} is approximately localized around the support of a single $C^0$ nodal basis function using Lagrange interpolation, while we localize the entropy inequality exactly using the $L^2$ projection of the entropy variables over a single DG element. 

\section{Numerical experiments}
\label{sec:numerical} 

In this section, we present numerical experiments which confirm the robustness and accuracy of the proposed \gnote{entropy correction artificial viscosity}. All experiments are implemented in Julia using the \verb+StartUpDG.jl+ and \verb+Trixi.jl+ \cite{ranocha2022adaptive} libraries. Unless otherwise stated, all experiments utilize uniform meshes. In 1D, we investigate both nodal DGSEM formulations based on collocation at Legendre-Gauss-Lobatto nodes (referred to in figures as ``nodal'' DG discretizations), as well as degree $N$ modal DG formulations based on an $(N+2)$ point Gauss-Legendre quadrature rule (referred to in figures as ``modal'' DG discretizations). In 2D, we focus on total degree $N$ approximations on triangular meshes, and utilize volume quadratures from \cite{xiao2010numerical} which are exact for degree $2N$ polynomials, as well as $(N+1)$ point Gauss quadratures on faces. 

\rnote{Unless otherwise specified, we utilize the local Lax-Friedrichs flux with Davis wavespeed estimate \cite{davis1988simplified}. However, Appendix~\ref{sec:additional_1d} contains additional numerical experiments with the HLLC flux \cite{batten1997choice} and an entropy stable Roe-like flux \cite{winters2017uniquely, waruszewski2022entropy} as well.} For time integration, we utilize the \verb+OrdinaryDiffEq.jl+ library \cite{rackauckas2017differentialequations}. The adaptive 4-stage 3rd order strong stability preserving Runge-Kutta (SSPRK43) method \cite{kraaijevanger1991contractivity, fekete2022embedded, ranocha2022optimized} is used for all experiments. In 1D, the absolute and relative tolerance are set to $(10^{-8}, 10^{-6})$, and in 2D, the absolute and relative tolerances are $(10^{-6}, 10^{-4})$. 

\subsection{Compressible Euler equations}
All experiments are performed for the 1D and 2D compressible Euler equations, which are described below. 
Let $\bm{u}$ denote the vector of conservative variables, which in 2D are
\[
\bm{u} = \LRc{\rho, \rho u_1, \ldots, \rho u_d, E} \in \mathbb{R}^{d+2}.
\]
Here, $\rho$ is density, $u_i$ is the velocity in the $i$th coordinate direction, and $E$ is the specific total energy. The pressure $p$ is related to the conservative variables through the constitutive relations
\[
p = (\gamma-1) \rho e, \qquad E = {e+\frac{1}{2} \sum_{i=1}^d u_i^2},
\]
where $\gamma = 1.4$ and $e$ is the internal energy density. The compressible Euler equations in $d$ dimensions are given by 
\begin{equation}
\pd{\bm{u}}{t} + \sum_{i=1}^d \pd{\bm{f}_i}{x_i}  = \bm{0},
\end{equation}
where $\bm{f}_i$ denote the convective fluxes along the $i$th coordinate direction. For $d=2$, the inviscid fluxes $\bm{f}_i$ are given by
\[
\bm{f}_1 = \begin{bmatrix}
\rho u_1\\
\rho u_1^2 + p\\
\rho u_1u_2 \\
u_1 (E + p)
\end{bmatrix}, 
\qquad
\bm{f}_2 = \begin{bmatrix}
\rho u_2\\
\rho u_1u_2\\
\rho u_2^2 + p\\
u_2 (E + p)
\end{bmatrix}.
\]
Setting velocity in either direction to zero recovers the 1D compressible Euler equations. 

There exist an infinite family of convex entropies for the compressible Euler equations \cite{harten1983symmetric}; however, the compressible Navier-Stokes equations admit a mathematical entropy inequality with respect to only a single entropy function $S(\bm{u})$ and entropy potential $\psi_m(\bm{u})$
\[
S(\bm{u}) = -\rho s, \qquad \psi_m(\bm{u}) = \rho u_m,
\]
where $s = \log\LRp{\frac{p}{\rho^\gamma}}$ denotes the physical entropy \cite{hughes1986new}. 
The derivative of the entropy with respect to the conservative variables yield the entropy variables $\bm{v}(\bm{u}) = \pd{S}{\bm{u}} = \LRc{v_1, v_2, v_3, v_4}$, where
\begin{equation}
v_1 = \frac{\rho e (\gamma + 1 - s) - E}{\rho e}, \qquad v_{1+ i}= \frac{\rho {{u}_i}}{\rho e}, \qquad v_{d+2} = -\frac{\rho}{\rho e} \label{eq:evars}
\end{equation}
for $i = 1,\ldots, d$.  The inverse mapping is given by
\begin{align*}
\rho = -(\rho e) v_{d+2}, \qquad 
\rho {u_i} = (\rho e) v_{1+i}, \qquad 
 E = (\rho e)\LRp{1 - \frac{\sum_{j=1}^d{v_{1+j}^2}}{2 v_{d+2}}},
\end{align*}
where $i = 1,\ldots,d$, and $\rho e$ and $s$ in terms of the entropy variables are 
\begin{equation*}
\rho e = \LRp{\frac{(\gamma-1)}{\LRp{-v_{d+2}}^{\gamma}}}^{1/(\gamma-1)}e^{\frac{-s}{\gamma-1}}, \qquad 
s = \gamma - v_1 + \frac{\sum_{j=1}^d{v_{1+j}^2}}{2v_{d+2}}.
\end{equation*}
Finally, explicit expressions for the Jacobian matrix $\pd{\bm{u}}{\bm{v}}$ are given in \cite{barth1999numerical} in terms of the sound speed $a$ and specific total enthalpy $H$. In 2D, the Jacobian matrix is
\begin{gather*}
\pd{\bm{u}}{\bm{v}} = \begin{bmatrix}
\rho 	& \rho u_1 		& \rho u_2 		& E\\
	& \rho u_1^2 + p	& \rho u_1 u_2 		& u_1 (E + p)\\
	&			  	& \rho u_2^2 + p	& u_2 (E + p)\\
	& 				&				& \rho H^2 - a^2 \frac{p}{\gamma - 1}
\end{bmatrix},\\
a = \sqrt{\gamma\frac{p}{\rho}}, \qquad H = \frac{a^2}{\gamma - 1} + \frac{1}{2} (u_1^2 + u_2^2),
\end{gather*}
where the lower triangular entries of $\pd{\bm{u}}{\bm{v}}$ are determined by symmetry.

\subsection{Density wave and high order accuracy}

We begin by examining the accuracy of the proposed discretization. We first consider the 2D density wave with amplitude $|A| < 1$:
\begin{gather*}
    \rho = 1 + A \sin(2\pi (x + y)), \quad  u = .1, \quad v = .2, \quad    p = 10.
\end{gather*}
We compute $L^2$ errors at final time $T= 1.7$ for amplitude $A = 0.5$ and show rates of convergence in Table~\ref{tab:density_wave_2d}. Optimal rates of convergence are observed for polynomial degrees $N=1,\ldots, 4$.
\begin{table}
\centering
\begin{tabular}{|c|cc|cc|cc|cc|cc|}
\hline
$h$	& $N = 1$ & Rate & $N = 2$ & Rate & $N = 3$ & Rate & $N = 4$ & Rate \\
\hline
$1/2$ & ${5.735\times 10^{-1}}$ & ${}$ & ${4.628\times 10^{-1}}$ & ${}$ & ${4.058\times 10^{-1}}$ & ${}$ & ${1.321\times 10^{-1}}$ & ${}$\\
$1/4$ & ${2.626\times 10^{-1}}$ & ${1.13}$ & ${8.553\times 10^{-2}}$ & ${2.44}$ & ${4.349\times 10^{-2}}$ & ${3.22}$ & ${1.297\times 10^{-2}}$ & ${3.35}$\\
$1/8$ & ${7.913\times 10^{-2}}$ & ${1.73}$ & ${1.713\times 10^{-2}}$ & ${2.32}$ & ${1.907\times 10^{-3}}$ & ${4.51}$ & ${3.425\times 10^{-4}}$ & ${5.24}$\\
$1/16$ & ${1.739\times 10^{-2}}$ & ${2.19}$ & ${3.035\times 10^{-3}}$ & ${2.50}$ & ${8.373\times 10^{-5}}$ & ${4.51}$ & ${1.016\times 10^{-5}}$ & ${5.08}$\\
$1/32$ & ${3.898\times 10^{-3}}$ & ${2.16}$ & ${3.663\times 10^{-4}}$ & ${3.05}$ & ${4.214\times 10^{-6}}$ & ${4.31}$ & ${3.230\times 10^{-7}}$ & ${4.98}$\\
\hline
\end{tabular}
\caption{Computed $L^2$ errors for the 2D density wave with amplitude $A = 0.5$ using \gnote{entropy correction artificial viscosity}.}
\label{tab:density_wave_2d}
\end{table}

Next, we examine the behavior over time of the $L^2$ error over time for the 1D entropy wave with amplitude $|A| < 1$. 
\begin{gather}
    \rho = 1 + A \sin(2\pi x), \quad  u = .1, \quad  p = 10.
    \label{eq:density_wave_1d}
\end{gather}
We consider both $A=0.5$, where the minimum density is far from zero, and $A=0.98$, where the minimum density is closer to zero. We evaluate errors for the standard nodal DG formulation, nodal DG with \gnote{entropy correction artificial viscosity}, and a flux differencing entropy stable nodal DG method using Ranocha's entropy conservative volume flux \cite{ranocha2020entropy}. 

\begin{figure}
\centering
\subfloat[$L^2$ error]{\includegraphics[width=.45\textwidth]{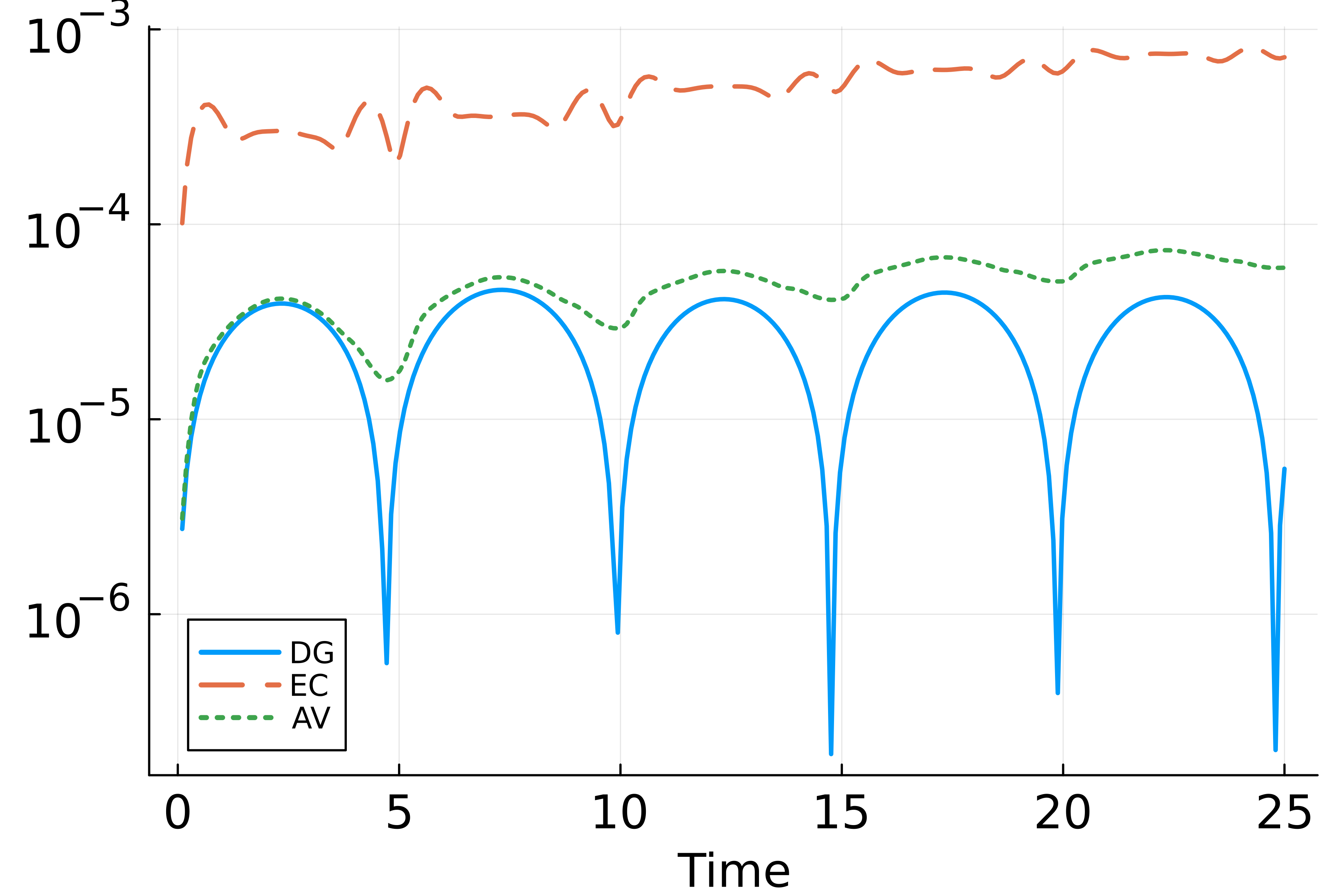}}
\hspace{.1em}
\subfloat[Change in entropy over time]{\includegraphics[width=.45\textwidth]{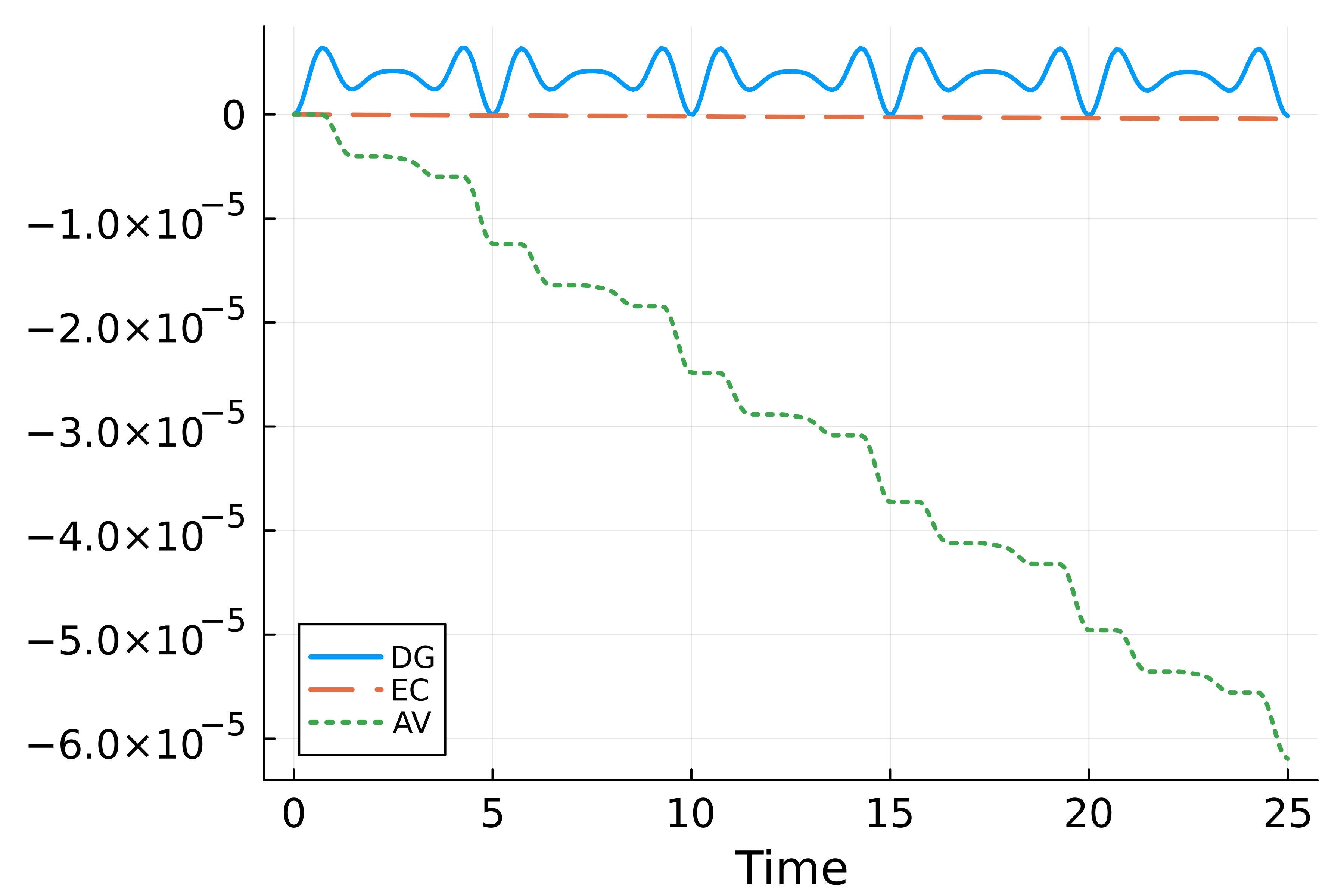}}
\caption{Evolution of $L^2$ error and entropy for the 1D density wave with amplitude $A=0.5$. Here, ``DG'' refers to the standard DG method, ``EC'' refers to a flux differencing entropy stable DG method, and ``AV'' refers to \gnote{DG with entropy correction artificial viscosity}.}
\label{fig:density_wave_error_over_time_1} 
\end{figure}

\begin{figure}
\centering
\subfloat[$L^2$ error]{\includegraphics[width=.45\textwidth]{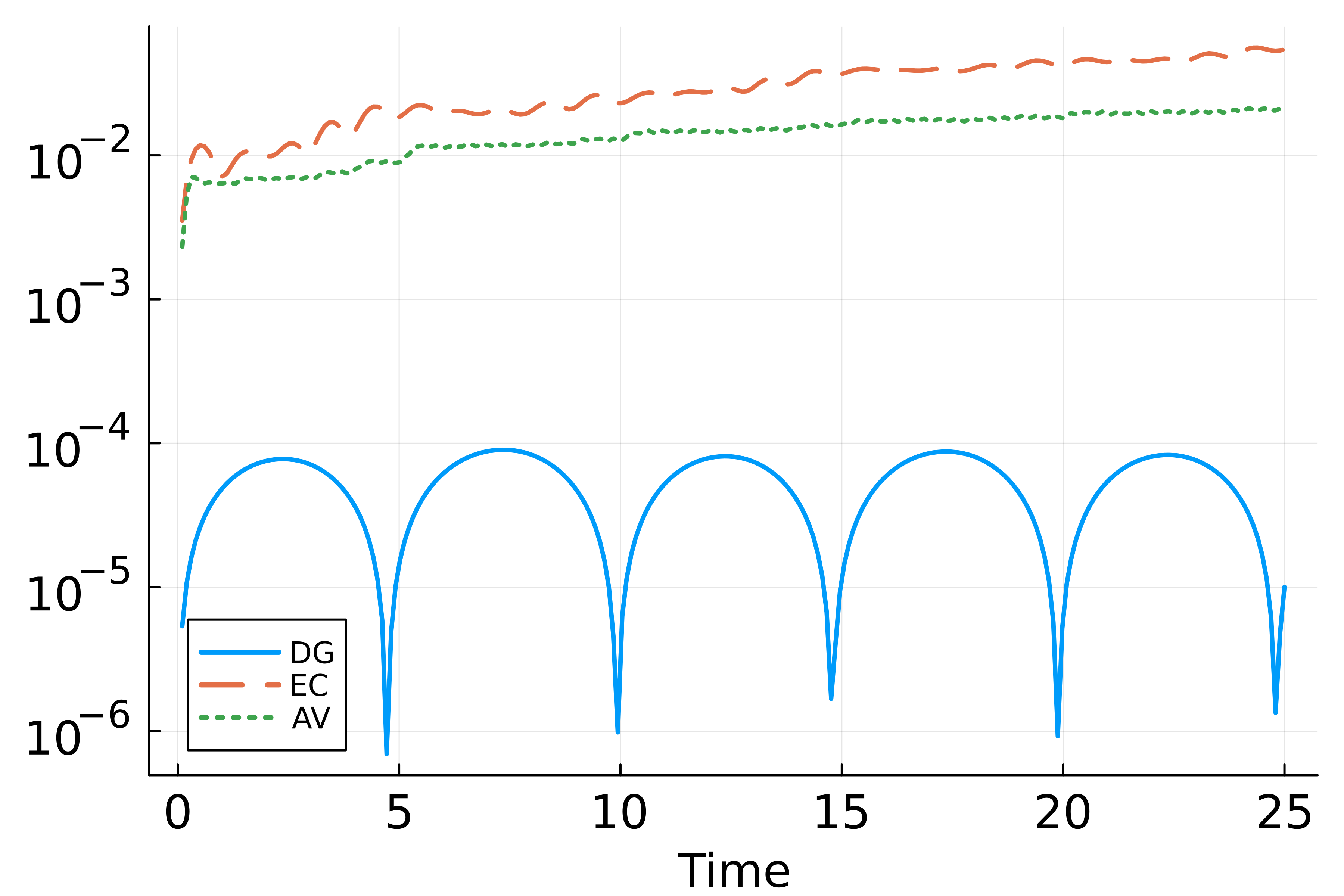}}
\hspace{.1em}
\subfloat[Change in entropy over time]{\includegraphics[width=.45\textwidth]{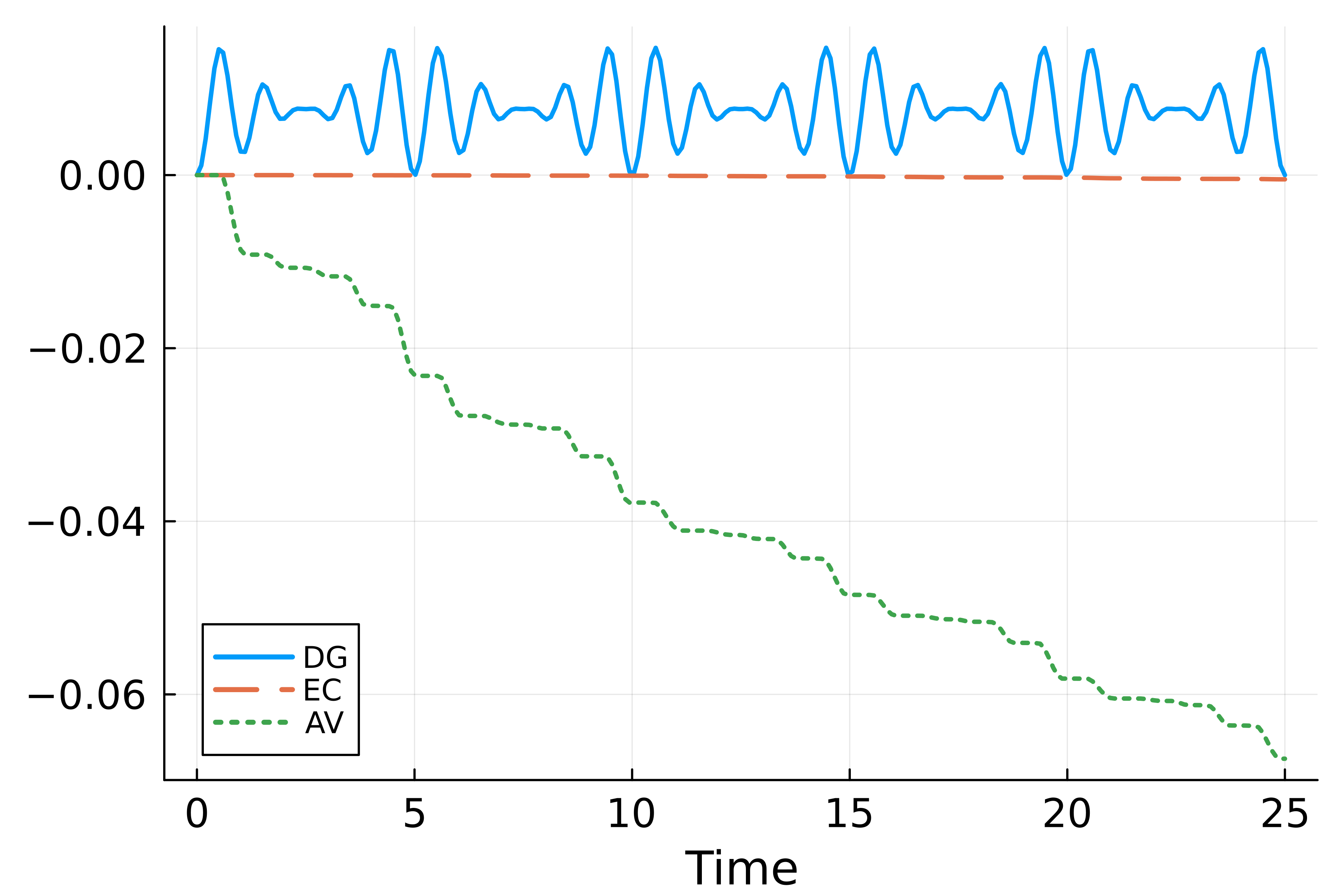}}
\caption{Evolution of $L^2$ error and entropy  for the 1D density wave with amplitude $A=0.98$.  Here, ``DG'' refers to the standard DG method, ``EC'' refers to a flux differencing entropy stable DG method, and ``AV'' refers to \gnote{DG with entropy correction artificial viscosity}.}
\label{fig:density_wave_error_over_time_2} 
\end{figure}

Figures~\ref{fig:density_wave_error_over_time_1} and \ref{fig:density_wave_error_over_time_2} shows plots of the $L^2$ errors and evolution of entropy up to time $T=25$ for a degree $N=7$ approximation over $4$ elements. For $A=0.5$, we observe that the $L^2$ errors are lowest for the standard nodal DG method, but that the errors for nodal DG with \gnote{entropy correction artificial viscosity} hover just above the standard nodal DG error. In contrast, the errors for entropy stable flux differencing nodal DG methods are about an order of magnitude larger. For $A=0.98$, both flux differencing and \gnote{entropy correction artificial viscosity} nodal DG result in larger errors compared with standard nodal DG; however, \gnote{entropy correction artificial viscosity} results in smaller errors than the flux differencing entropy stable scheme. We also observe in both Figures~\ref{fig:density_wave_error_over_time_1} and \ref{fig:density_wave_error_over_time_2} that the standard nodal DG scheme does not result in an entropy which is non-increasing in time. 

Finally, we note that we observe similar patterns in the evolution of error and entropy over time when repeating the same 1D density wave experiment with $A=0.5$ for a standard modal overintegrated DG scheme with $(N+2)$ Gauss quadrature points. For $A=0.98$, we observe that errors for the modal entropy stable schemes (both using flux differencing and \gnote{entropy correction artificial viscosity}) were larger in magnitude and resulted in a smaller maximum stable time-step size. This may be due to the sensitivity of the entropy projection for near-vacuum states \cite{chan2022entropyprojection}.

\subsection{\gnote{Local} linear stability for a smooth background flow} 

Flux differencing split form and entropy stable nodal DG methods are known to be less \gnote{locally} linearly stable than standard weak form nodal DG methods \cite{gassner2022stability, ranocha2021preventing}. In \cite{gassner2022stability}, the authors show that flux differencing with entropy conservative volume fluxes can be interpreted as adding potentially anti-diffusive correction to a \gnote{locally} linearly stable central scheme. It is known that a central volume flux recovers a standard nodal DG weak formulation \cite{gassner2016split}, and since the approach taken in this work adds a small dissipative correction to a standard weak DG formulation, we expect the \gnote{entropy correction artificial viscosity} approach in this paper to be more \gnote{locally} linearly stable than a split formulation.

\begin{figure}
\centering
\subfloat[Flux diff.\ with EC volume flux]{\includegraphics[width=.32\textwidth]{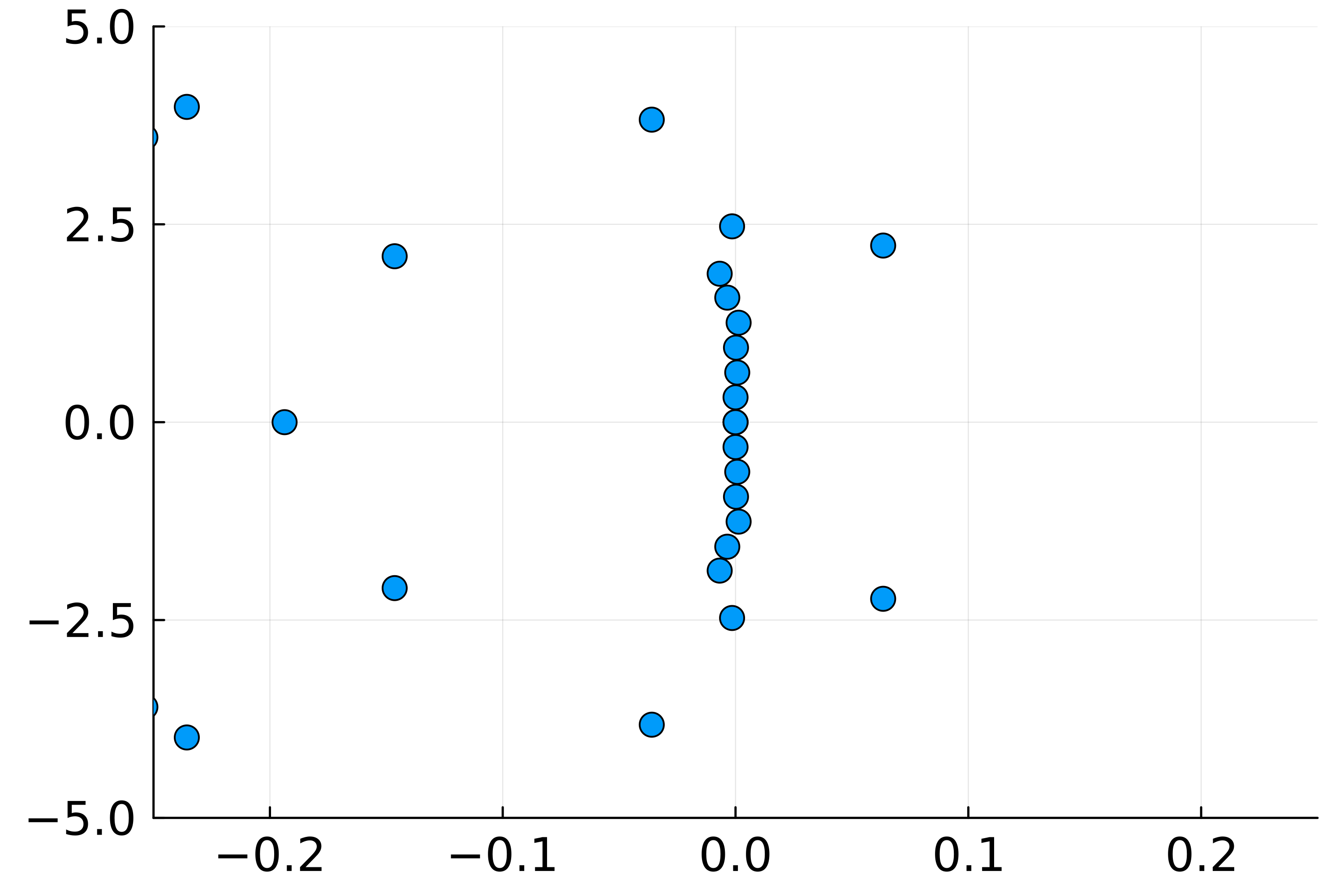}}
\hspace{.1em}
\subfloat[Entropy stable AV]{\includegraphics[width=.32\textwidth]{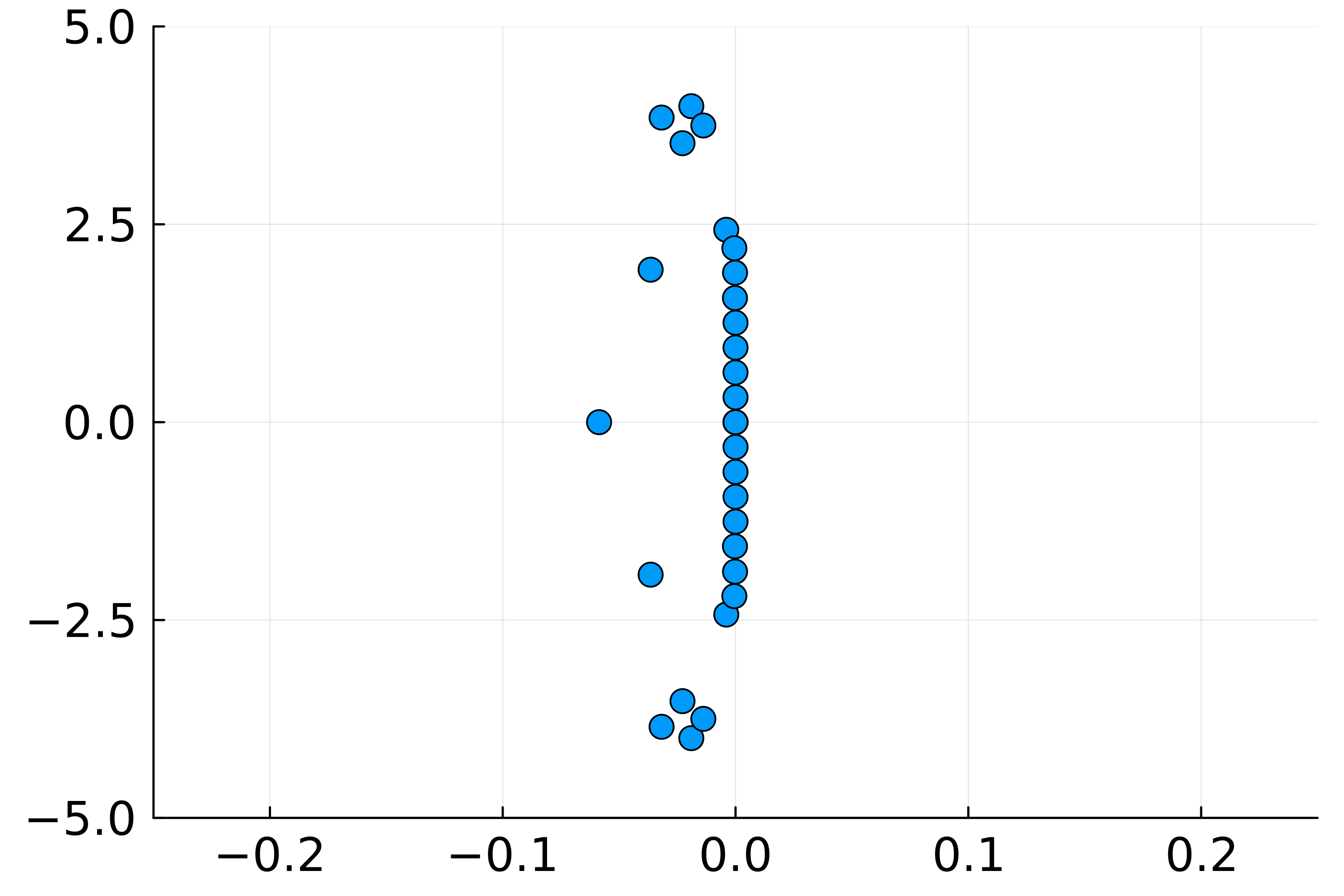}}
\hspace{.1em}
\subfloat[Standard DG]{\includegraphics[width=.32\textwidth]{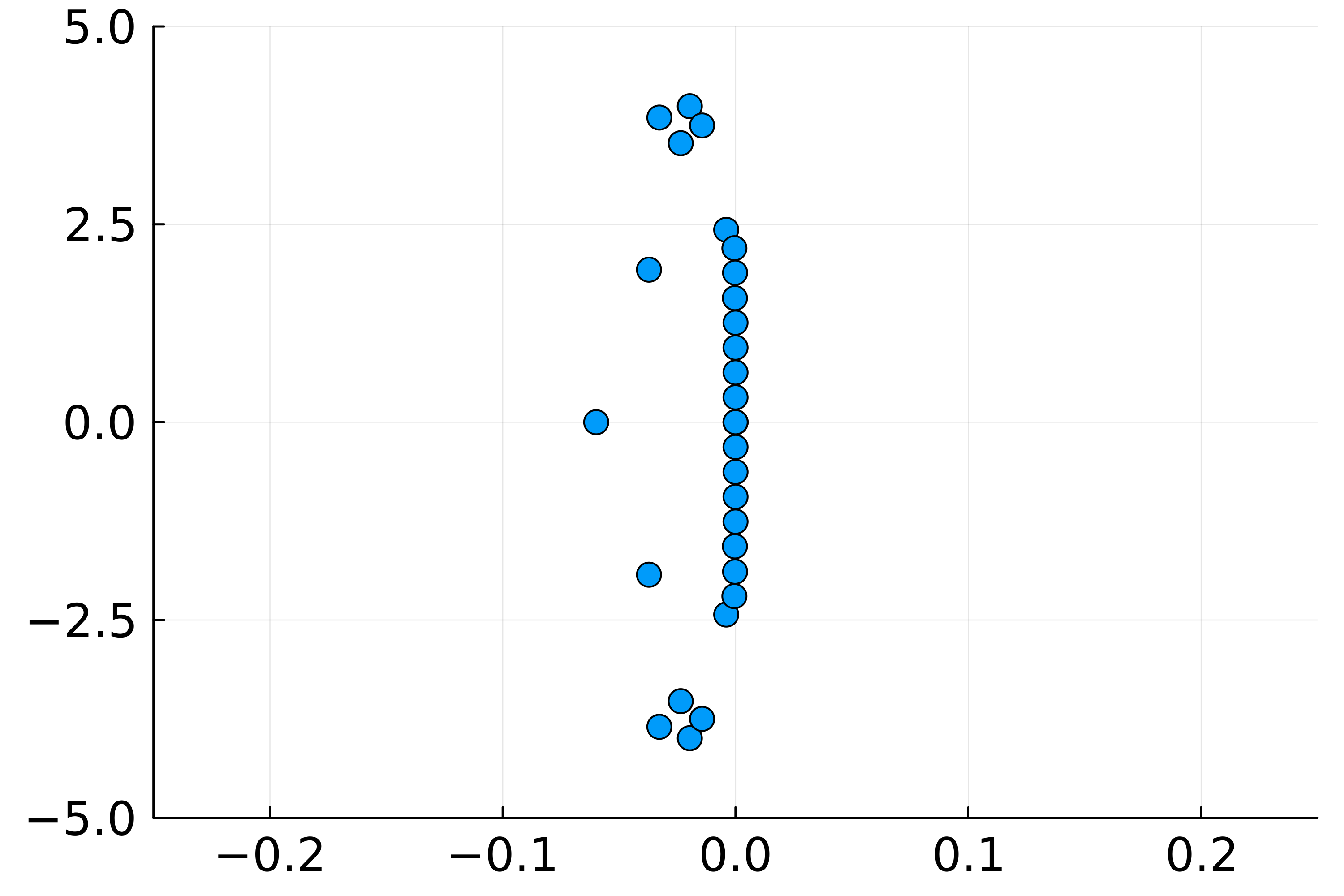}}
\caption{Spectra of the linearized Jacobian for \rnote{$A = 0.98$ with} a nodal DG discretization with $N=7, K=4$ and a local Lax-Friedrichs interface flux. The plots are zoomed near the origin to highlight the presence of eigenvalues with positive real parts. }
\label{fig:spectra}
\end{figure}

\begin{table}
\centering
\subfloat[Degree $N=3$ approximation]{
\begin{tabular}{|c|c|c|c|c|}
\hline
						& $K = 8$	 				& $K = 16$ 				& $K=32$ 			\\
\hline			
EC flux differencing		 	& $0.1349$				& $0.2560$				& $0.3883$				\\ 
\hline	
\gnote{Entropy correction AV}	& $0.0127$				& $2.2629 \times 10^{-13}$	& $2.7297\times 10^{-7}$		\\
\hline
Standard weak form DG 		& $6.9170 \times 10^{-8}$		& $6.7954 \times 10^{-10}$	& $2.2555\times 10^{-12}$	\\	
\hline
\end{tabular}
}
\\
\vspace{1em}
\subfloat[Degree $N=7$ approximation]{
\begin{tabular}{|c|c|c|c|c|}
\hline
						& $K=4$ 					& $K=8$					&$K = 16$	\\
\hline			
EC flux differencing		 	& $0.0634$				& $0.1163$				& $0.1971$	\\ 
\hline	
\gnote{Entropy correction AV}	& $5.132\times 10^{-7}$		& $0.0132$				& $0.0023$	\\
\hline
Standard weak form DG 		& $5.9776\times 10^{-14}$	& $1.1105\times 10^{-14}$	& $1.2291 \times 10^{-13}$	\\	
\hline
\end{tabular}
}
\caption{Maximum real parts of the linearized spectra for \rnote{$A = 0.98$ with} nodal DG discretizations with $K$ elements and a local Lax-Friedrichs interface flux. }
\label{tab:spectra}
\end{table}

We analyze the \gnote{local} linear stability of a nodal DGSEM formulation\footnote{We also experimented with a modal DG formulation using an $(N+2)$ point Gauss quadrature, but observed that the spectra included eigenvalues with large $O(100)$ positive real parts for the standard weak form DG scheme both with and without artificial viscosity. However, since we cannot differentiate between standard weak form DG and the \gnote{entropy correction artificial viscosity} based on the spectra, we do not consider those results here.} by computing the spectral of the linearized Jacobian using automatic differentiation \cite{revels2016forward}, where the background flow is given by the 1D density wave initial condition \eqref{eq:density_wave_1d} \rnote{with $A = 0.98$}. A local Lax-Friedrichs interface flux is used in all cases, and the entropy conservative flux of Ranocha \cite{ranocha2020entropy} is used for the volume flux within the entropy stable flux differencing nodal DG discretization. \gnote{Both the initial condition and the $\bm{K}_{ij}$ matrices in the entropy correction artificial viscosity are evaluated at Lobatto nodes. 

We emphasize that the computed spectra in Figure~\ref{fig:spectra} and Table~\ref{tab:spectra} for entropy correction artificial viscosity are sensitive to the choice of initial condition. We evaluated the linearized spectra using the solution after evolving the initial condition to final time $T = 2.0$ before computing the linearized operator, which we observed reduces the sensitivity of the maximum real part of the linearized spectra to some discretization choices (e.g., initializing the solution using interpolation instead of $L^2$ projection\footnote{For the results reported here, $L^2$ projection is used to initialize the solution.}). However, because the computed spectra results should be viewed somewhat skeptically due to their sensitivity, we present results for several different solution resolutions to make it easier to observe broader trends in the maximum real part of the spectra. We note that this sensitivity to initial condition is not due purely to the non-differentiability of the $\min$ function in \eqref{eq:eps}, as it persists even when replacing this with a differentiable ``smooth minimum''. 
}

Figure~\ref{fig:spectra} shows a zoom near the origin of the spectra for degree $N=7$ and a mesh of $K=4$ uniform elements. \gnote{Table~\ref{tab:spectra} reports the computed maximum real part of the spectra for degree $N=3$ and degree $N=7$ approximations at various mesh resolutions. We consistently observe that entropy correction artificial viscosity results in smaller maximum real parts than EC flux differencing (sometimes significantly smaller), though standard weak form DG typically results in the smallest maximum real part of the spectra.\footnote{The exception is the case of $N=3$ and $K=16$ elements, where entropy correction artificial viscosity results in a smaller maximum real part than standard weak form DG. However, for standard weak form DG, the maximum real part is $O(10^{-10})$ and still very small.}}\bnote{We hypothesize that the presence of positive maximum real parts of the spectra for entropy correction artificial viscosity may be due to the fact that the system is linearized with respect to the conservative variables, while the artificial viscosity dissipation is dissipative (i.e., symmetric and positive semi-definite) only with respect to projected entropy variables. 

Finally, we note that when reducing the amplitude of the density wave to $A = 0.5$, the maximum real parts of the linearized spectra for both entropy correction artificial viscosity and standard weak form DG are between $O(10^{-7})$ and $O(10^{-15})$, while for EC flux differencing the maximum real parts are between $O(10^{-2})$ and $O(10^{-3})$.}


\subsection{Convergence of the artificial viscosity coefficients}

Next, we examine the magnitude of the \gnote{entropy correction artificial viscosity} coefficient $\epsilon_k(\bm{u}_h)$ for both smooth and discontinuous solution profiles. We compute artificial viscosity coefficients for the following smooth solution field
\[
(\rho, u_1, u_2, p) =  \LRp{1 + \frac{1}{2} \sin(0.1 + \pi x) \sin(0.2 + \pi y), \frac{1}{2}\sin(0.2 + \pi x) \sin(0.1 + \pi y), 0, \rho^\gamma}
\]
and the following discontinuous solution field
\[
(\rho, u_1, u_2, p) = \begin{cases}
     (1, 0, 0, 1) & \LRb{0.3 x + y} < 0.5\\
     (2, .1, .2, 2^\gamma) & \text{otherwise}.
\end{cases}
\]

To compute the \gnote{entropy correction artificial viscosity} coefficient, we compute the $L^2$ projection onto degree $N$ polynomials of the conservative variables given by these solution fields, then determine the artificial viscosity coefficient $\epsilon_k(\bm{u}_h)$ using \eqref{eq:eps}. The quadrature is chosen based on the discussion in Section~\ref{sec:quadrature} such that the quadrature errors are of the same order of accuracy as the volume entropy residual. Specifically, we use a volume quadrature which is exact for degree $(2N+1)$ polynomials, and the surface quadrature is taken to be an $(N+2)$ point Gauss quadrature, which is sufficient to exactly integrate degree $(2N+2)$ polynomials. 

\begin{figure}
\centering
\subfloat[Smooth solution]{\includegraphics[width=.47\textwidth]{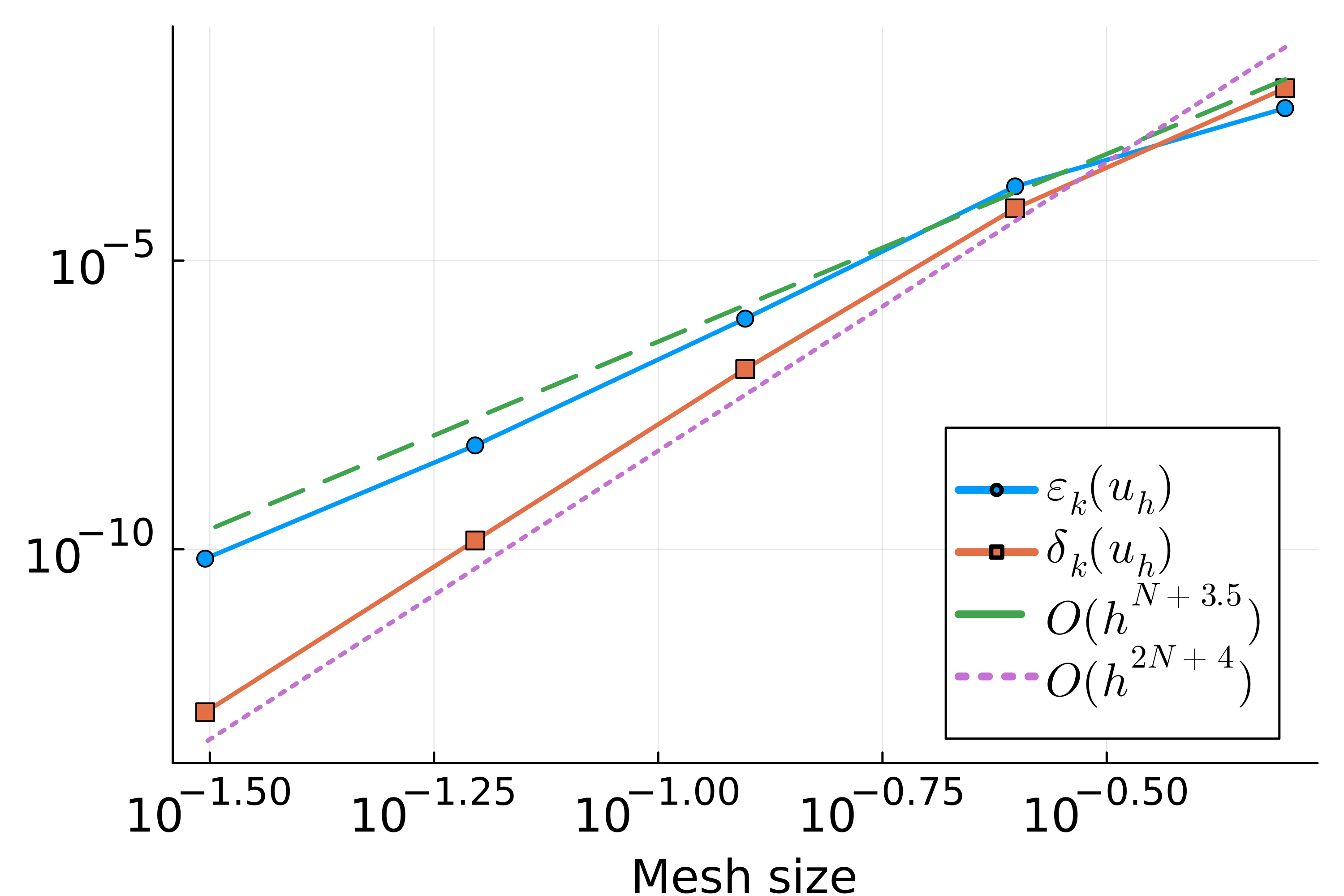}}
\hspace{.1em}
\subfloat[Discontinuous solution]{\includegraphics[width=.47\textwidth]{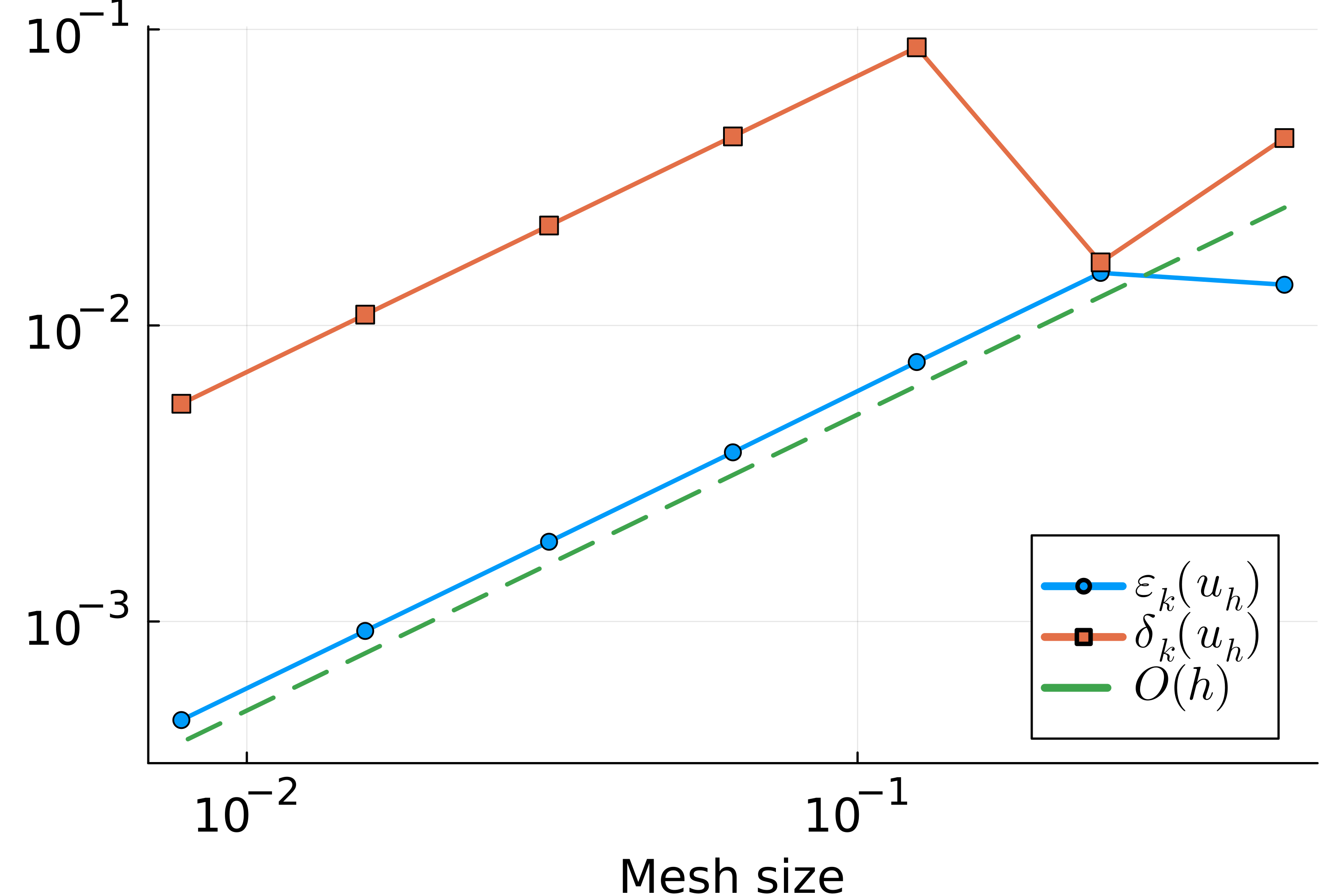}}
\caption{Convergence of the volume entropy residual $\delta_k(\bm{u}_h)$ \eqref{eq:entropy_ineq_error} and the \gnote{entropy correction artificial viscosity} coefficient $\epsilon_k(\bm{u}_h)$ \eqref{eq:eps} for fixed 2D smooth and discontinuous solution fields with $N=3$.}
\label{fig:conv_AV_coeff}
\end{figure} 

Figure~\ref{fig:conv_AV_coeff} shows the maximum values of $\epsilon_k(\bm{u}_h)$ and $\delta_k(\bm{u}_h)$ (the volume entropy residual given by \eqref{eq:entropy_ineq_error}) over degree $N=3$ uniform triangular meshes. We observe that, for a smooth solution field, the volume entropy residual converges at the expected rate of $O(h^{2N+2+d})$, while the \gnote{entropy correction artificial viscosity} coefficient converges at a rate of $O(h^{N+1.5 + d})$. For a discontinuous solution field, we observe $O(h)$ convergence for both quantities once the mesh is sufficiently refined. 

We note that, in practice, we use quadratures with lower degrees of exactness. We observe that, when taking the volume and surface quadratures to be exact for degree $2N$ and $2N+1$ polynomials respectively, the volume entropy residual and \gnote{entropy correction artificial viscosity} coefficient appear \bnote{to} achieve rates of $O(h^{2N+1+d})$ and $O(h^{N+0.5+d})$ for smooth solution fields, losing one order of convergence. 


\subsection{Modified Sod shock tube}

Next, we consider the modified Sod shock tube problem \cite{toro2013riemann} on the domain on $[0,1]$
\begin{equation}
(\rho, u, p) = \begin{cases}
(1, .75, 1) & x < 0.3\\
(.125, 0, .1) & \text{otherwise}.
\end{cases}
\label{eq:modSod}
\end{equation}

\begin{figure}
\centering
\subfloat[Nodal DG with AV]{\includegraphics[width=.45\textwidth, trim={9em 8em 9em 5em}, clip]{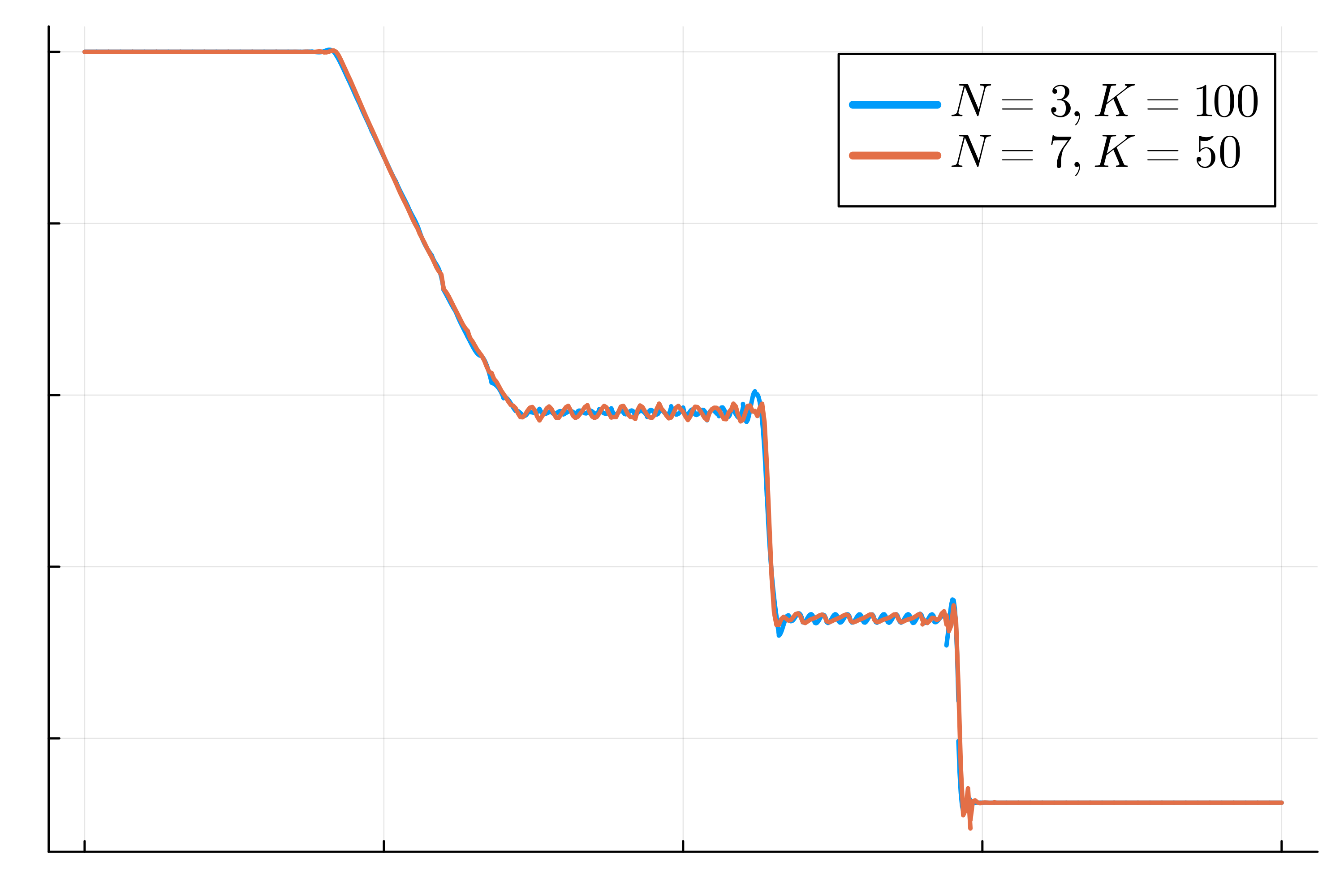}}
\hspace{.1em}
\subfloat[Nodal DG with flux differencing]{\includegraphics[width=.45\textwidth, trim={9em 8em 9em 5em}, clip]{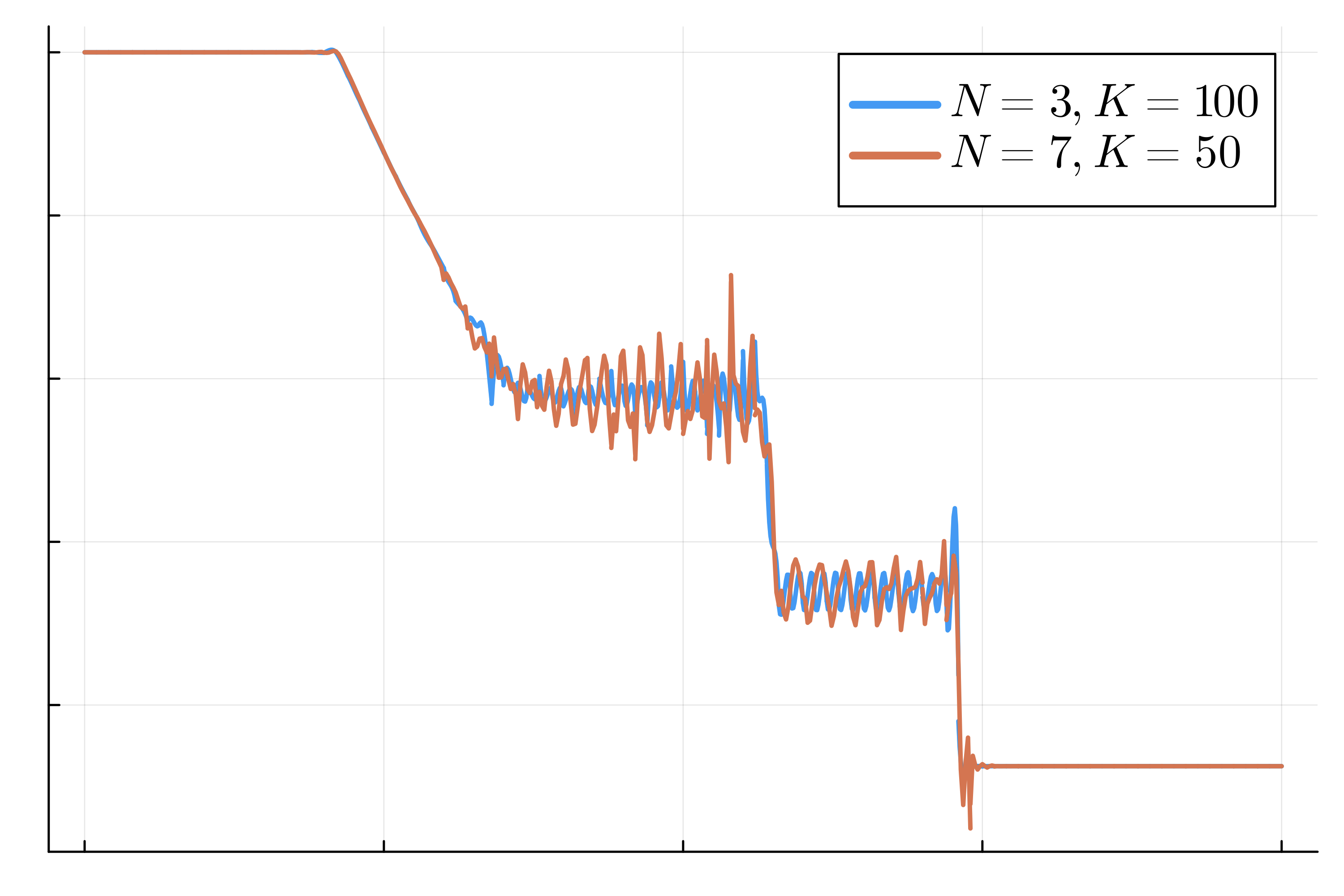}}\\
\subfloat[Modal DG with AV]{\includegraphics[width=.45\textwidth, trim={9em 8em 9em 5em}, clip]{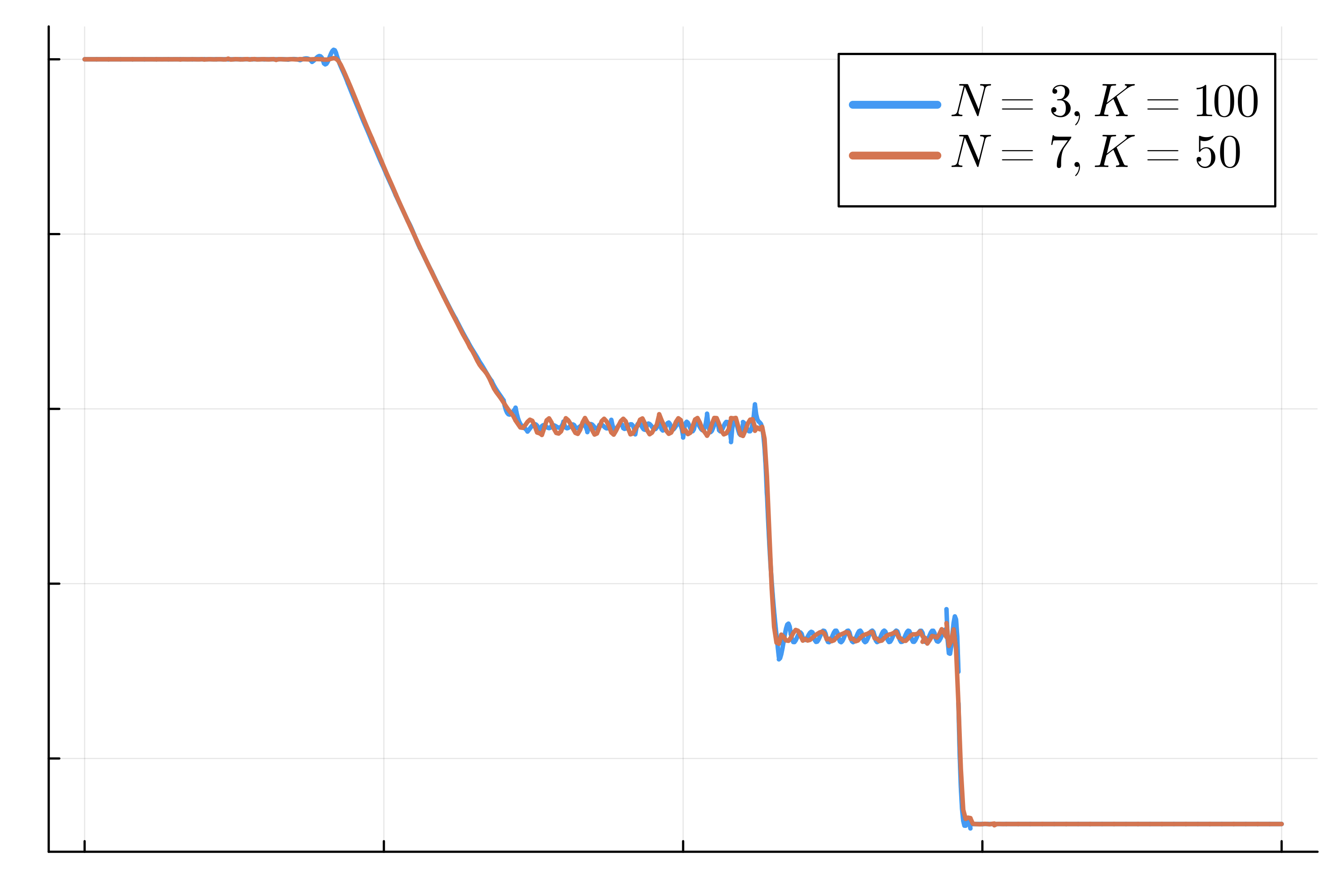}}
\hspace{.1em}
\subfloat[Modal DG with flux differencing]{\includegraphics[width=.45\textwidth, trim={9em 8em 9em 5em}, clip]{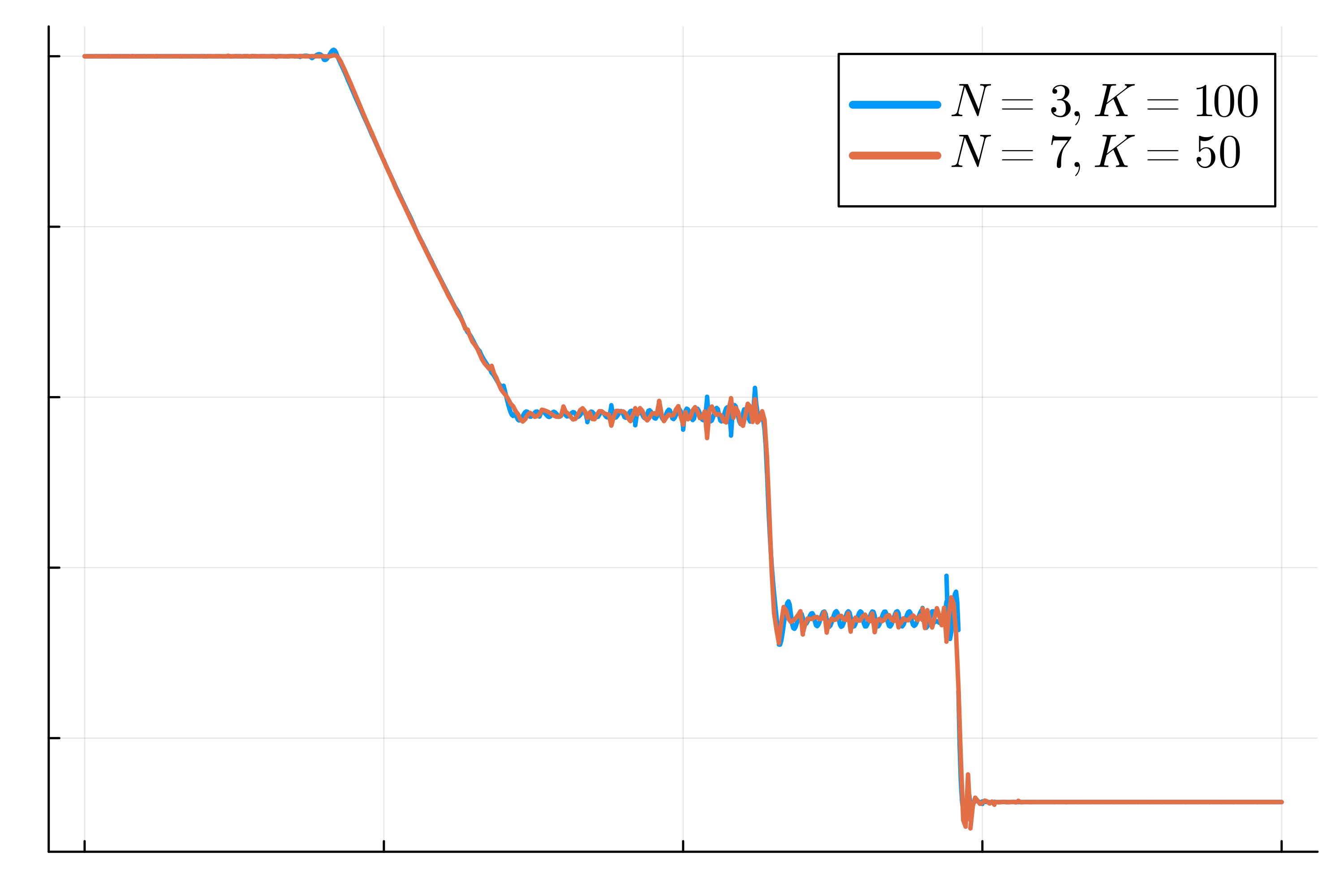}}
\caption{Comparison of \rnote{density for} entropy stable DG methods constructed using flux differencing and entropy correction artificial viscosity. } %
\label{fig:modSod}
\end{figure}

Figure~\ref{fig:modSod} shows computed solutions at time $T = 0.2$ using both nodal \cite{chen2017entropy} and modal \cite{chan2018discretely, chan2019skew} flux differencing entropy stable DG methods, as well as both nodal and modal DG methods with entropy stability enforced via artificial viscosity. We note that the nodal and modal DG \gnote{with entropy correction artificial viscosity} are similar except for the presence of a small solution artifact near the sonic point for the nodal DG method. This solution artifact decreases in magnitude under either mesh or degree refinement. 

We observe that, while all methods produce spurious oscillations due to under-resolution of the shock, the oscillations present for DG with \gnote{entropy correction artificial viscosity} are significantly smaller compared with the oscillations present for the flux differencing entropy stable nodal DG method. Moreover, the oscillations for the flux differencing entropy stable nodal DG methods increase in magnitude as the degree $N$ increases, while the oscillations for flux differencing modal entropy stable DG method and DG with \gnote{entropy correction artificial viscosity} remain about the same magnitude for both degrees $N=3$ and $N=7$. Using a contact-preserving entropy stable flux such as HLLC \cite{batten1997choice} or a Roe-type matrix dissipation \cite{winters2017uniquely, waruszewski2022entropy} reduces but does not eliminate such oscillations. Additional numerical results using contact-preserving interface fluxes are provided in \ref{sec:hllc}. 

We note that there exist several additional techniques to suppress the spurious oscillations present in flux differencing entropy stable nodal DG methods. Earlier papers on high order entropy stable DG methods utilized a comparison principle, blending a standard DG method with an entropy stable DG method \cite{carpenter2014entropy}. Shock capturing \cite{hennemann2021provably}, bounds-preserving limiting \cite{lin2023positivity}, and heuristic artificial viscosity methods \cite{mateo2022entropy} have also been shown to be effective.

\subsection{Shu-Osher problem}

We consider the Shu-Osher sine-shock interaction problem \cite{shu2009high} next, which is posed on the domain $[-5, 5]$ with initial condition
\[
(\rho, u, p) = \begin{cases}
(3.857143, 2.629369, 10.3333) & x < -4\\
(1 + .2\sin(5x), 0, 1) & x \geq -4.
\end{cases} 
\]

\begin{figure}[h]
\centering
\subfloat[Nodal DG with flux differencing ]{\includegraphics[width=.45\textwidth, trim={9em 8em 9em 5em}, clip]{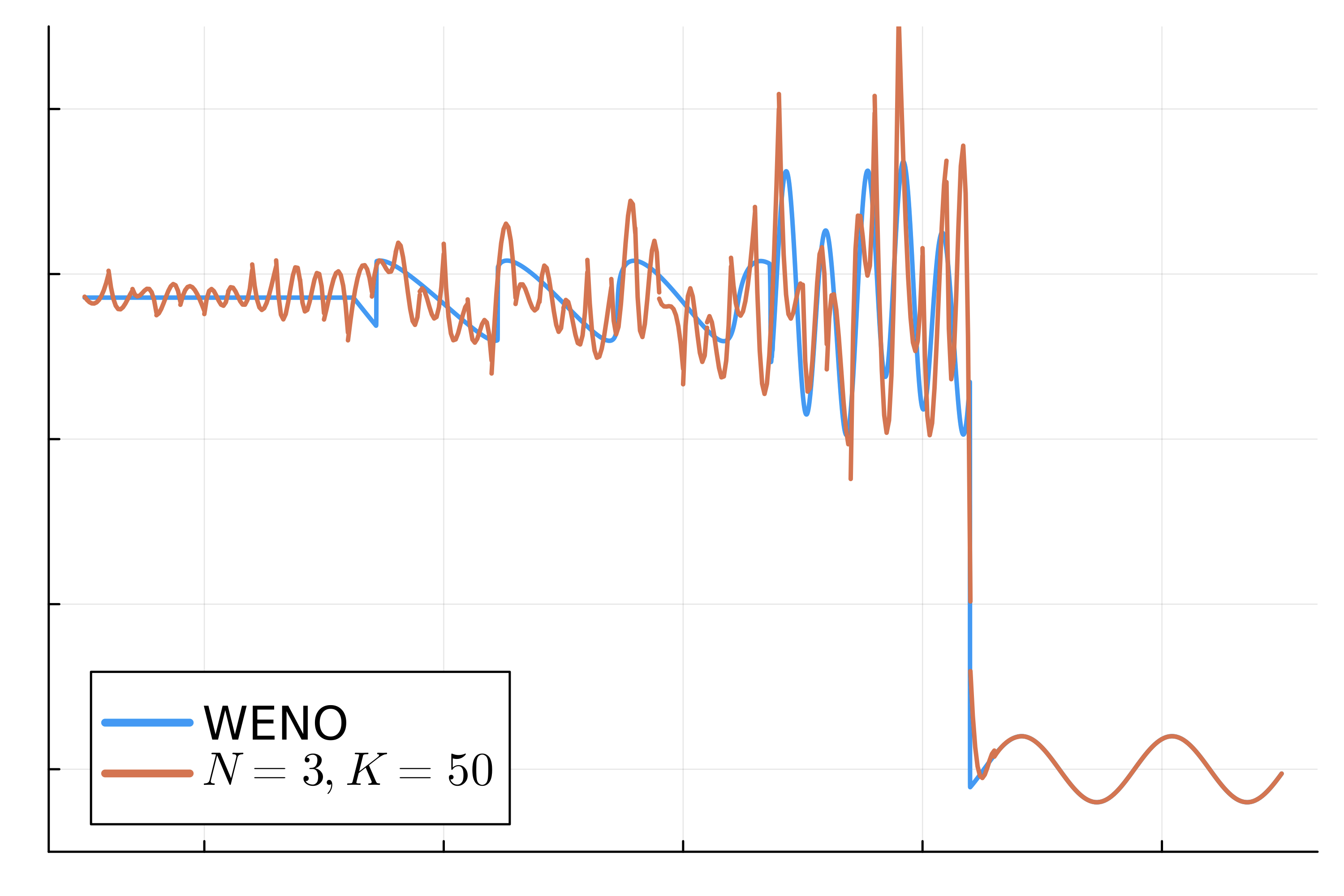}}
\hspace{.1em}
\subfloat[Nodal DG with AV]{\includegraphics[width=.45\textwidth, trim={9em 8em 9em 5em}, clip]{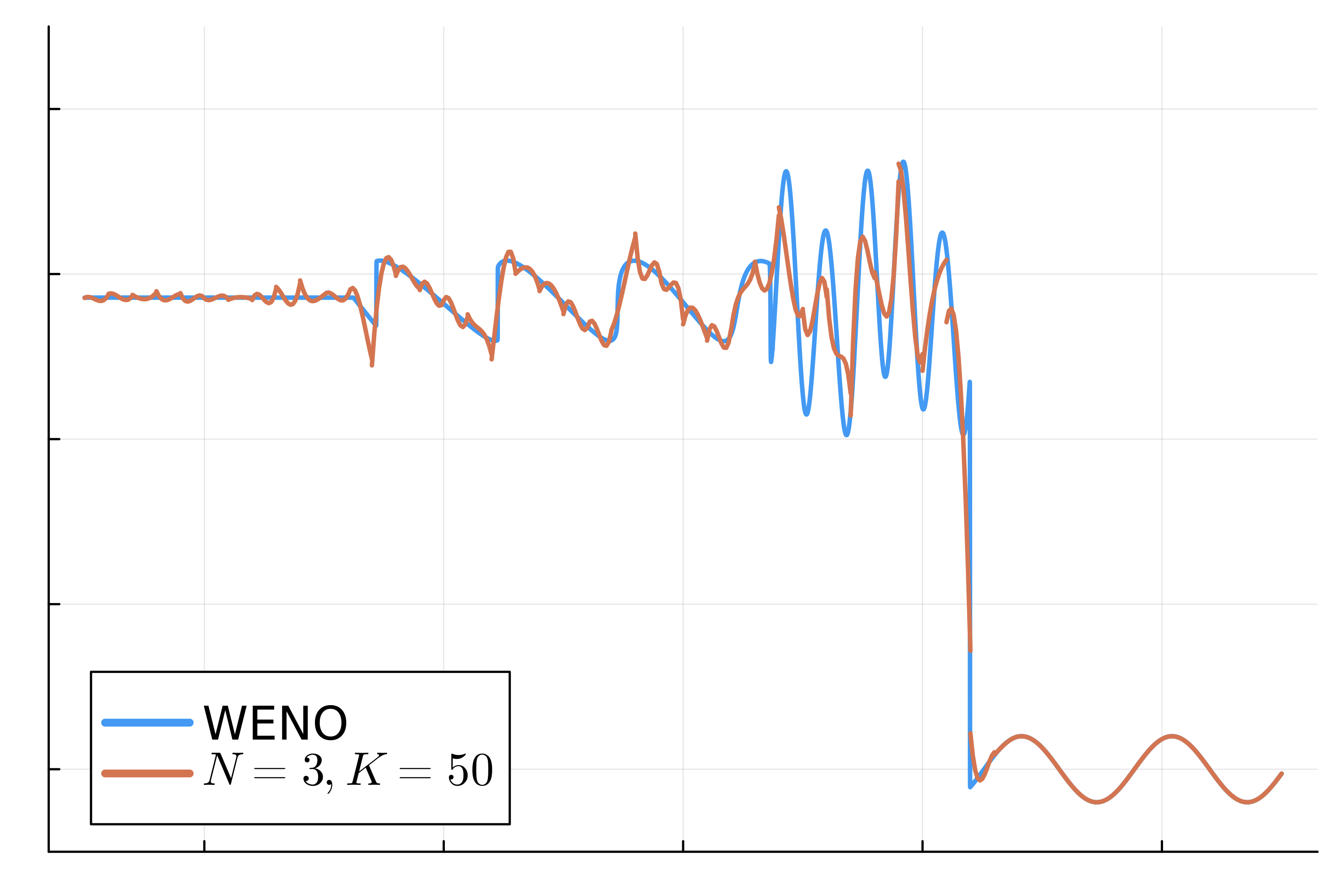}}\\
\subfloat[Modal DG with flux differencing ]{\includegraphics[width=.45\textwidth, trim={9em 8em 9em 5em}, clip]{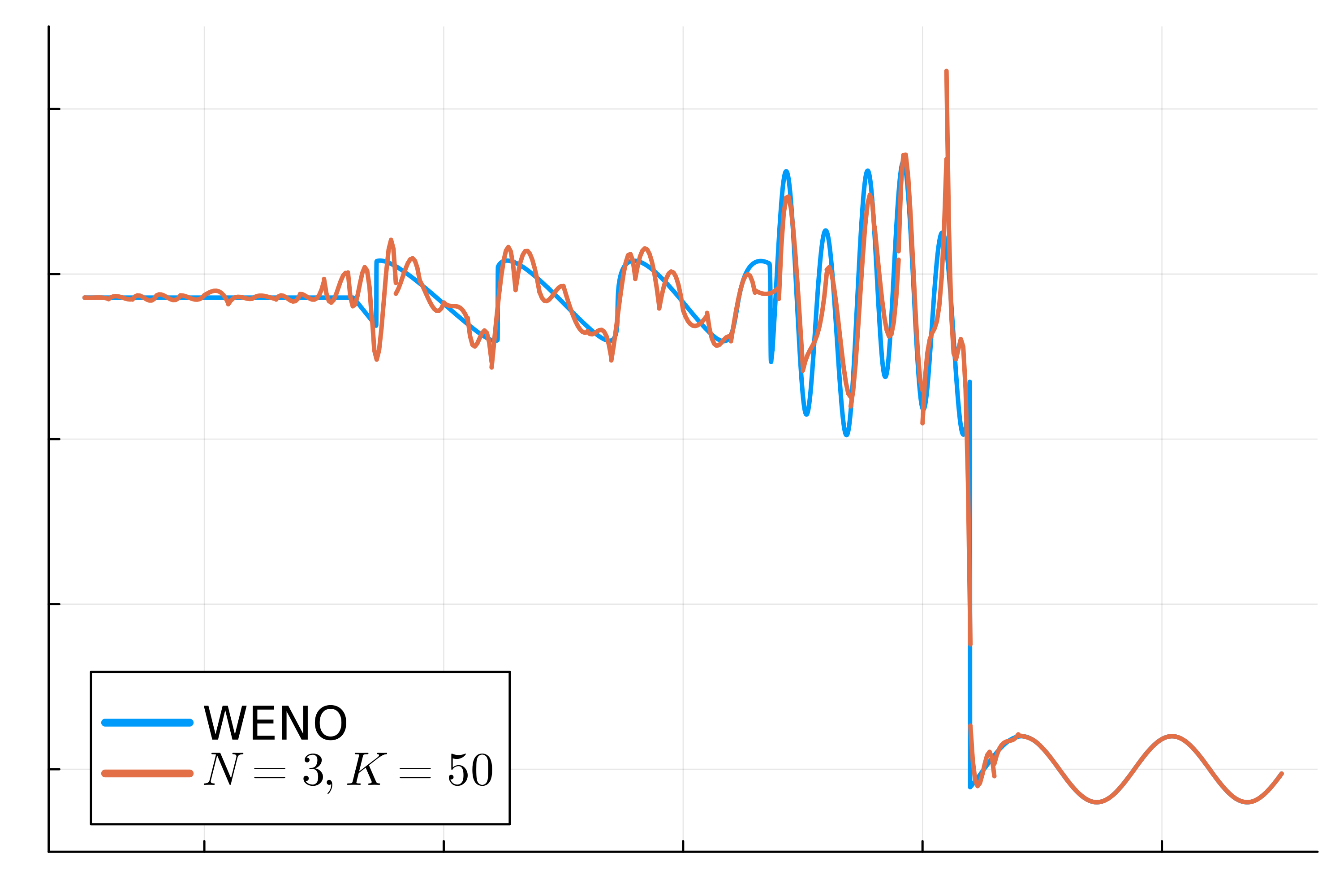}}
\hspace{.1em}
\subfloat[Modal DG with AV]{\includegraphics[width=.45\textwidth, trim={9em 8em 9em 5em}, clip]{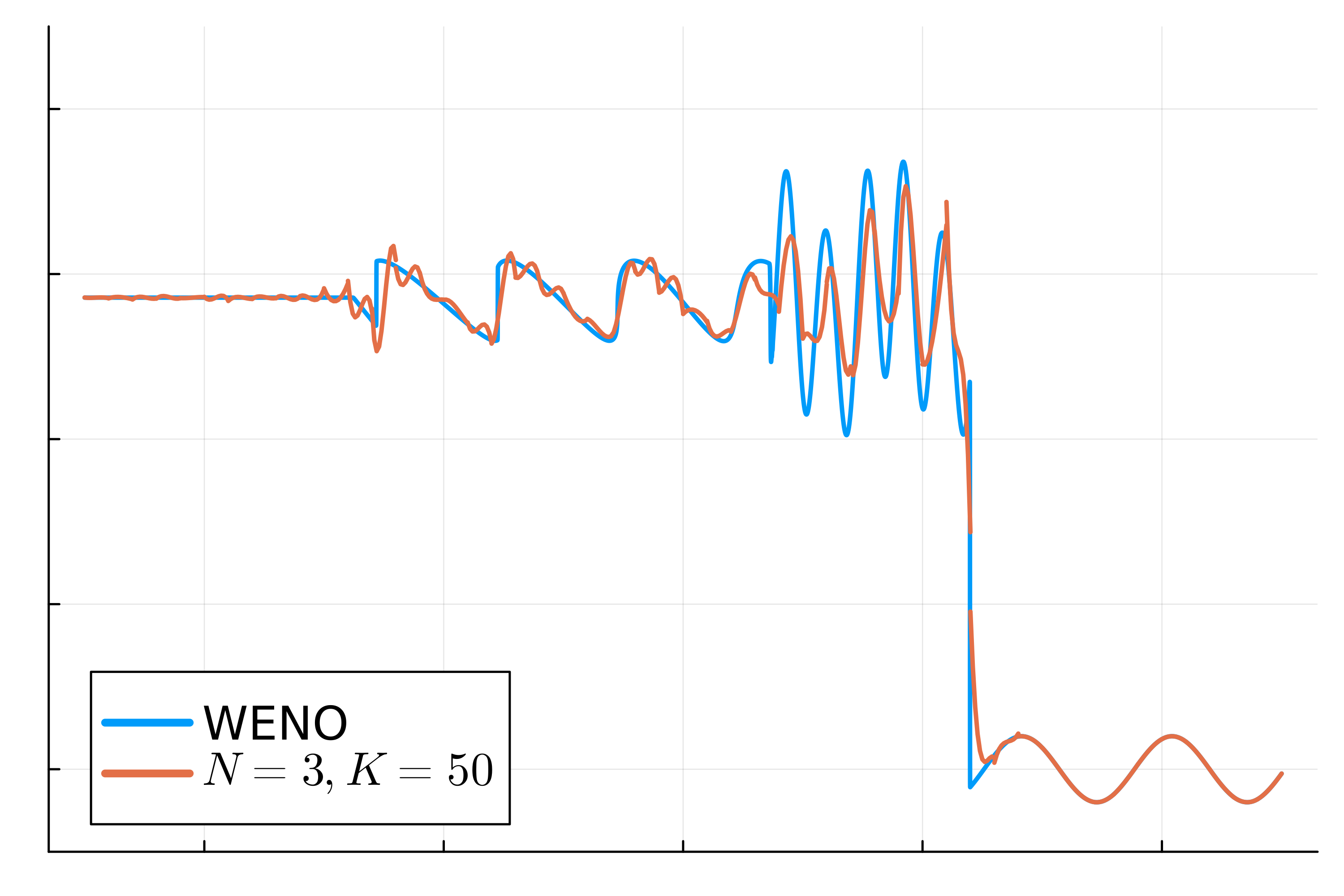}}
\caption{Entropy stable DG solutions \rnote{(density)} of the Shu-Osher problem with $N=3$, $50$ elements.}
\label{fig:shu_osher_N3_K50}
\end{figure}

\begin{figure}
\centering
\subfloat[Nodal DG with flux differencing ]{\includegraphics[width=.45\textwidth, trim={9em 8em 9em 5em}, clip]{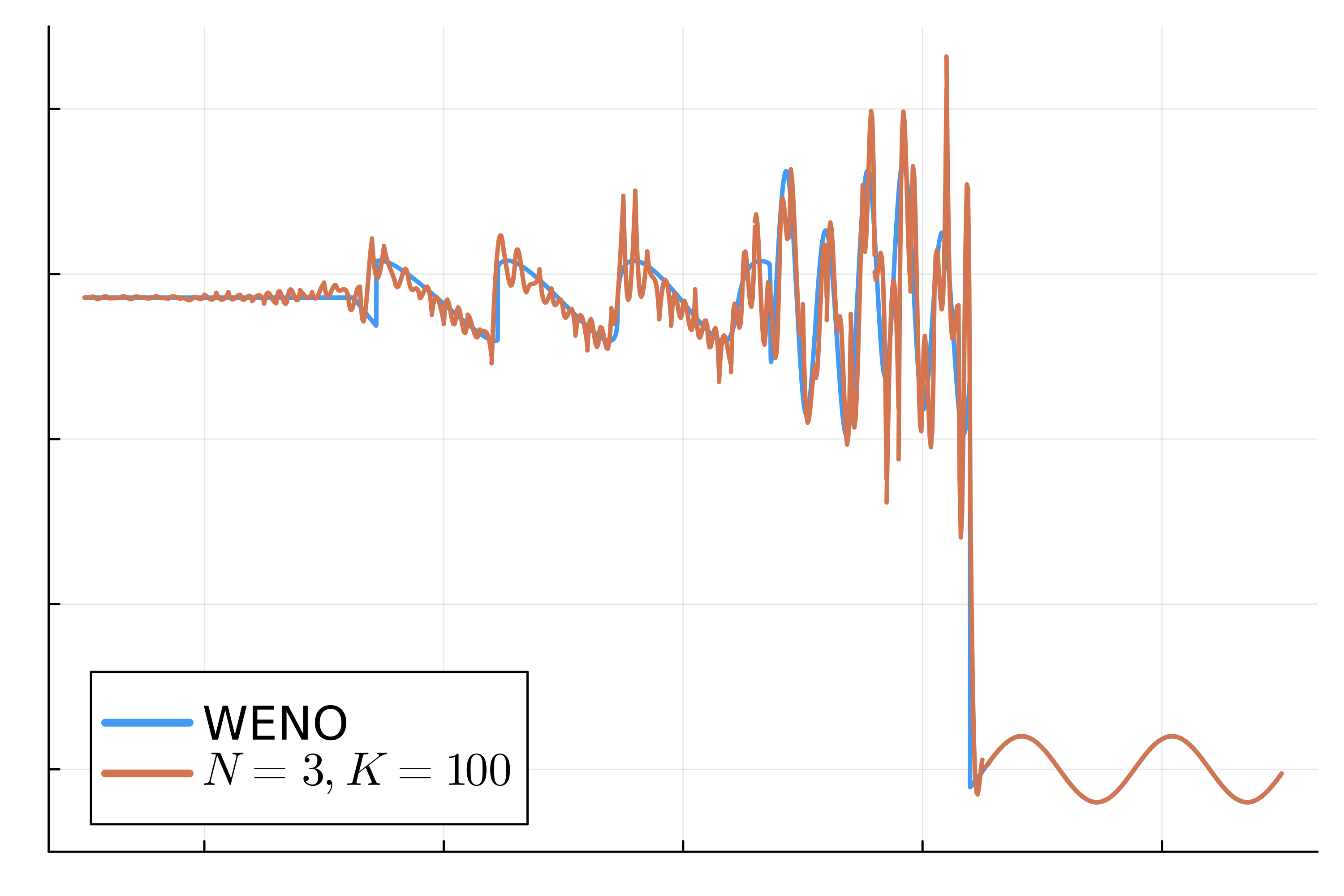}}
\hspace{.1em}
\subfloat[Nodal DG with AV]{\includegraphics[width=.45\textwidth, trim={9em 8em 9em 5em}, clip]{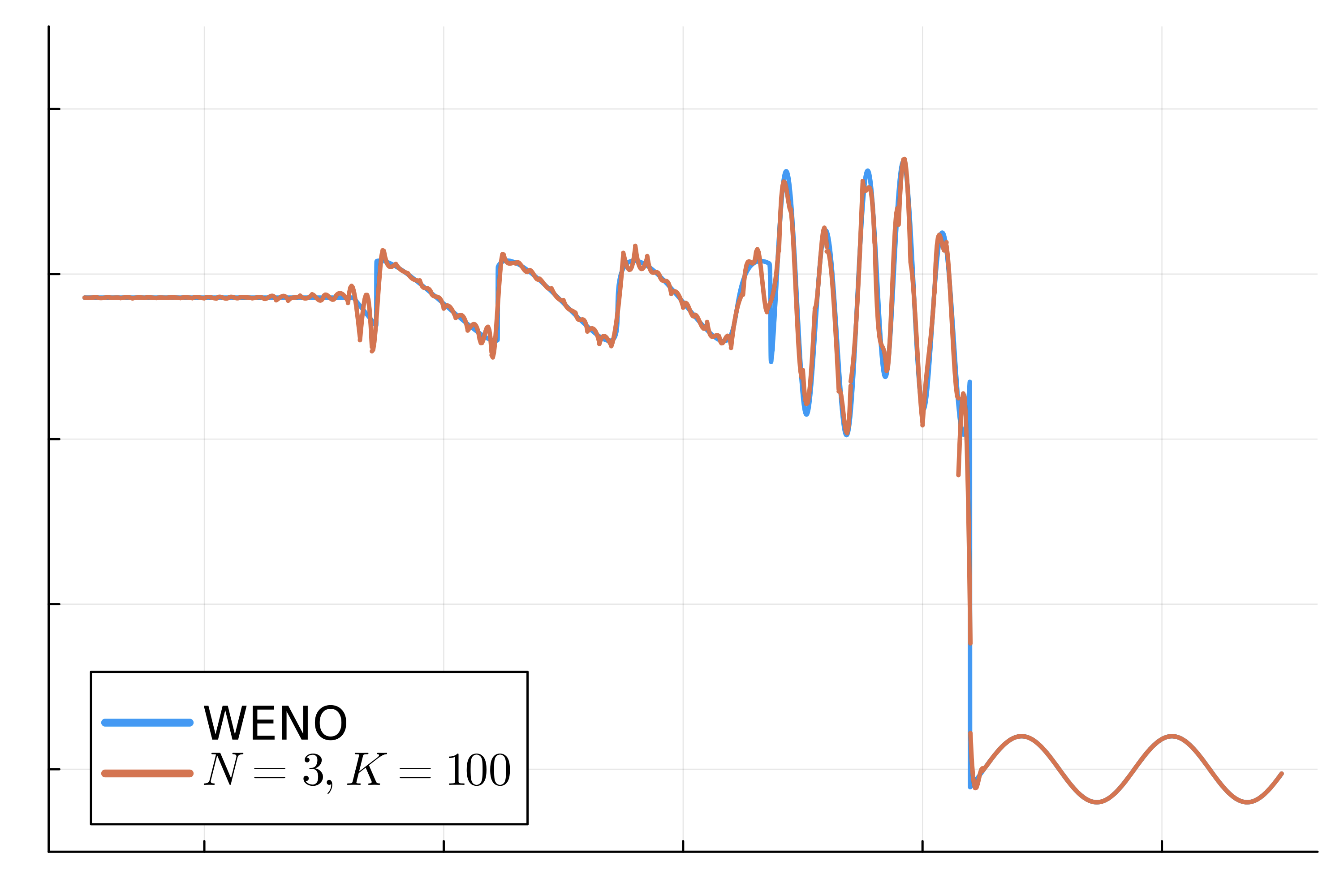}}\\
\subfloat[Modal DG with flux differencing ]{\includegraphics[width=.45\textwidth, trim={9em 8em 9em 5em}, clip]{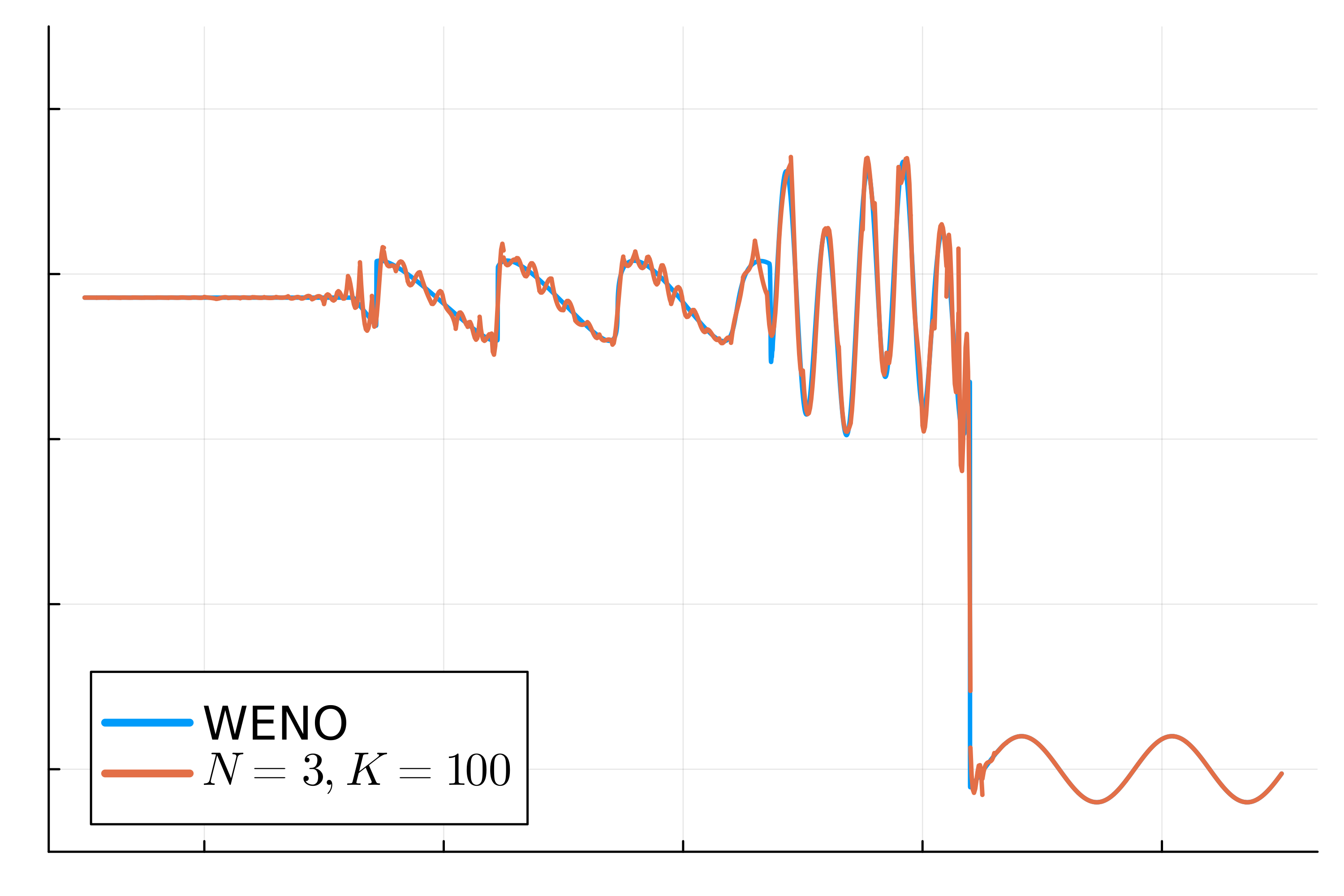}}
\hspace{.1em}
\subfloat[Modal DG with AV]{\includegraphics[width=.45\textwidth, trim={9em 8em 9em 5em}, clip]{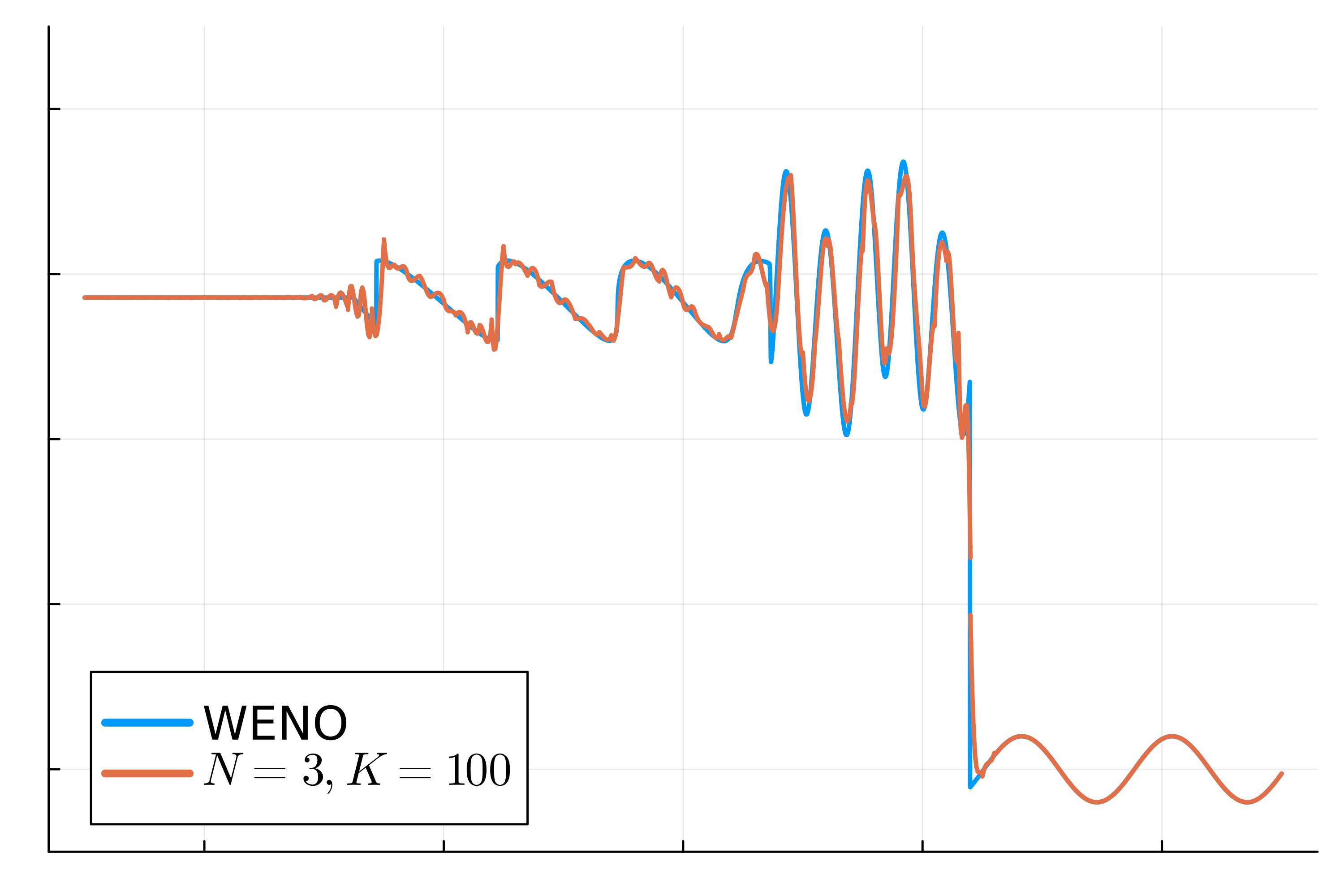}}
\caption{Entropy stable DG solutions \rnote{(density)} of the Shu-Osher problem  with $N=3$, $100$ elements.}
\label{fig:shu_osher_N3_K100}
\end{figure}

\begin{figure}
\centering
\subfloat[Modal DG with flux differencing ]{\includegraphics[width=.32\textwidth, trim={9em 8em 60em 27.5em}, clip]{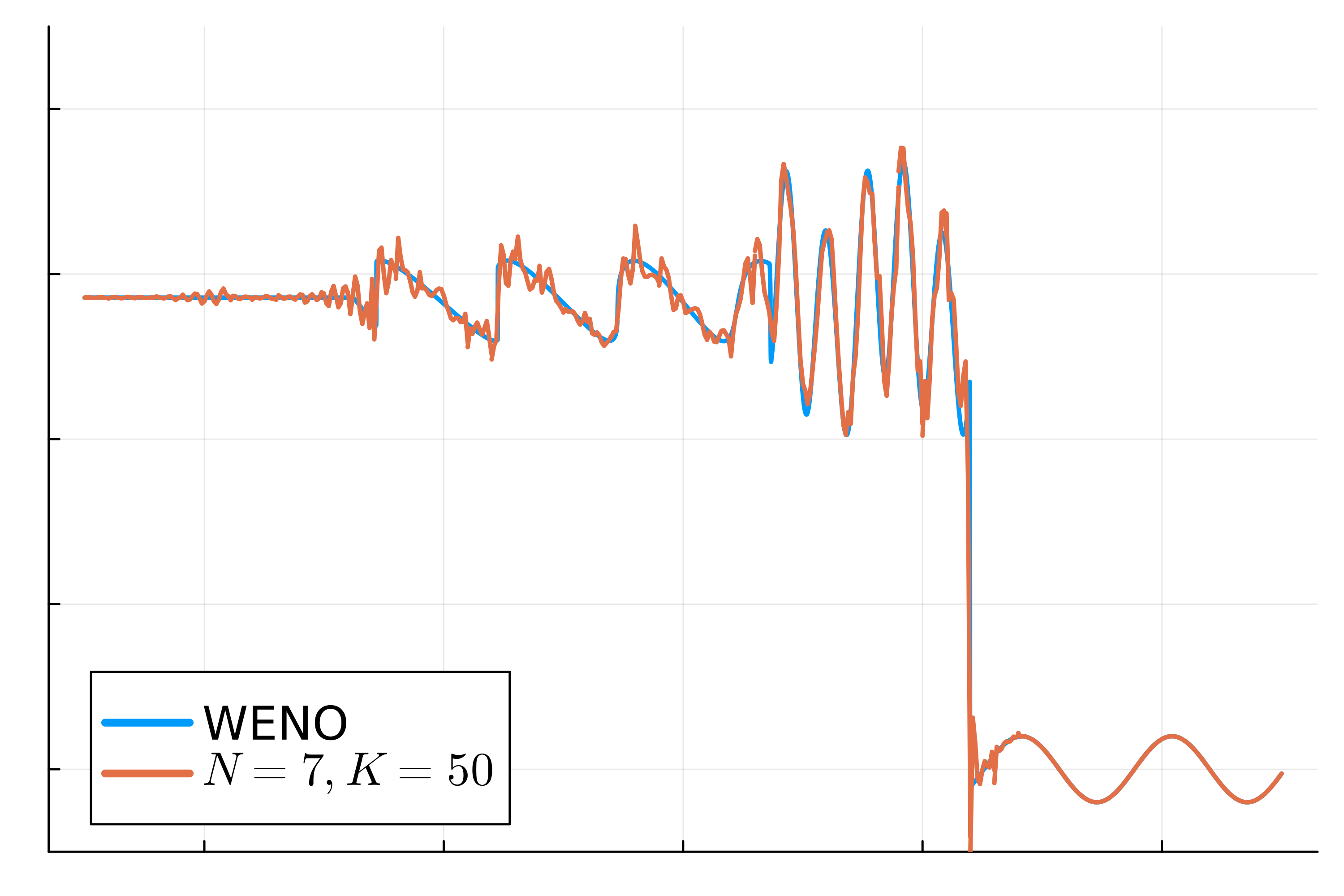}}
\hspace{.1em}
\subfloat[Nodal DG with AV]{\includegraphics[width=.32\textwidth, trim={9em 8em 60em 27.5em}, clip]{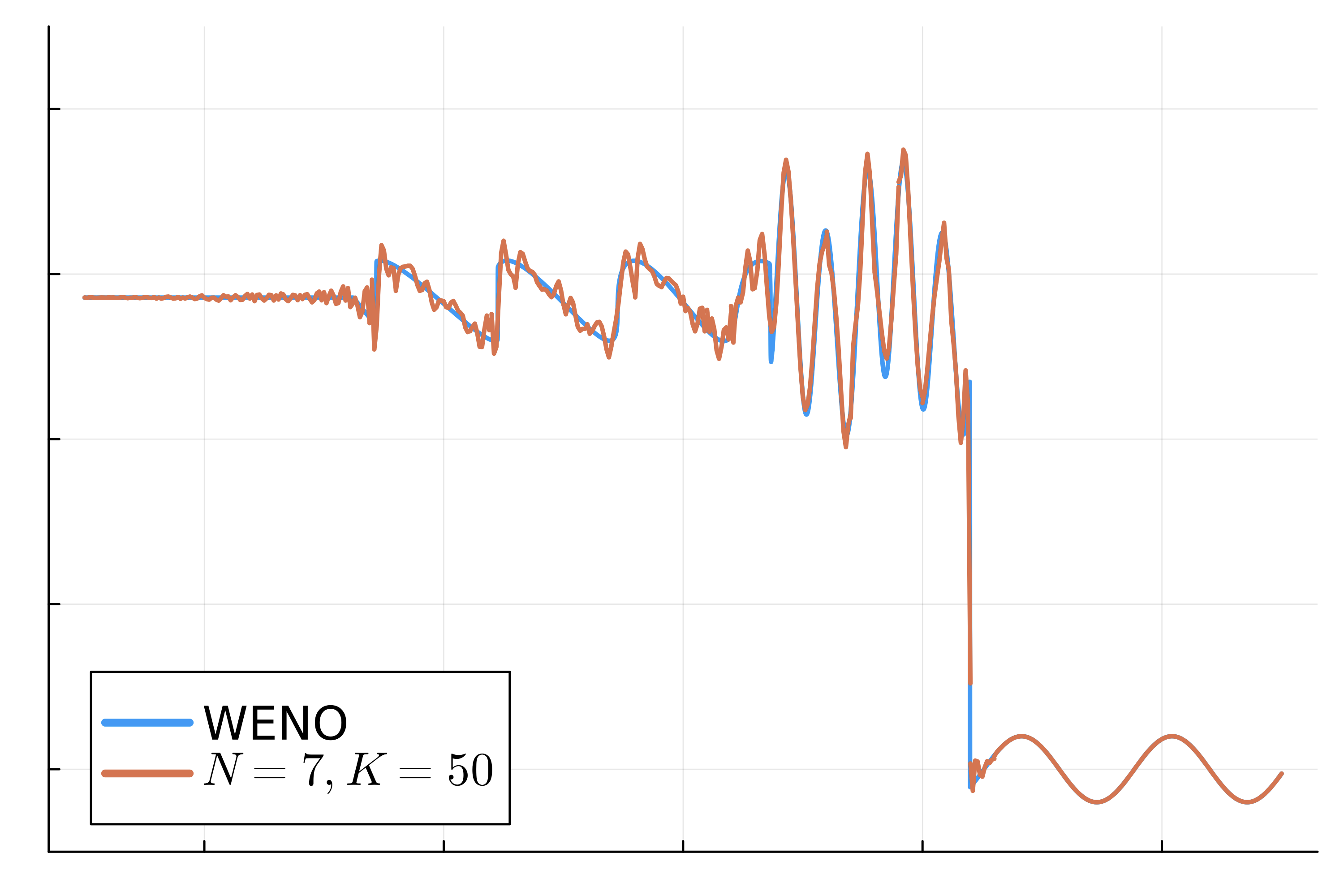}}
\hspace{.1em}
\subfloat[Modal DG with AV]{\includegraphics[width=.32\textwidth, trim={9em 8em 60em 27.5em}, clip]{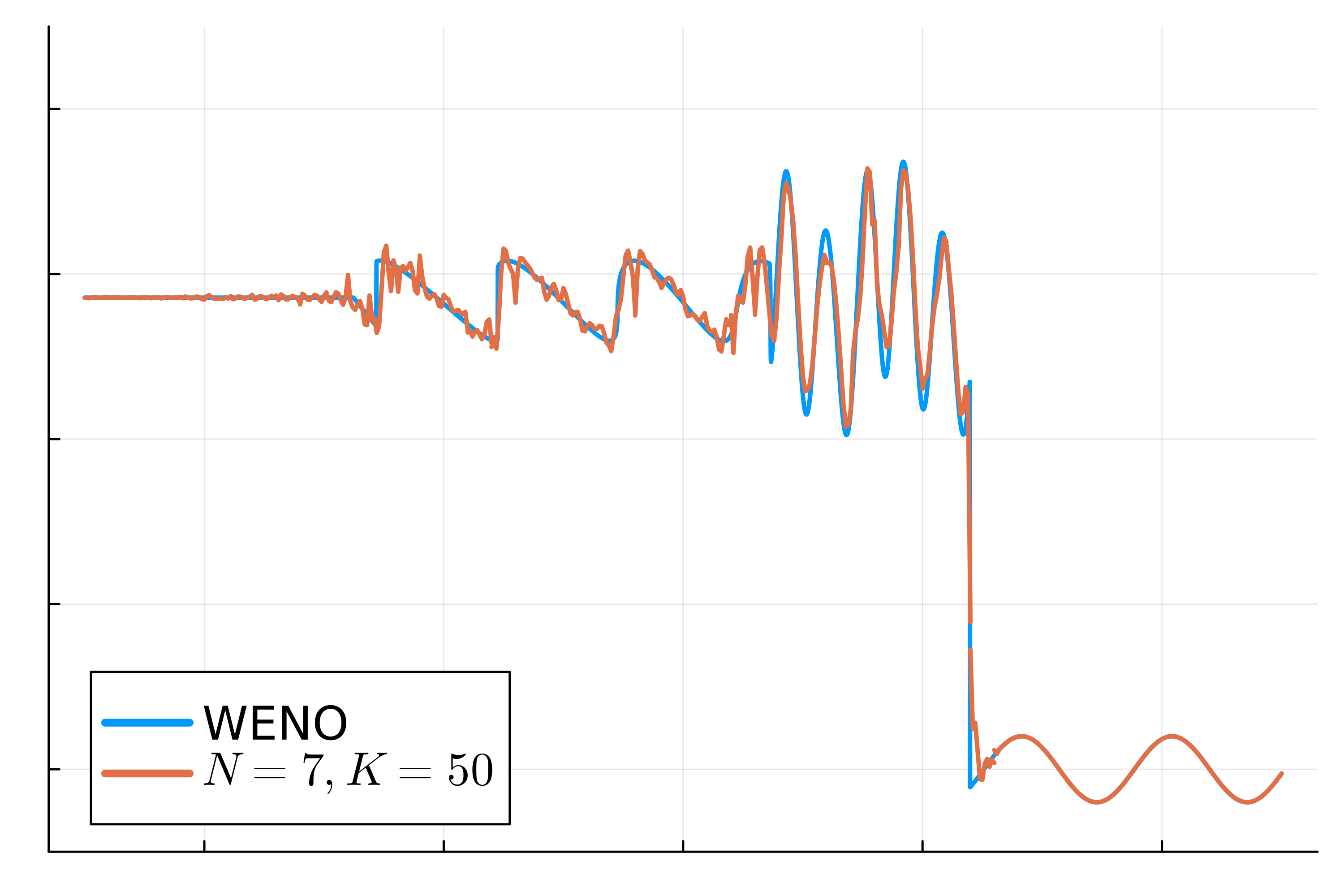}}
\caption{Entropy stable DG solutions \rnote{(density)} of the Shu-Osher problem  with $N=7$, $50$ elements. Results for nodal DG with flux differencing are not included as the simulation did not run to completion.}
\label{fig:shu_osher_N7_K50}
\end{figure}

We compare the computed DG solutions to a reference solution computed using 5th order WENO with 25000 cells. Figures~\ref{fig:shu_osher_N3_K50}, \ref{fig:shu_osher_N3_K100}, and \ref{fig:shu_osher_N7_K50} show various entropy stable nodal and modal DG solutions for an under-resolved case (degree $N=3$ with 50 elements), a more resolved moderate order case ($N=3$ with $100$ elements) and a more resolved high order case ($N=7$ with $50$ elements). Note that the $N=7$, $50$ element case results in the same total number of unknowns as the $N=3$, $100$ element case. 

We observe that nodal flux differencing entropy stable schemes tend to produce solutions with significant spurious oscillations. These oscillations are significantly reduced for either modal flux differencing entropy stable DG schemes or DG with \gnote{entropy correction artificial viscosity}. We also observe that DG results with \gnote{entropy correction artificial viscosity} are slightly more diffusive than the modal flux differencing entropy stable DG results, as mentioned in Remark~\ref{remark:dissipation}.

We also compare the number of timesteps taken by the adaptive SSPRK43 method for the specified tolerances. The nodal flux differencing entropy stable DG method required the fewest timesteps (15206). However, the method only remained stable for $N=3$, and results in large magnitude spurious oscillations. The modal flux differencing entropy stable DG scheme contains significantly fewer spurious oscillations, and runs stably for both $N=3$ and $N=7$. However, the number of time-steps required is roughly double that of the nodal entropy stable DG scheme (35480 for $N=3$ and $32793$ for $N=7$). Both the nodal and modal artificial viscosity schemes behaved similarly with respect to the number of timesteps; for $N=3$, roughly 20k time-steps were required (21781 for nodal, 19741 for modal), while for $N=7$, roughly 30k time-steps were required (30963 for nodal, 33119 for modal). We note that these results depend heavily on the tolerance and the choice of time-stepper; future work will analyze the maximum stable time-step restriction more rigorously. 


%

\subsection{A 2D Riemann problem}
\label{sec:riemann}
\begin{figure}
\centering 
\subfloat[Density]{\includegraphics[height=0.34\textheight, trim={31.5em 7em 31.5em 7em}, clip]{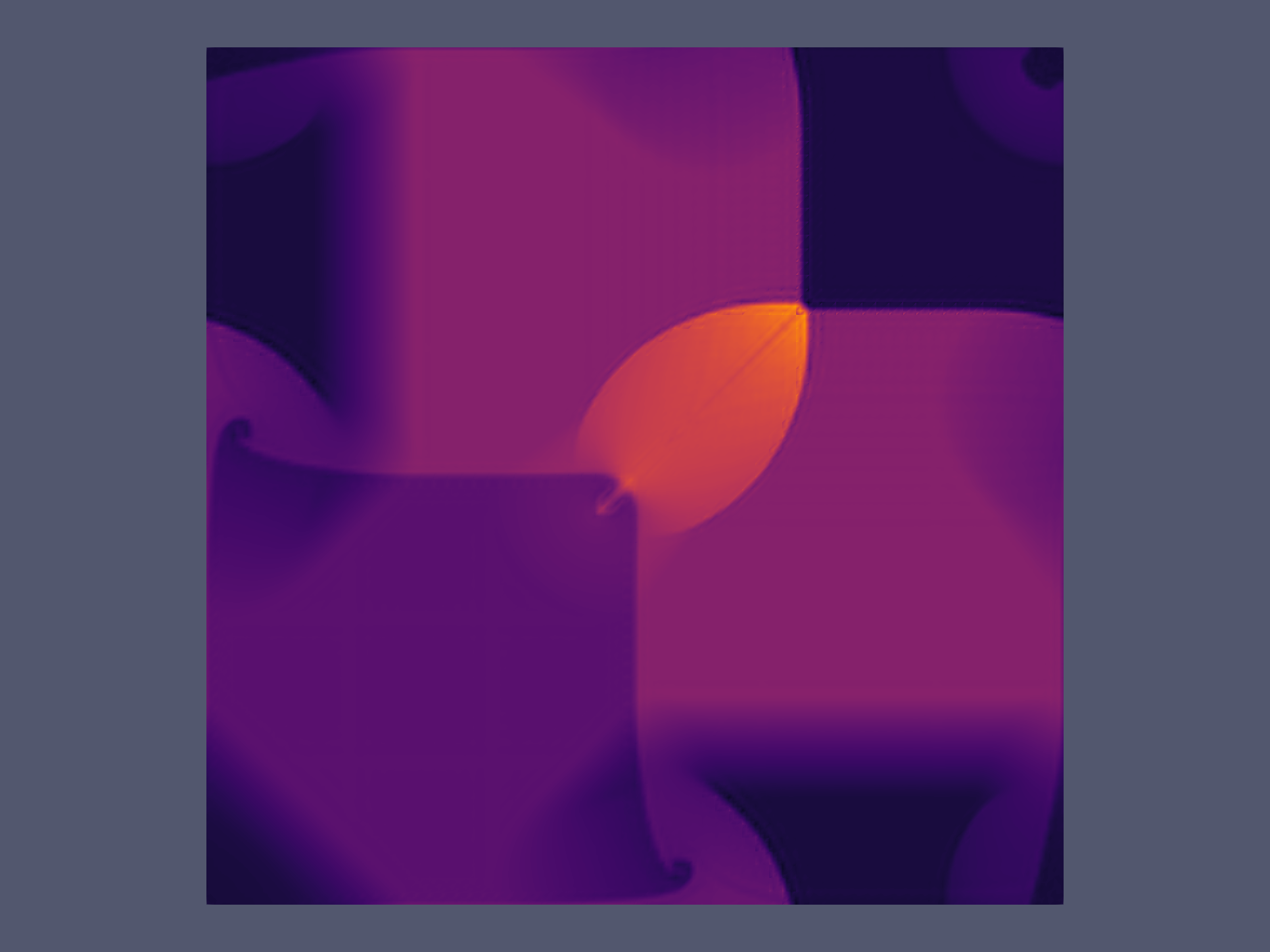}}
\hspace{.25em}
\subfloat[Viscosity coefficient $\epsilon_k(\bm{u}_h)$]{\includegraphics[height=0.34\textheight, trim={31.5em 7em 16em 7em}, clip]{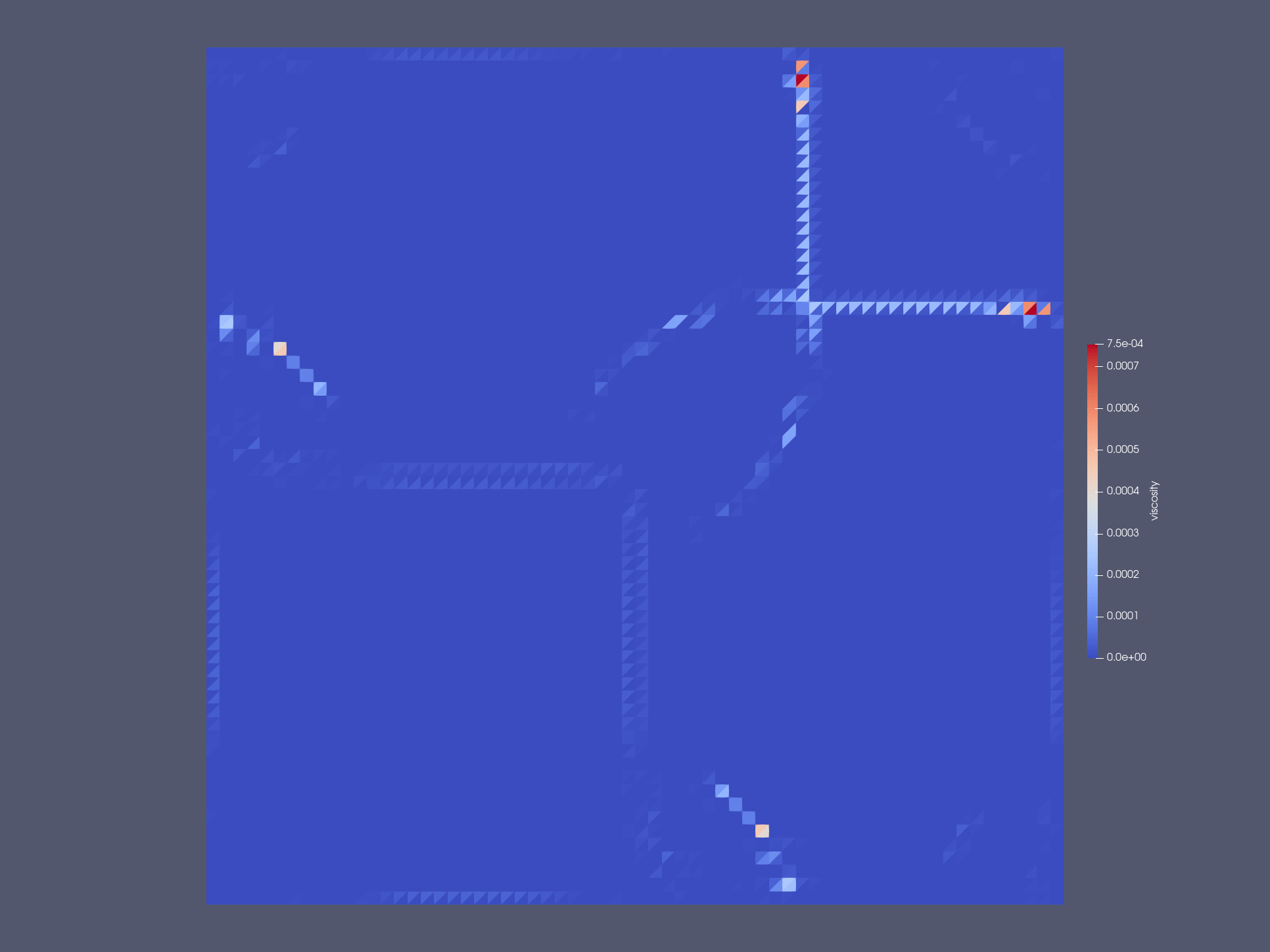}}
\caption{Solution of the 2D Riemann problem with degree $N=3$ on a $64\times 64 \times 2$ triangular mesh.}
\label{fig:riemann_2D}
\end{figure}

Next we consider a periodic version \cite{chan2018discretely} of a 2D Riemann problem from \cite{kurganov2002solution}. 
The problem is posed on a periodic domain $[-1, 1]^2$ with initial condition
\[
(\rho, u_1, u_2, p) = \begin{cases}
(0.5313, 0, 0, 0.4) & x > 0, y > 0\\
(1, 0.7276, 0, 1) & x < 0, y > 0\\
(0.8, 0, 0, 1) & x < 0, y < 0\\
(1, 0, 0.7276, 1) & x > 0, y < 0.
\end{cases}
\]
Figure~\ref{fig:riemann_2D} shows the density at final time $T=0.25$, which we note looks very similar to the density computed by a modal flux differencing entropy stable DG method of the same degree and mesh resolution \cite{chan2018discretely}. Figure~\ref{fig:riemann_2D} also shows a visualization of the viscosity coefficient $\epsilon_k(\bm{u}_h)$ given by \eqref{eq:eps}. While the \gnote{entropy correction artificial viscosity} coefficient appears to serve as an accurate troubled cell indicator, we emphasize that it was constructed only to replicate the entropy inequality satisfied by flux differencing entropy stable DG methods. 

\subsection{Long-time Kelvin-Helmholtz instability}

\begin{figure}
\centering 
\subfloat[$N=3$, $64\times 64\times 2$ mesh]{\includegraphics[height=0.35\textheight, trim={31.5em 7em 31.5em 7em}, clip]{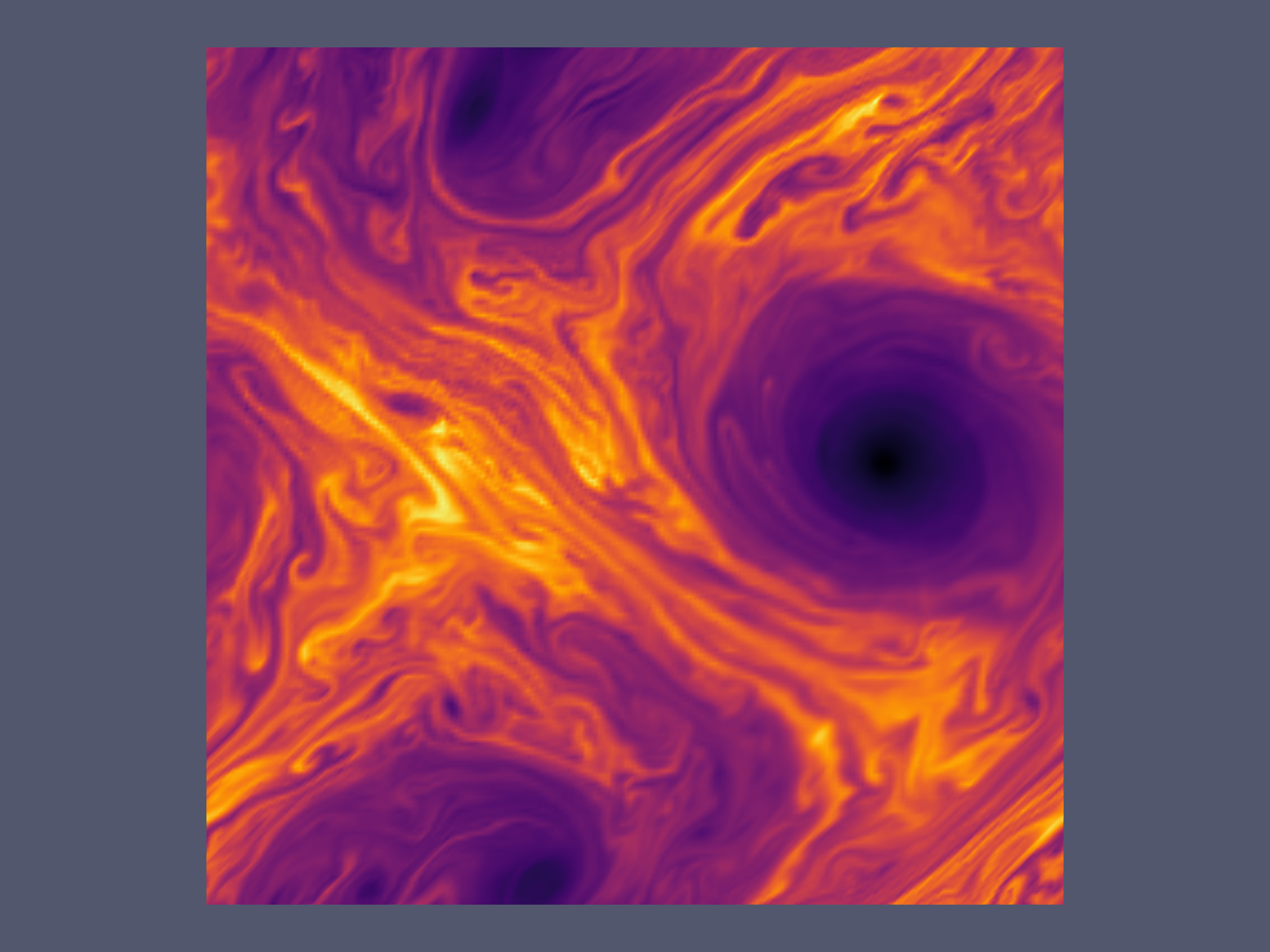}}
\hspace{.5em}
\subfloat[$N=7$, $32\times 32\times 2$ mesh]{\includegraphics[height=0.35\textheight, trim={31.5em 7em 31.5em 7em}, clip]{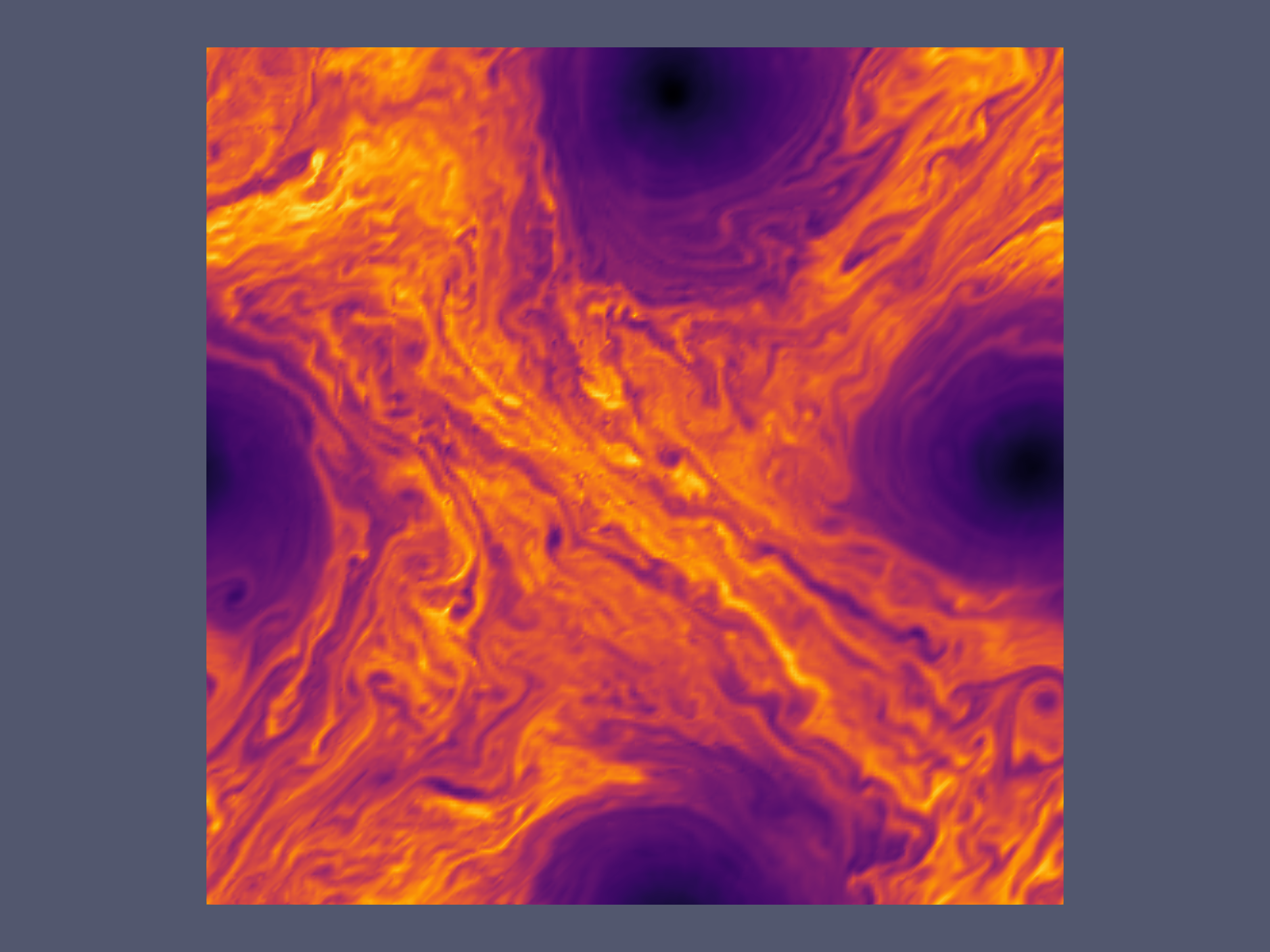}}
\caption{Solution \rnote{(density)} of the 2D long-time Kelvin-Helmholtz instability using DG with \gnote{entropy correction artificial viscosity} on triangular meshes.}
\label{fig:KHI}
\end{figure}

We conclude with a long-time Kelvin-Helmholtz instability introduced in \cite{chan2022entropyprojection}. This example is run until final time $T=25$ in order to develop pressure shocklets and small scale features resembling turbulence. The initial condition is modified to add a small $O(1/100)$ non-symmetric scaling to the velocity in order to break symmetry. We note that, due to the lack of a viscous limit for the compressible Euler equations, numerical solutions to the Kelvin-Helmholtz instability depend strongly on discretization parameters \cite{fjordholm2017construction} and do not converge to a unique solution under degree or mesh refinement. Thus, this test case should only be considered a test of robustness and a rough qualitative measure of how numerically dissipative a scheme is. 

Figure~\ref{fig:KHI} shows the solutions for both degree $N=3$ and degree $N=7$ solutions with roughly the same number of global degrees of freedom. As expected, the higher order approximation is able to resolve smaller scale features. Compared with the flux differencing entropy stable schemes in \cite{chan2022entropyprojection}, DG with \gnote{entropy correction artificial viscosity} appears to be slightly more dissipative. However, as argued in \cite{gassner2022stability}, because of the lack of \gnote{local} linear stability for flux differencing entropy stable schemes, it is not clear whether small scale structures present in a flux differencing entropy stable DG solution are physical or due to the introduction and evolution of non-physical numerical artifacts.

%
%

\section{Conclusion} 

In this paper, we introduce an \gnote{entropy correction artificial viscosity} method for recovering entropy stability for standard weak form DG discretizations. The artificial viscosity coefficients incorporate a volume entropy residual which utilizes the entropy projection. We prove that this estimate is super-convergent for sufficiently regular solutions and prove that, if interface fluxes are computed using the entropy projection, a standard DG method with \gnote{entropy correction artificial viscosity} satisfies the same global entropy inequality satisfied by flux differencing entropy stable DG methods. 

Numerical experiments suggest that DG methods with \gnote{entropy correction artificial viscosity} are slightly more dissipative than flux differencing entropy stable DG schemes. However, DG methods with \gnote{entropy correction artificial viscosity} enjoy many of the same properties of flux differencing DG methods (such as robustness and high order accuracy), while also reducing spurious oscillations and improving \gnote{local} linear stability. 

The Julia codes used to generate the results in this paper are available at \url{https://github.com/jlchan/paper-artificial-viscosity-entropy-stable-2025}. 

\section{Acknowledgements}

The authors gratefully acknowledge support from National Science Foundation under awards DMS-1943186 and DMS-223148. The authors also thank Alex Cicchino, Siva Nadarajah, Brian Christner, Raymond Park, Philipp \"{O}ffner, Lucas Wilcox, and Ayaboe Edoh for helpful discussions. The authors also acknowledge the Atum.jl library (\url{https://github.com/mwarusz/Atum.jl}), whose implementation of the matrix dissipation flux was used in this work \cite{winters2017uniquely, waruszewski2022entropy}. \gnote{Finally, the author thanks the three anonymous reviewers, whose comments significantly improved the quality and presentation of this manuscript.}

\appendix
\section{Additional 1D numerical experiments}
\label{sec:additional_1d}

In this section, we include some additional one-dimensional experiments. 

\subsection{Modified Sod with HLLC and Roe-type matrix dissipation interface fluxes}
\label{sec:hllc}

In addition to the local Lax-Friedrichs flux, we also experimented with HLLC interface fluxes \cite{batten1997choice} and interface fluxes based on Roe-type matrix dissipation \cite{winters2017uniquely, waruszewski2022entropy}, both of which are contact-preserving. These contact preserving fluxes were also shown to improve the order of convergence for entropy stable nodal DG methods \cite{hindenlang2020order} and avoid spurious oscillations for flux differencing nodal DG methods applied to a variant of the original Sod shock tube \cite{waruszewski2022entropy}. 

\begin{figure}
\centering
\subfloat[Flux differencing nodal DG]{\includegraphics[width=.45\textwidth, trim={9em 8em 9em 5em}, clip]{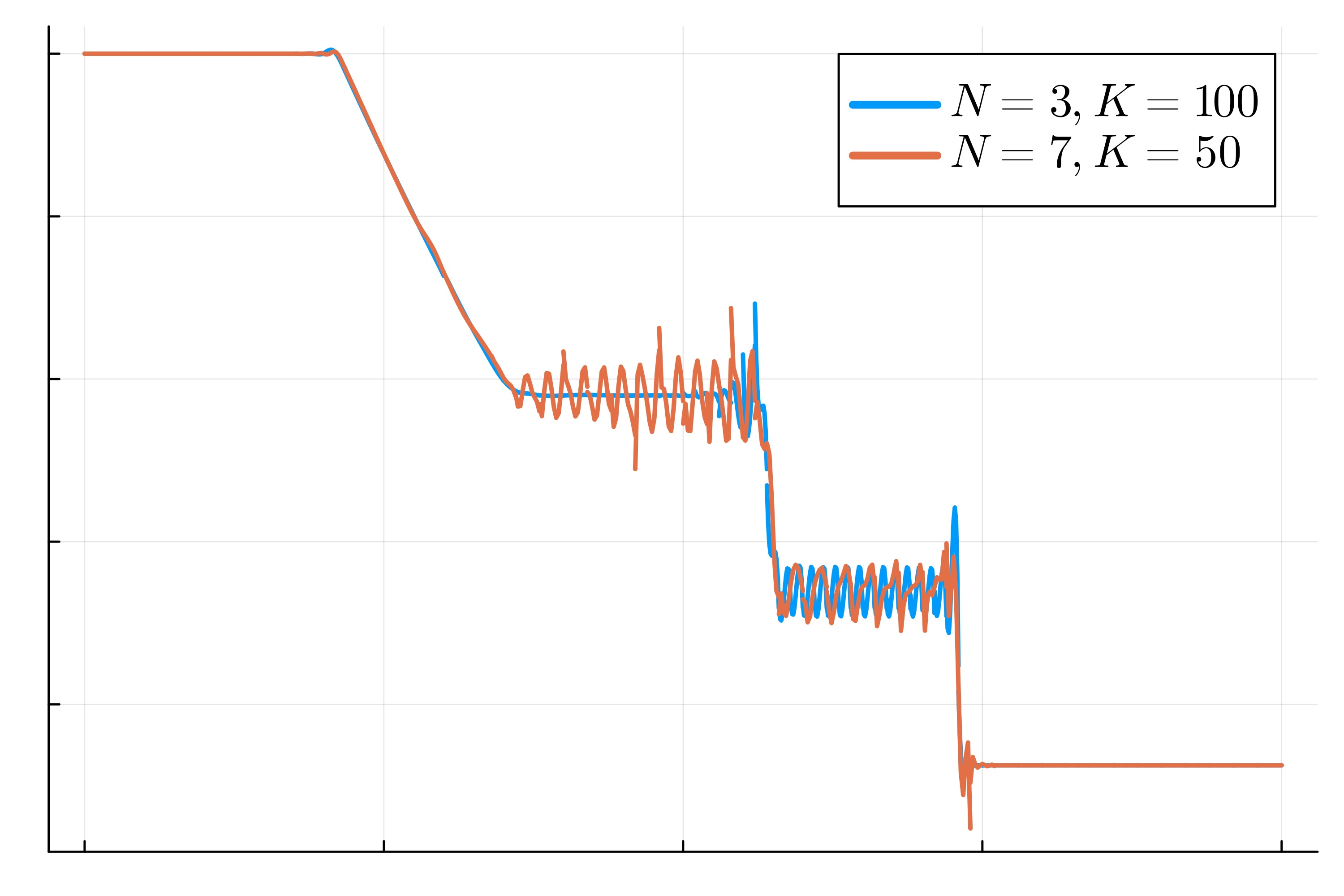}}
\hspace{.1em}
\subfloat[Nodal DG with AV]{\includegraphics[width=.45\textwidth, trim={9em 8em 9em 5em}, clip]{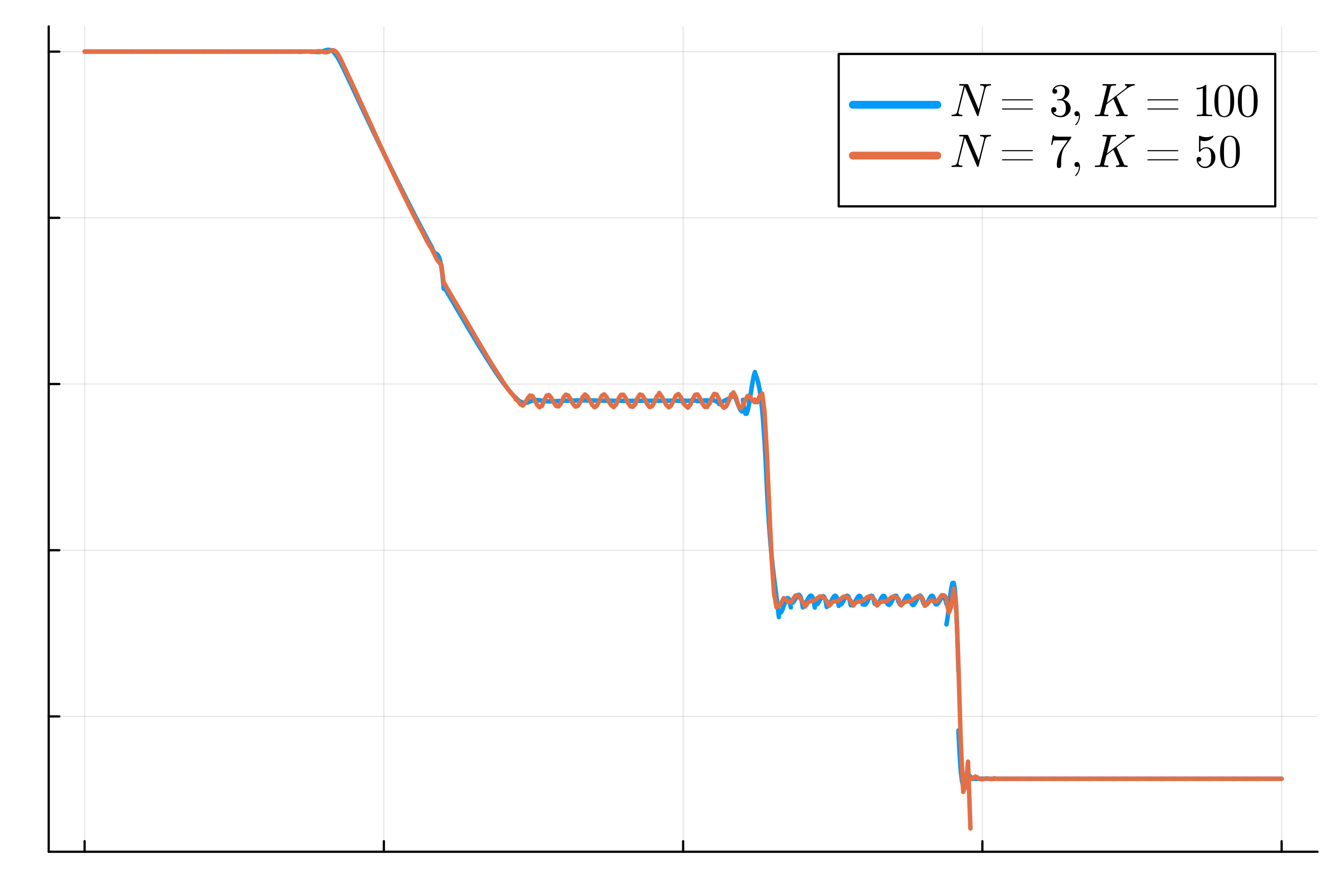}}\\
\subfloat[Flux differencing modal DG]{\includegraphics[width=.45\textwidth, trim={9em 8em 9em 5em}, clip]{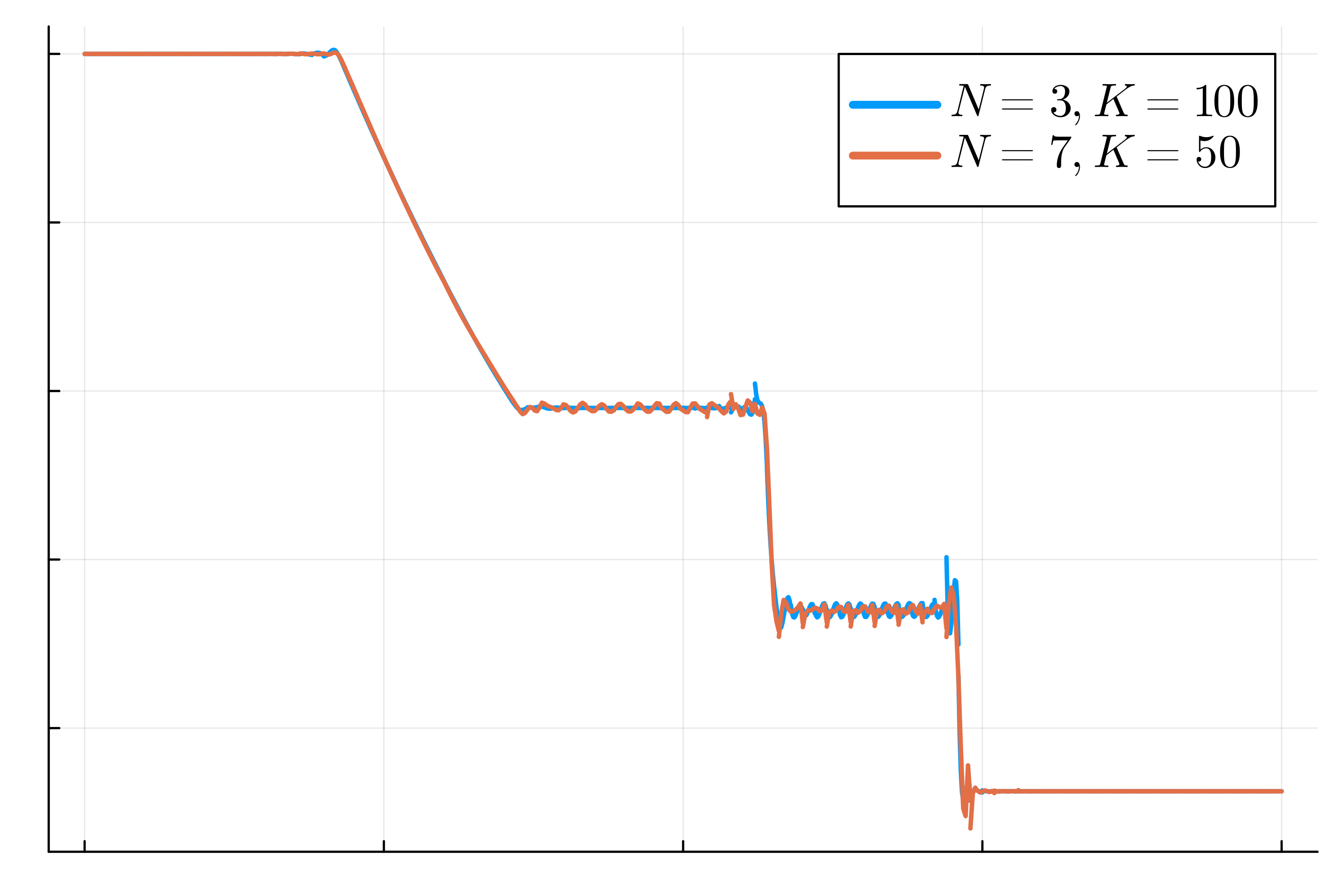}}
\hspace{.1em}
\subfloat[Modal DG with AV]{\includegraphics[width=.45\textwidth, trim={9em 8em 9em 5em}, clip]{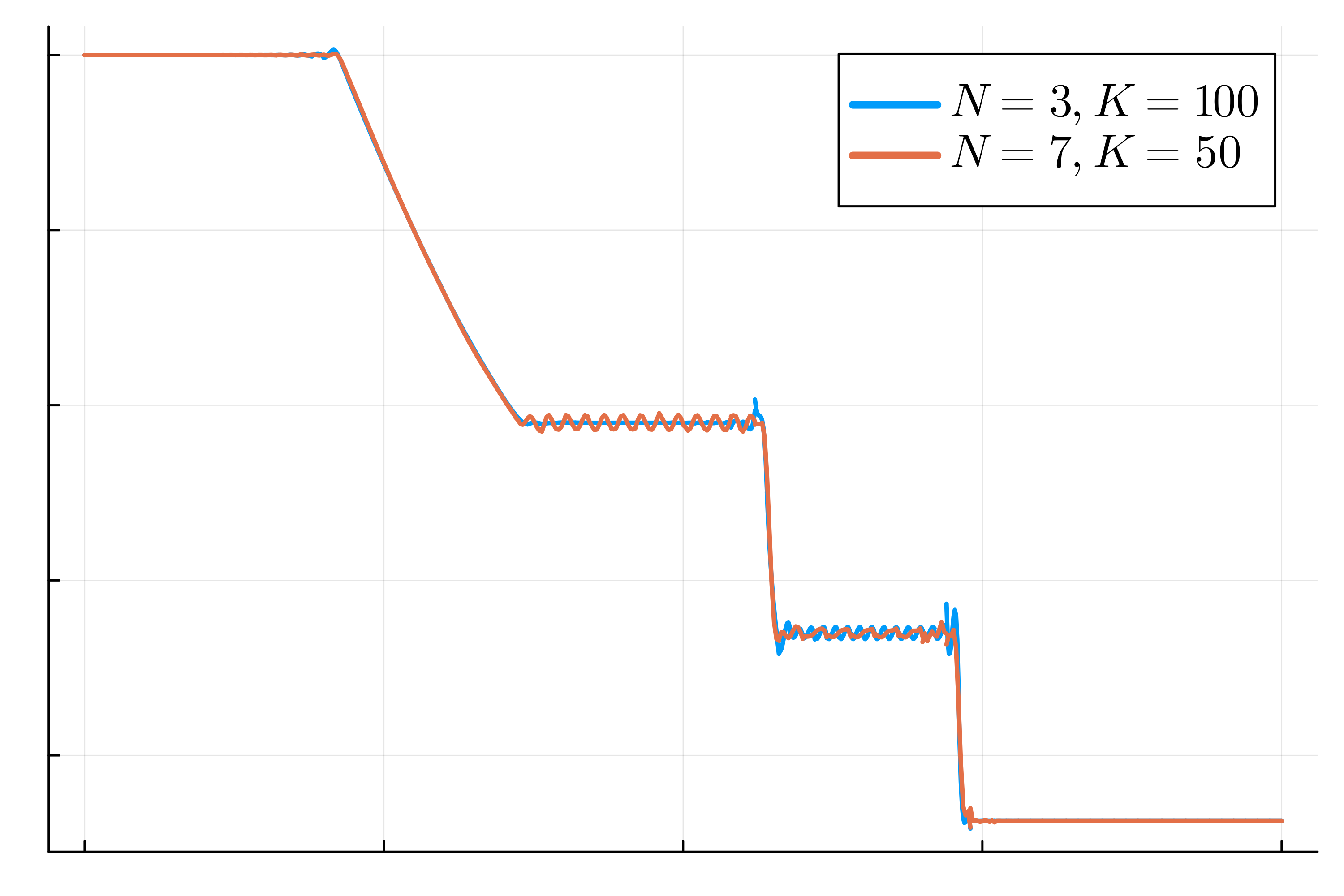}}\\
\caption{Solutions  \rnote{(density)} of modified Sod with the HLLC interface flux.} 
\label{fig:modSod_better_flux}
\end{figure}

Figure~\ref{fig:modSod_better_flux} shows numerical results for the modified Sod problem using the HLLC flux (the matrix dissipation flux of \cite{winters2017uniquely, waruszewski2022entropy} produced very similar results). For degree $N=3$, we observe that the use of HLLC and matrix dissipation fluxes reduce oscillations between the rarefaction wave and the contact discontinuity, but that they do not reduce oscillations between the contact discontinuity and shock. However, for $N=7$, we do not observe a similar reduction in oscillations between the contact discontinuity and shock for either contact-preserving flux. 


\subsection{Modified Sod with smaller post-shock density and pressure}

\begin{figure}
\centering
\subfloat[Nodal DG with AV]{\includegraphics[width=.47\textwidth, trim={9em 8em 9em 5em}, clip]{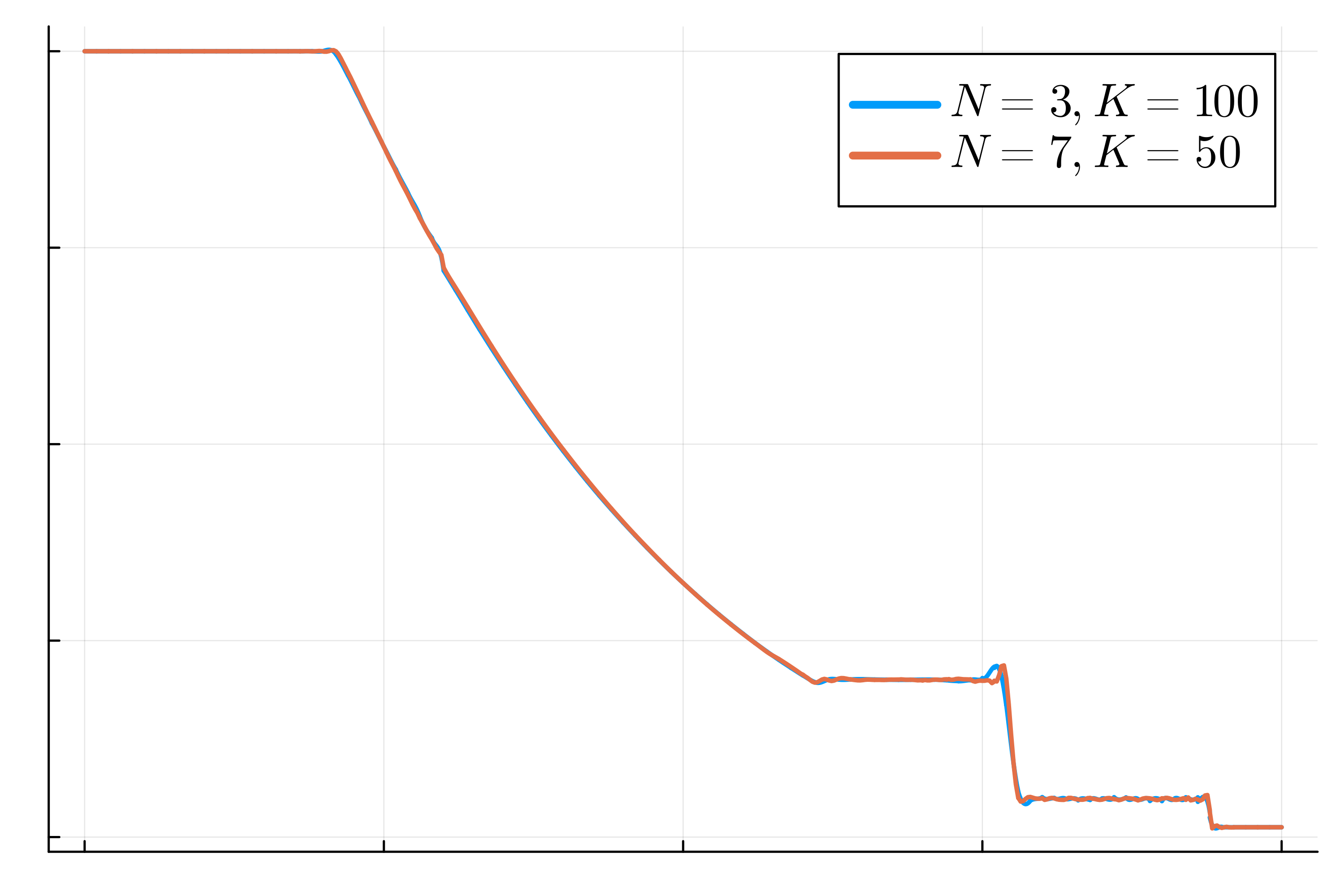}}
\hspace{.25em}
\subfloat[Modal DG with AV]{\includegraphics[width=.47\textwidth, trim={9em 8em 9em 5em}, clip]{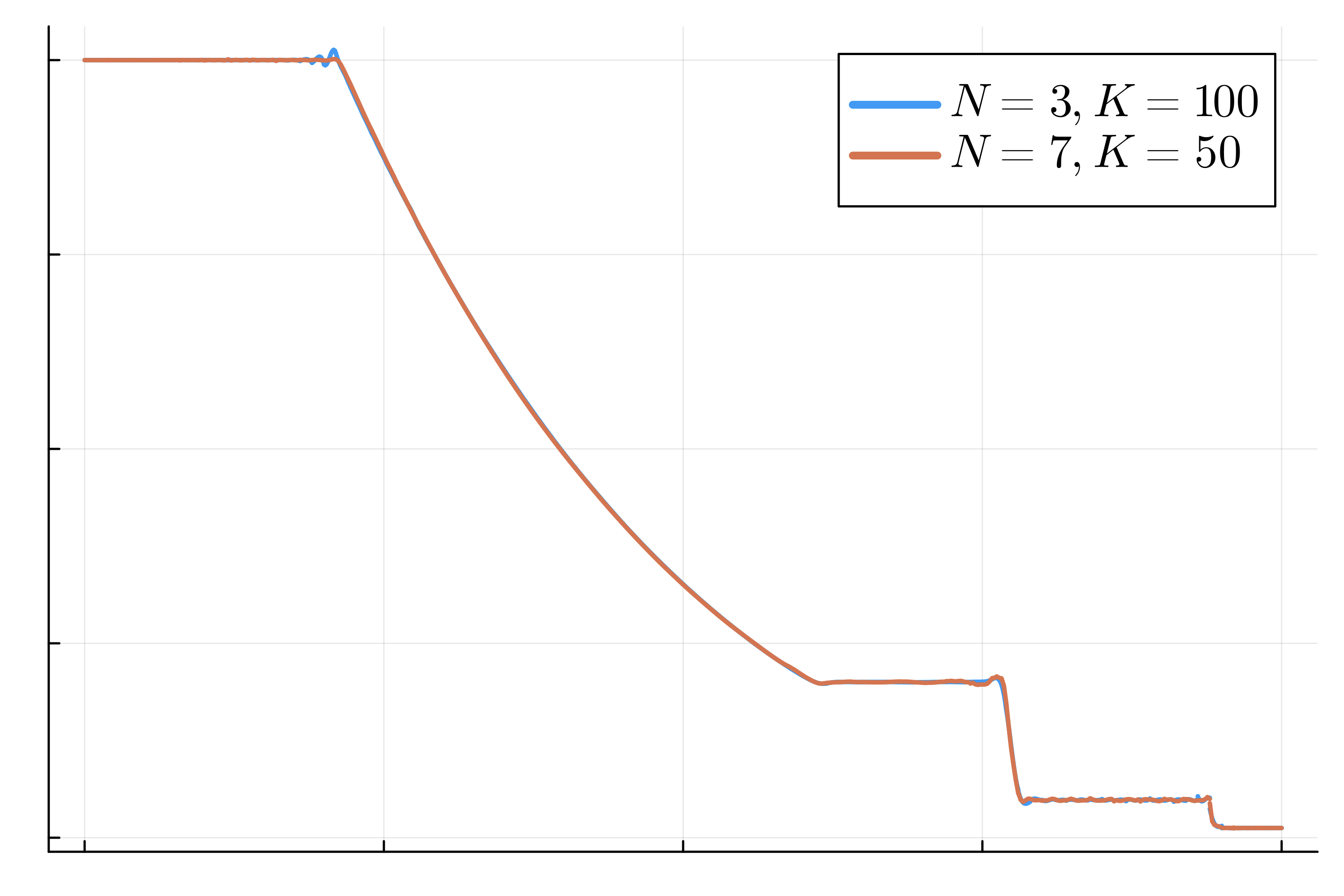}}
\caption{Solutions \rnote{(density)} of modified Sod with smaller post-shock density and pressure using standard DG with \gnote{entropy correction artificial viscosity}.} 
\label{fig:modSod_near_vacuum}
\end{figure}

Next, we examined a version of modified Sod with smaller post-shock density and pressure, where we decreased the post-shock density and pressure by an additional factor of 10, such that $(\rho, u, p) = (.0125, 0, .01)$ for $x > 3$ in \eqref{eq:modSod}. For this initial condition, both nodal and modal flux differencing entropy stable DG methods crash before the final time of $T = 0.2$. Figure~\ref{fig:modSod_near_vacuum} shows the resulting solution profile for nodal and modal DG with \gnote{entropy correction artificial viscosity}. We note that the solution remains positive, despite the fact that for the modal DG discretization, the entropy projection (which is sensitive when density and pressure are small) is used to compute interface fluxes \cite{chan2022entropyprojection}. 

Finally, we note that for this version of modified Sod, the artificial viscosity near the shock was roughly $60\times$ larger than for the standard modified Sod problem \eqref{eq:modSod}. However, the number of time-steps taken by the adaptive time-stepper for this smaller density and pressure case only increased by a factor of $\approx 1.267$ compared with the standard modified Sod initial condition \eqref{eq:modSod}, indicating that the magnitude of the artificial viscosity was not consistently large enough to induce a parabolic  $O(h^2)$ maximum stable explicit time-step restriction. 

\subsection{Enforcing additional entropy inequalities}

To demonstrate the impact of using the entropy projection $\tilde{\bm{u}}$ in the volume entropy residual \eqref{eq:entropy_ineq_error}, we compare a DG method where an additional entropy inequality is enforced. The second entropy inequality we enforce is the same as \eqref{eq:entropy_ineq_error} except that the entropy potential $\psi_m$ is evaluated at the DG solution $\bm{u}_h$ rather than the entropy projection $\tilde{\bm{u}}$:
\begin{equation}
\sum_{m=1}^d \LRs{ \LRp{-\bm{f}_m(\bm{u}_h), \pd{\Pi_N\bm{v}(\bm{u}_h)}{x_m}}_{D^k} + \LRa{\psi_m(\bm{u}_h)n_m, 1}_{\partial D^k}}.
\label{eq:entropy_ineq_error_2}
\end{equation}
We note that we cannot use this modified entropy inequality on its own; enforcing only \eqref{eq:entropy_ineq_error_2} results in unstable simulations where the adaptive time-step size converges to zero. However, we can enforce both entropy inequalities \eqref{eq:entropy_ineq_error} and \eqref{eq:entropy_ineq_error_2} by computing two artificial viscosity coefficients using \eqref{eq:eps} (one for each entropy inequality) and taking the maximum \cite{guermond2011entropy}. 

\begin{figure}
\centering
\includegraphics[width=.47\textwidth, trim={9em 8em 9em 5em}, clip]{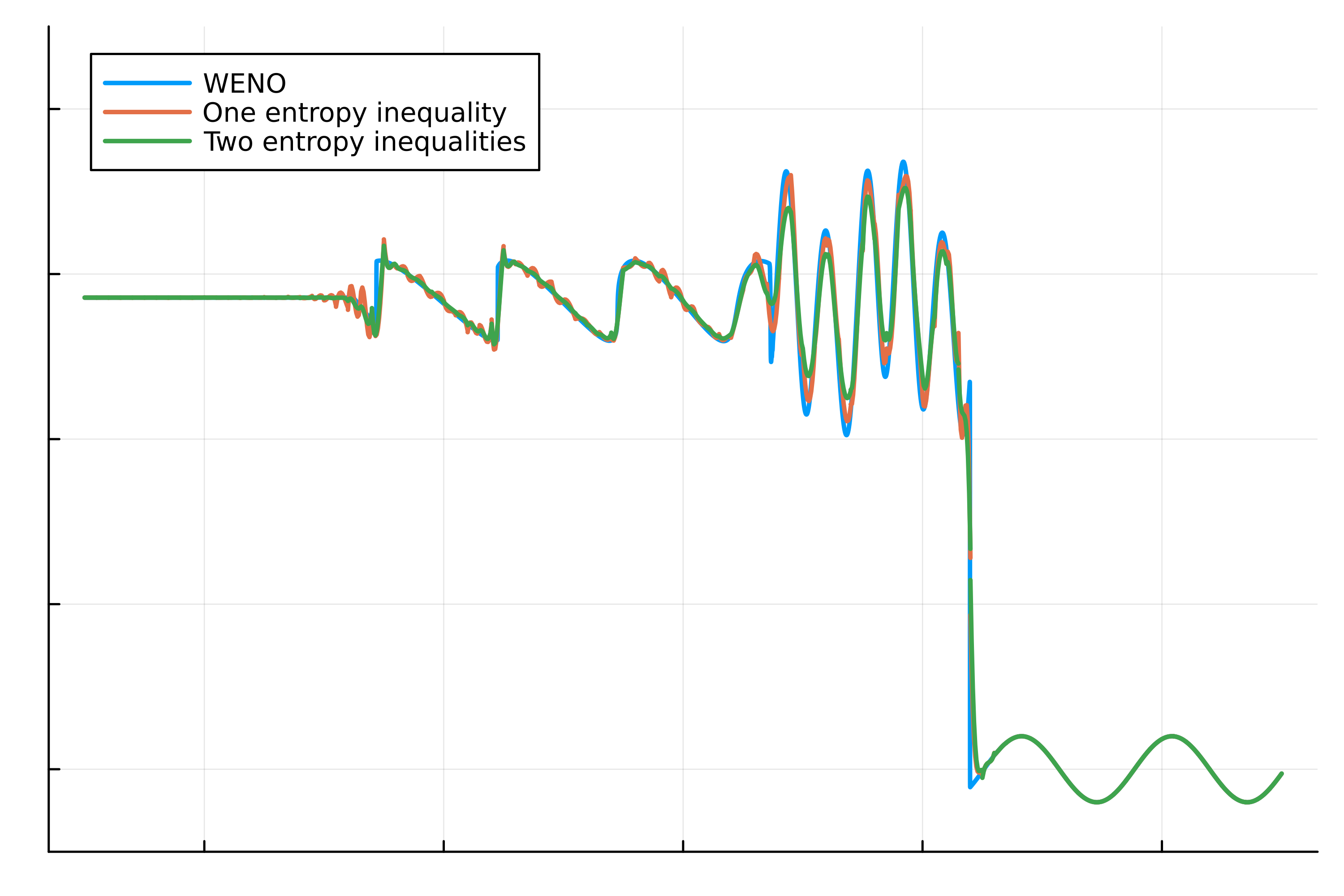}
\hspace{.1em}
\includegraphics[width=.47\textwidth, trim={9em 8em 9em 5em}, clip]{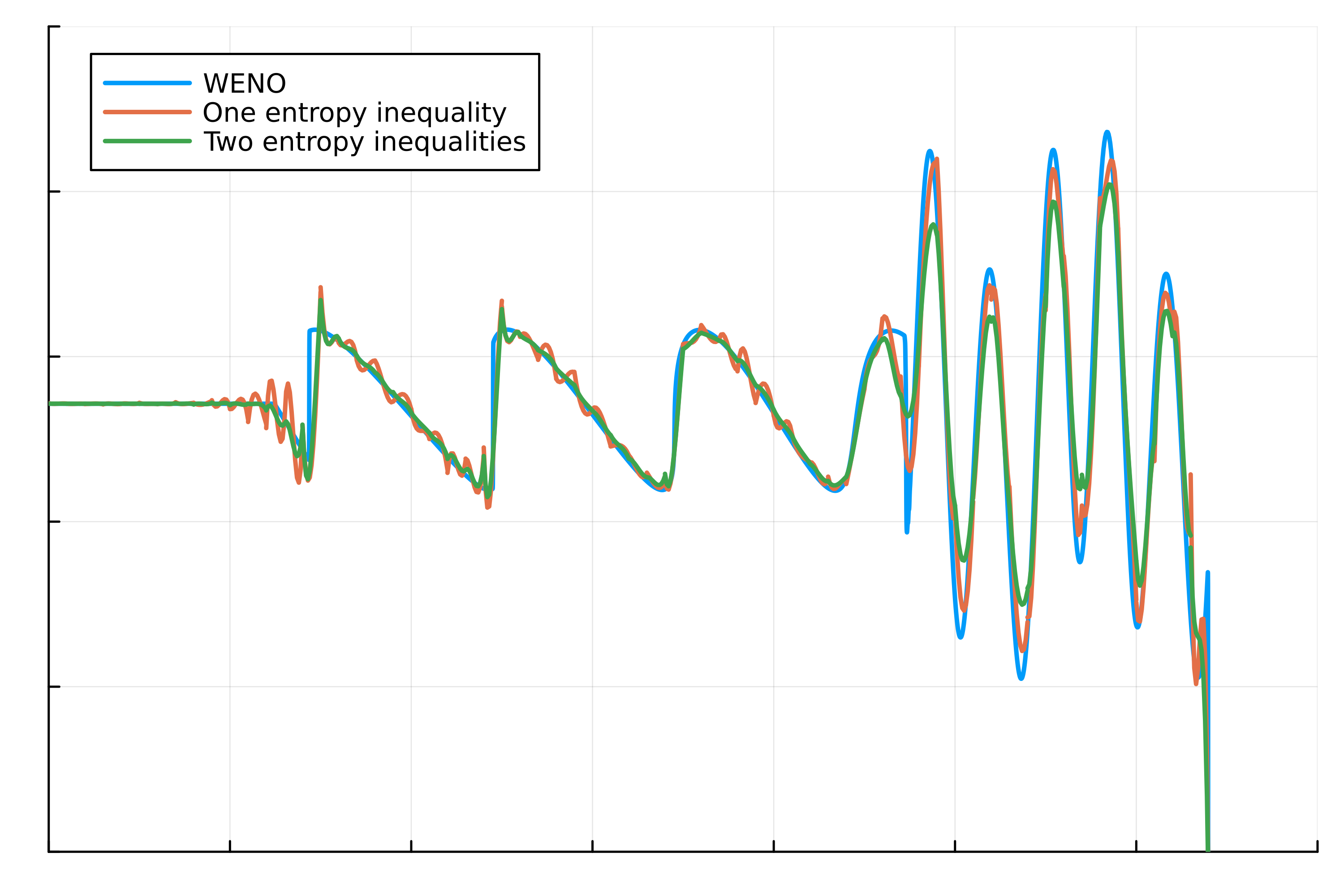}
\caption{Comparison of DG solutions (\rnote{density}, degree $N=3$, $100$ elements) to the Shu-Osher problem when enforcing one and two entropy inequalities.} 
\label{fig:two_entropy_inequalities}
\end{figure}

Figure~\ref{fig:two_entropy_inequalities} shows the results of a degree $N=3$ modal DG formulation under one and both of these entropy inequalities. We observe that enforcing the second entropy inequality without the entropy projection $\tilde{\bm{u}}$ results in a more diffusive solution. This is consistent with the fact that the volume entropy residual with the entropy projection \eqref{eq:entropy_ineq_error} converges to zero at a faster rate than the volume entropy residual without entropy projection \eqref{eq:entropy_ineq_error_2} based on estimates in \cite{gaburro2023high}. Enforcing this additional entropy inequality also appears to provide some shock capturing effects, as oscillations around shocks are smoothed out. \vnote{We note that a similar approach could be used to combine entropy correction artificial viscosity with shock capturing artificial viscosity; one could take the maximum of the entropy correction artificial viscosity coefficient and the shock capturing artificial viscosity coefficient.}


\section{A subcell version of the viscosity coefficient}
\label{sec:subcell}

If we wish to allow the local \gnote{entropy correction artificial viscosity} coefficient $\epsilon_k(\bm{u}_h)$ to vary spatially within an element, we can determine an optimal choice for $\epsilon_k(\bm{u}_h)$ by minimizing the $L^2$ norm of $\epsilon_k(\bm{u}_h)$ subject to the entropy identity \eqref{eq:cell_entropy_ineq_condition} and a non-negativity constraint. This optimization problem turns out to admit an analytical solution. 
\begin{lemma}
Let $\delta_k(\bm{u}_h)$ be the volume entropy residual \eqref{eq:entropy_ineq_error}. Consider the following inequality constrained optimization problem:
\begin{gather*}
\min_{\epsilon_k} \nor{\epsilon_k}^2_{D^k}\\
\sum_{i,j=1}^d \LRp{\epsilon_k \bm{K}_{ij}\bm{\Theta}_j, \bm{\Theta}_i}_{D^k} \geq -\min(0, \delta_k(\bm{u}_h)),\\
\epsilon_k(\bm{x}) \geq 0, \quad \forall \bm{x} \in D^k.
\end{gather*}
This optimization problem has an explicit solution
\begin{equation}
\epsilon_k(\bm{u}_h) = -\min(0, \delta_k(\bm{u}_h))\frac{a}{\nor{a}^2_{D^k}}, \qquad a(\bm{x}) = \sum_{i,j=1}^d \bm{\Theta}_i^T\bm{K}_{ij}\bm{\Theta}_j \geq 0.
\label{eq:explicit_sol}
\end{equation}
\label{lemma:opt}
\end{lemma}
\begin{proof}
The optimization problem can be rewritten in a more abstract form:
\begin{gather}
\min_\epsilon \nor{\epsilon}^2 \nonumber\\
\LRp{a, \epsilon} \geq b \label{eq:opt_abstract}\\
\epsilon \geq 0, \quad \forall \bm{x} \in D^k. \nonumber
\end{gather}
where $a(\bm{x}), b \geq 0$ and we have dropped the $k, D^k$ subscripts for simplicity of notation. Note that $a(\bm{x}) = \sum_{i,j=1}^d \bm{\Theta}_i^T \bm{K}_{ij}\bm{\Theta}_j$ and $b = -\min(0, \delta_k(\bm{u}_h))$ recovers the optimization problem in Lemma~\ref{lemma:opt}.

Observe that if $\epsilon'$ satisfies $b \leq \LRp{a, \epsilon'}$, then $b \leq \LRp{a, \epsilon'} \leq \nor{a}\nor{\epsilon'}$ and $b / \nor{a} \leq \nor{\epsilon'}$. Thus, if we can find an $\epsilon \geq 0$ that meets the lower bound such that $\nor{\epsilon} = b / \nor{a}$, the solution is both feasible and optimal. One can verify that $\epsilon = b\frac{a}{\nor{a}^2} \label{eq:opt_analytic}$ satisfies both conditions.
\end{proof}
\begin{remark}
We note that the proof of Lemma~\ref{lemma:opt} implies that \eqref{eq:opt_abstract} is equivalent to an equality constrained minimum norm problem $\min_\epsilon \nor{\epsilon}^2$ such that $\LRp{a, \epsilon} = b$. 
\end{remark}

Note that, if $a$ is constant over an element, \eqref{eq:explicit_sol} reduces to the piecewise constant \gnote{entropy correction artificial viscosity} coefficient in Section~\ref{sec:piecewise_const}. To avoid division by small numbers, we again compute the ratio in \eqref{eq:opt_abstract} using the regularized ratio described in Remark~\ref{remark:ratio}.

\begin{remark}
We observe in numerical experiments that, even when using the subcell \gnote{entropy correction artificial viscosity} $\epsilon_k(\bm{u}_h)$, evaluating the viscous matrices $\bm{K}_{ij}$ at the average solution state as described in Remark~\ref{remark:K_avg} does not appear to degrade accuracy, results in a larger maximum stable time-step, and produces smaller spurious oscillations around shocks and under-resolved solution features. Thus, we assume that $\bm{K}_{ij}$ is evaluated at averaged solution states $\bar{\bm{u}}_h$ for both the case when $\epsilon_k(\bm{u}_h)$ is piecewise constant and has subcell variations.
\end{remark}

\begin{figure}
\centering
\subfloat[$L^2$ error (density wave)]{\includegraphics[width=.32\textwidth]{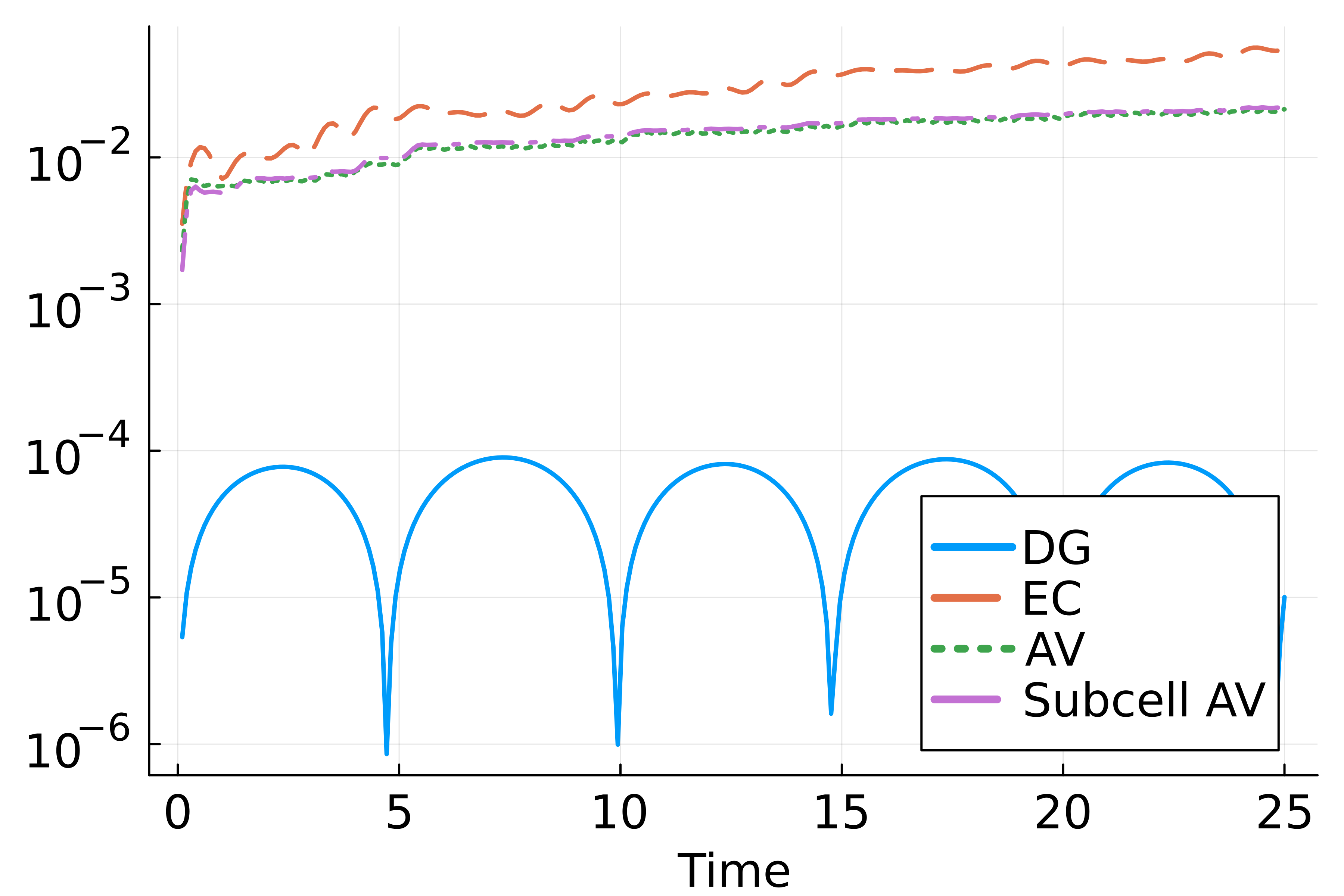}}
\hspace{.1em}
\subfloat[Entropy (density wave)]{\includegraphics[width=.32\textwidth]{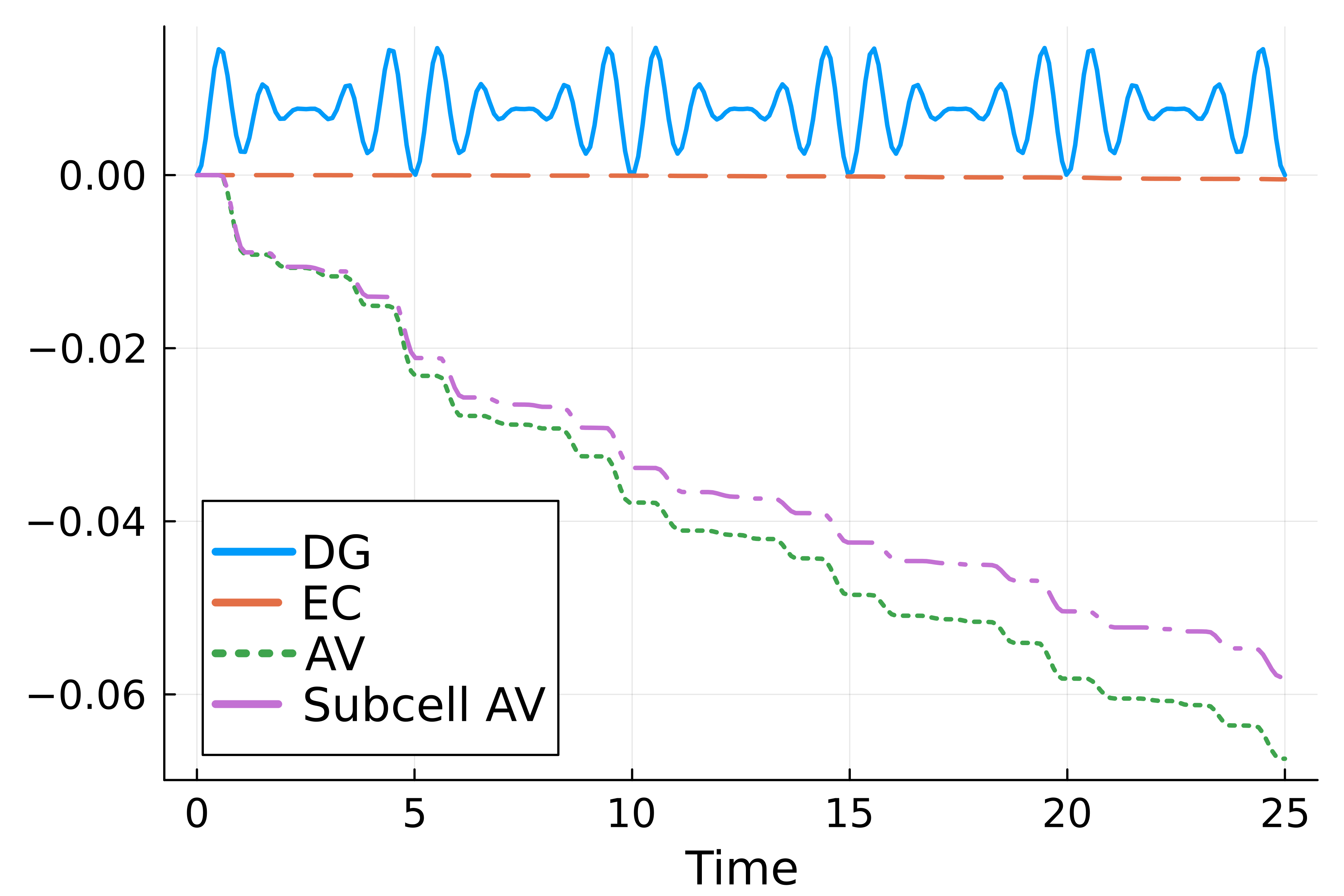}}
\hspace{.1em}
\subfloat[Zoom of density (Shu-Osher)]{\includegraphics[width=.32\textwidth]{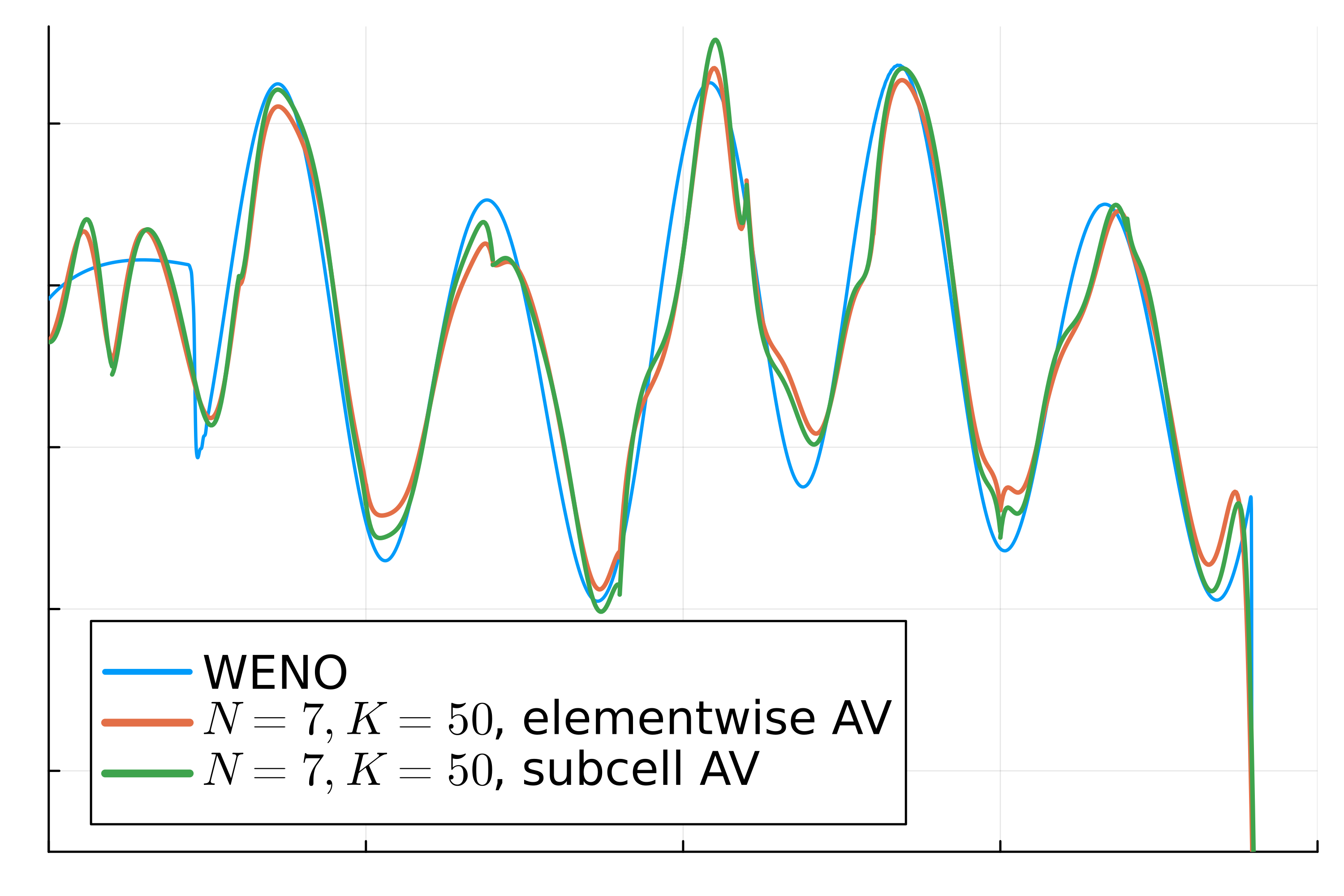}}
\caption{Comparison of element-wise constant and subcell \gnote{entropy correction artificial viscosity} for the 1D density wave and Shu-Osher problems using a degree $N=7$ DG approximation.}
\label{fig:elementwise_vs_subcell}
\end{figure}

Despite the norm-minimizing nature of the subcell \gnote{entropy correction artificial viscosity}, there does not appear to be a significant difference between the piecewise constant and subcell \gnote{entropy correction artificial viscosity} in practice. Figure~\ref{fig:elementwise_vs_subcell} illustrates these differences for the 1D density wave \eqref{eq:density_wave_1d} with amplitude $A=.98$ and the Shu-Osher sine-shock interaction problem. Both problems use degree $N=7$; the density wave uses a mesh of $4$ elements, while the Shu-Osher problem uses a mesh of $50$ elements. The subcell \gnote{entropy correction artificial viscosity} results in slightly less dissipative results for the density wave and the Shu-Osher problem. This effect becomes less pronounced for smaller polynomial degrees.

\bibliographystyle{plain}
\bibliography{reference.bib}

\end{document}